\documentclass[11pt,amsfonts]{amsart}
\usepackage{tgpagella}
\usepackage[english]{babel}
\usepackage{inputenc}
\usepackage{mathabx}
\usepackage{amsmath,amsthm,amsfonts,amssymb,amscd}
\usepackage{mathpazo}
\usepackage{amssymb,amsfonts}
\usepackage[all,arc]{xy}
\usepackage{enumerate}
\usepackage{mathrsfs}
\usepackage{tgtermes}
\usepackage{marginnote}
\usepackage{mathtools}
\usepackage{tikz,amsmath, amssymb,bm,color}
\usetikzlibrary{shapes,arrows}
\usetikzlibrary{calc}
\usetikzlibrary{snakes}
\usepackage{amsthm,thmtools,xcolor}
\DeclarePairedDelimiter\ceil{\lceil}{\rceil}

\usepackage[all]{xy}
\usepackage{tikz-cd}
\usetikzlibrary{cd}
\theoremstyle{plain}
\newtheorem{thm}{Theorem}[section]
\newtheorem{cor}{Corollary}[section]
\newtheorem{prop}{Proposition}
\newtheorem{lem}{Lemma}[section]

\newtheorem{thmx}{Theorem}

\newtheorem{lemx}{Lemma}

\theoremstyle{definition}
\newtheorem{defn}{Definition}[section]

\theoremstyle{remark}
\newtheorem{rem}{\textit{Remark}}[section]

\numberwithin{equation}{section}
\usepackage{color}

\usepackage[colorlinks = true,
linkcolor = blue,
urlcolor  = blue,
citecolor = blue,
anchorcolor = blue]{hyperref}
\usepackage{hyperref}
\makeatletter
\let\c@equation\c@thm
\makeatother
\numberwithin{equation}{section}

\DeclareMathOperator{\dist}{\mathrm{dist}}

\DeclareMathOperator{\supp}{\mathrm{supp}}
\newcommand\underrel[3][]{\mathrel{\mathop{#3}\limits_{%
			\ifx c#1\relax\mathclap{#2}\else#2\fi}}}
\title[Zakharov-Kuznetsov  Equation]{On the propagation of regularity for solutions of the Zakharov-Kuznetsov  equation}

\author{Argenis. J.  Mendez}

\address{Centro de  Modelamiento Matemático, Universidad de Chile, Santiago de Chile.}

\email{amendez@dim.uchile.cl}

\thanks{}


\date{August, 2020.}
\commby{Argenis Mendez}
\begin{document}

\begin{abstract}
In this work, we study some special properties of smoothness concerning to the initial value problem associated with the Zakharov-Kuznetsov-(ZK) equation in the $n-$ dimensional setting, $n\geq 2.$

It is known that the solutions of the ZK equation in the $2d$ and $3d$ cases verify special regularity properties. More precisely, the regularity of the initial data on a family of half-spaces propagates with infinite speed. Our objective in this work is to extend this analysis to the case in that the regularity of the initial data is measured on a fractional scale.
  To describe this phenomenon we present  new  localization formulas  that allow us to  portray  the regularity of the solution on a certain class of subsets of the euclidean space.
\end{abstract}

\maketitle

%\tableofcontents

\section{Introduction}
In this work we  are  interested in to describe some regularity  properties of    solutions  to the initial value problem (IVP) associated to  the  \emph{Zakharov-Kuznetsov } (ZK) equation
\begin{equation}\label{zk4}
\left\{
\begin{array}{ll}
\partial_{t}u+\partial_{x_{1}}\Delta u+u\partial_{x_{1}}u=0, &\,t\in\mathbb{R} \\
u(x,0)=u_{0}(x),&x=(x_{1},x_{2},\dots,x_{n})\in \mathbb{R}^{n} ,  n\geq 2, \\
\end{array} 
\right.
\end{equation}
where $\Delta=\partial_{x_{1}}^{2}+\partial_{x_{2}}^{2}+\cdots+\partial_{x_{n}}^{2}$ is the \emph{ n-dimensional  Laplacian}.  

This equation  was deduced  by Zakharov and Kuznetsov  \cite{ZK} to describe the ionic-acoustic  waves uniformly magnetized plasma  in the two dimensional and three dimensional cases. More precisely, this equation was derived  as a long  wave small amplitude limit  of the Euler-Poisson system in the ''cold-plasma'' approximation. Later on, this  long wave limit was  rigorously  described  by Lannes, Linares and Saut \cite{LLS}.  Also the ZK equation  has been derived from the  Vlasov-Poisson system  in a combined  cold ions and long wave limit by Kwan \cite{HK}. 

The ZK equation  has a Hamiltonian structure  and  it has at least   formally  three conserved quantities,  namely
\begin{equation*}
\mathcal{I}_{1}[u](t)=\int_{\mathbb{R}^{n}} u(x,t)\,\mathrm{d}x=\mathcal{I}_{1}[u](0),\,\, \mathcal{I}_{2}[u](t)=\int_{\mathbb{R}^{n}} (u(x,t))^{2}\,\mathrm{d}x= \mathcal{I}_{2}[u](0)
\end{equation*}
and
\begin{equation*}
\mathcal{I}_{3}[u](t)=\frac{1}{2}\int_{\mathbb{R}^{n}}\left|\nabla u(x,t)\right|^{2}\,\mathrm{d}x-\frac{1}{3}\int_{\mathbb{R}^{n}}(u(x,t))^{3}\,\mathrm{d}x =\mathcal{I}_{3}[u](0).
\end{equation*}
%%%%%%Global and local well posedness issues%%%%
Due to its  physical relevance the ZK equation  have called the attention  in the recent years. Nevertheless, before describe the main goal in this work we require to describe the space solution  where the property to be described has  sense  from the mathematical point of view.  In this  direction,  we   give a brief description of the  Initial Value Problem (IVP) issues  associated to  \eqref{zk4}.

Since   
the IVP  for the ZK equation  have been  broadly studied in the recent years, the ZK literature have been increasing more and more.  So that,  trying  to describe  the major part of the results associated to \eqref{zk4} is a difficult task. Thus, we    present  a short review that summarize the IVP   issues   according to the physical  dimension. In the particular case that   the physical  dimension is  $n=2,$
Faminskii \cite{Fami} proved global well-posedness in $H^{j}(\mathbb{R}^{2}),\,j\in \mathbb{Z}^{+},\,j\geq 1,$ later 
 Linares and Pastor \cite{LPAS1}  proved global well-posedness in $H^{s}(\mathbb{R}^{2}), s>3/4,$  additionally and  independently   and simultaneously, Gr\"{u}nrock and Herr \cite{grum3} and Molinet and Pilod \cite{MP} proved local well posedness in $H^{\frac{1}{2}+}(\mathbb{R}^{2}).$  Additionally,     it was proved  recently Kinoshita \cite{Kinoshita} local well posedness in $H^{-\frac{1}{4}+}(\mathbb{R}^{2})$   that   according to the scaling argument  it is optimal up to the end-point.

Concerning  to the case in which the  dimension  is  $n=3,$    Linares and Saut \cite{LS} proved local well-posedness in $H^{s}(\mathbb{R}^{3}),s>9/8.$ Later, \cite{LS} Ribaud and Vento \cite{RV1}  proved local well-posedness in  $H^{s}(\mathbb{R}^{2}),s>1,$ 
Molinet and Pilod \cite{MP}   prove global that the results in  \cite{RV1}  can be extended  globally in time. More recently,   Herr and Kinoshita \cite{HERRK}  have  proved   that for dimension $n\geq 3,$  the IVP   associated to \eqref{zk4} is locally well-posedness on $H^{s}(\mathbb{R}^{n}),s>\frac{n}{2}-2.$ Additionally, Herr and Kinoshita \cite{HERRK} also establish that in dimension $n=3$  and for real solutions,  the IVP  associated to \eqref{zk4} is globally well posed in $L^{2}(\mathbb{R}^{3})$  and in dimension $n=4$ 
it is globally well-posed for real-valued initial data in $H^{1}
(\mathbb{R}^{4})$ 
with sufficiently small $L^{2}(\mathbb{R}^{4})$-norm.

Although  of the improvements concerning to the local and global theory,  the  regularity properties associated to the solutions of \eqref{zk4} we intend to describe in this work      depends strongly on  having a priori- estimates for  $\|\nabla u\|_{L^{1}_{T}L^{\infty}}<\infty.$ So that,    before  we  firstly  describe the  space where the properties  to be described have sense. We recall a result  that  becomes a direct  consequence  of combining  energy estimates, the commutator estimates in \cite{KATOP2}, the Sobolev embedding and  the arguments in \cite{BS}.
\begin{thm}\label{t2}
	Given $u_{0}\in H^{s}(\mathbb{R}^{n})$ with $s>\frac{n}{2}+1,$ there exist  $T=T\left(\|u_{0}\|_{H^{s}}\right)>0,$ and a unique solution  $u=u(x,t)$ of the IVP \ref{zk4} such that 
	\begin{equation}
	u\in C\left([0,T]:H^{s}(\mathbb{R}^{n})\right).
	\end{equation}
	Moreover, the map data-solution $u_{0}\longmapsto u(x,t)$ from $H^{s}(\mathbb{R}^{n})$ into $C\left([0,T]:\right.$ $\left. H^{s}(\mathbb{R}^{n})\right)$  is locally  continuous.
\end{thm}
The Theorem \ref{t2} is the  basis space  to describe the  properties we intend to describe in this work since it allow us to guarantee that  $\nabla u \in C\left([0,T]: H^{s-1}(\mathbb{R}^{n})\right)\subset L^{1}\left([0,T]: L^{\infty}(\mathbb{R}^{n})\right).$

Since the space  solution has been set   we proceed to   establish the main goal  of this  work, that  is mainly based in to extend the study of propagation of regularity  found by Linares and Ponce \cite{LPZK} in solutions of the ZK equation in the $2-$dimensional case as well as in the $3d-$case to a more general context where the regularity to be considered be  fractional. Roughly speaking, the propagation of regularity phenomena  describe the behavior of the regularity  of the solution  when the initial data enjoy of some  extra smoothness on a  particular class of  subsets of the physical space. Specifically, this class of sets  are   strips and half-spaces, that in our work  will be   indicated according to the following notation: For  
   $\sigma$     a non-null vector in $ \mathbb{R}^{n}$ and  $\alpha\in \mathbb{R}$  the \emph{half-space} $\mathcal{H}_{\{\sigma,\alpha\}}$  will be indicated by   
\begin{equation*}
\mathcal{H}_{\{\sigma,\alpha\}}:=\left\{x\in \mathbb{R}^{n}\,|\, \sigma\cdot x>\alpha \right\},
\end{equation*} 
where $\cdot$ denotes the canonical inner product in $\mathbb{R}^{n}.$
Additionally,  for $\gamma, \beta\in\mathbb{R}$ with  $\gamma<\beta$ we define the  \emph{strip}
\begin{equation*}
\mathcal{Q}_{\{\sigma,\alpha,\beta\}}:=\left\{x\in \mathbb{R}^{n}\,|\, \gamma< \sigma\cdot x< \beta\right\}.
\end{equation*}
Formally, the description of the propagation of regularity phenomena  in solutions of the ZK equation is summarized in the following theorem.
\begin{thm}[\cite{LPZK}]\label{t1}
	Let $u_{0}\in H^{\frac{5}{2}+}(\mathbb{R}^{3}).$  If for some $\sigma=\left(\sigma_{1},\sigma_{2},\sigma_{3}\right)\in \mathbb{R}^{3}$ with 
	\begin{equation*}
	\sigma_{1}>0,\quad \sigma_{2},\sigma_{3}\geq 0\qquad \mbox{and}\qquad \sqrt{3}\sigma_{1}>\sqrt{\sigma_{2}^{2}+\sigma_{3}^{2}},
	\end{equation*}
	and for some  $j\in\mathbb{Z}^{+},\, j\geq 3$
	\begin{equation}\label{a2}
	\mathcal{N}_{j}:=\sum_{|\alpha|=j}\int_{\mathcal{H}_{\{\sigma,\beta\}}}\left(\partial_{x}^{\alpha}u_{0}(x)\right)^{2}\,\mathrm{d}x<\infty,
	\end{equation}
	then the corresponding solution of the IVP for the ZK equation \eqref{zk4} satisfies that for any $\nu\geq0,\ \epsilon>0$ and $\tau>4\epsilon$
	\begin{equation}\label{d1}
	\begin{split}
	&\sup_{0\leq t\leq T}\sum_{|\alpha|\leq j}\int_{\mathcal{H}_{\{\sigma,\beta-\nu t+\epsilon\}}}\left(\partial_{x}^{\alpha}u(x,t)\right)^{2}\,\mathrm{d}x\\
	&+\sum_{|\alpha|=j+1}\int_{0}^{T}\int_{\mathcal{Q}_{\{\sigma,\beta-\nu t+\epsilon, \beta-\nu t+\tau\}}}\left(\partial_{x}^{\alpha}u(x,t)\right)^{2}\,\mathrm{d}x\,\mathrm{d}t\\
	&\leq c=c\left(\|u_{0}\|_{H^{s}}; \left\{\mathcal{N}_{l}:\, 1\leq l\leq j\right\}; j;\sigma;\nu ;T;\epsilon;\tau \right).
	\end{split}
	\end{equation}
\end{thm}
\begin{rem}
	Similar results holds in the $2-$dimensional  case, for a more detailed description see \cite{LPZK}.
\end{rem}
The property described in Theorem \ref{t1}  is inherent  to some nonlinear  dispersive models, e.g  in the one dimensional case this issue have  been verified in solutions of the KdV and the Benjamin-Ono equation by Isaza, Linares, Ponce  see \cite{ILP1} and \cite{ILP2} resp. Also, Kenig, Linares, Ponce and Vega \cite{KLPV}  studied this subject in solutions of the KdV by considering initial data with   fractional regularity. Additionally, combining the approach in \cite{ILP2} and \cite{KLPV}  it was verified that   the dispersive generalized Benjamin-Ono \cite{AM1} as well as the fractional KdV equation \cite{AM2} also satisfy this property. Recently,  Mu\~{n}oz, Ponce and Saut  \cite{MPS} proved  that the solutions   of the Intermediate long-wave equation satisfy this property.

As shows  Theorem \ref{t1}, this property is not only  inherent to  one-dimensional nonlinear dispersive models, but on the contrary, its validity has been established  in multidimensional  nonlinear   dispersive  models such as:  Kadomtsev-Petviashvili II (KP-II) (see Isaza, Linares, Ponce \cite {ILP3}) as well as in the fifth order   Kadomtsev-Petviashvili II (KP5-II)  and the Benjamin-Ono-Zakharov-Kuznetsov (BO-ZK) equations  by   Nascimento \cite{Nac1}, \cite{Nac2} resp.

Note that the case in which  the additional regularity of  the initial data on the half-space $\mathcal{H}_{\{\sigma,\beta\}}$ (see \eqref{a2})  is given on a  fractional scale,  it does not fall under the scope of Theorem \ref{t1}.  
The study of this case constitutes the  main objective in  this work. More precisely, we managed to show that even in the case in which the additional regularity of the data is measured on a fractional scale this  is propagated with infinite speed.

In summary, our main result  reads as follows:
\begin{thmx}\label{zk9}
	Let $u_{0}\in H^{s}(\mathbb{R}^{n})$ with $s>s_{n}:=\frac{n+2}{2}.$  If for some $\sigma=(\sigma_{1},\sigma_{2},\dots,\sigma_{n})\in \mathbb{R}^{n},\, n\geq 2$  with
	\begin{equation}
	\sigma_{1}>0,\,\, \sigma_{2},\dots,\sigma_{n}\geq 0\qquad \mbox{and}\qquad \sqrt{\sigma_{2}^{2}+\sigma_{3}^{2}+\dots+\sigma_{n}^{2}}<\sqrt{3}\sigma_{1},
	\end{equation}
	and for some $s\in \mathbb{R},s>s_{n}$
	\begin{equation}\label{a1}
	\left\|J^{s}u_{0}\right\|_{L^{2}\left(\mathcal{H}_{\{\sigma,\beta\}}\right)}<\infty %=\int_{\mathcal{H}_{\{\sigma,\beta\}}} \left(J^{s}u_{0}(x)\right)^{2}\,\mathrm{d}x<\infty,
	\end{equation}	
	then the corresponding solution  $u=u(x,t)$ of the IVP \eqref{zk4} satisfies:  For any $\nu \geq 0,\, \epsilon>0$ and $\tau\geq5 \epsilon$ 
	\begin{equation}\label{g1}
	\sup_{0\leq t\leq T}\int_{\mathcal{H}_{\{\sigma,\beta+\epsilon-\nu t\}}}\left(J^{r}u(x,t)\right)^{2}\mathrm{d}x\leq c^{*}
	\end{equation} 	
	for any $r\in (0,s]$ with $c^{*}=c^{*}\left(\epsilon; \sigma;T; \nu ; \|u_{0}\|_{H^{s_{n}+}}; \|J^{s}u_{0}\|_{L^{2}\left(\mathcal{H}_{\{\sigma,\beta\}}\right)}\right)>0.$
	
	In addition, for any $\nu\geq 0,\, \epsilon>0$ and $\tau\geq 5\epsilon$
	\begin{equation}\label{g2.1}
	\int_{0}^{T}\int_{\mathcal{Q}_{\{\sigma ,\epsilon-\nu t+\beta ,\tau -\nu t+\beta}\}}\left(J^{s+1}u(x,t)\right)^{2}\,\mathrm{d}x\,\mathrm{d}t\leq c^{*},
	\end{equation}
	with $c^{*}=c^{*}\left(\epsilon;\tau; \sigma; T; \nu ; \|u_{0}\|_{H^{s_{n}{+}}}; \|J^{s}u_{0}\|_{L^{2}(\mathcal{H}_{\{\sigma,\beta\}})}\right)>0.$
\end{thmx}
The proof of the Theorem \ref{zk9} is mainly based in  combining the ideas used in the proof of Theorem \ref{t1}, as well as the ideas used in the  study of propagation of regularity for solutions of the KdV equation \cite{KLPV}. More precisely, it combines  an inductive argument together   with  weighted energy estimates, where  the  class of weights used  enjoy of some particular  properties that allow to capture the information  related to the regularity in  certain subsets of $\mathbb{R}^{n},n=2,3.$   Also, the method of proof  uses strongly the   \emph{Kato's smoothing effect}, this is  a property found originally by Kato in the KdV context (see Kato \cite{KATO1}).

In this sense, our contribution is mainly based on establishing certain localization formulas on half-spaces and strips  for the operator $ J^s, s>0.$
 Similarly, we show that under certain conditions it is possible to establish relationships between $J^ s,s\geq 1$ and  $\partial_{x}^{\alpha},\alpha\in \left(\mathbb{Z}^{+}\right)^{n},$ when we are restricted to certain class of subsets of the euclidean space, for a more detailed description see Lemma \ref{A}.

%we derive a  new  localization formulas for the  operator $J^{s}, s>0$ by making a systematic use of pseudo-differentials operators. Also,  we establish certain relationships between the operators $\partial_{x}^{\alpha}, \alpha\in(\mathbb{Z}^{+})^{n}$ and  $J^{s}, s\geq 1$     for   particular values of $s,$  that allow us   to recover the results in Theorem \ref{t1}.
  Although the proof of the theorem \ref{zk9} is somewhat technical,  the properties it describes are quite intuitive and have its particular flavor. In this sense,    we will present a geometric description of these when we restrict ourselves to dimension $n=2$ and $n=3$ due to its physical relevance. 
  
  In dimension  $n=2$ two situations arise according to Theorem \ref{zk9}.   The first one, namely   $\sigma=(\sigma_{1},0)$  with $\sigma_{1}>0.$
  Under this condition, the   extra regularity of the initial data $u_{0}$  in the half space $\mathcal{H}_{\{\sigma,\beta\}}$ see \eqref{a1}), that in this case is fractional (cf.\eqref{a2}).  This dynamics is exemplified in figure \ref{fig:1}, where the arrows indicate the sense of propagation  (to the left), as well as two  gray zones denoting  the regions $\mathcal{Q}_{\{\sigma,\epsilon+\beta-\nu t,\tau+\beta-\nu t\}}$ and $\mathcal{H}_{\{\sigma,\epsilon+\beta-\nu t\}},$  denoting the set of propagation and the strip  that carry out the smoothing effect resp. The set  $\mathcal{H}_{\{\sigma,\epsilon+\beta-\nu t\}}$ denotes the  moving region of the plane   
  that carries out the information corresponding to the propagation of  regularity,
  that despite   being fractional it is propagated to the left with infinite speed. Instead, the set $\mathcal{Q}_{\{\sigma,\epsilon+\beta-\nu t,\tau+\beta-\nu t\}}$ corresponds to the region where  the solution is  smoother by   one  local-derivative. More precisely, in this region  is present  the Kato's smoothing effect. 
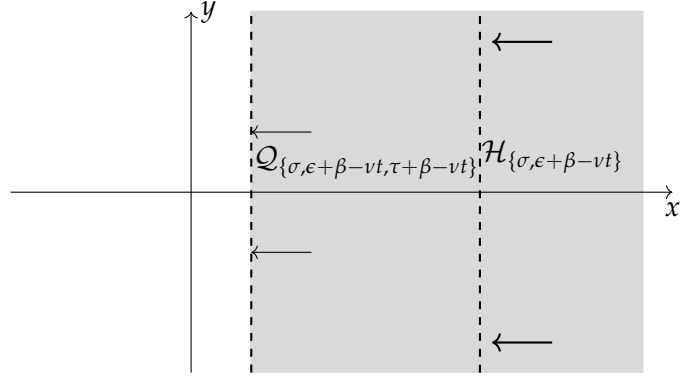
\begin{figure}[h!]
	\begin{center}
		\begin{tikzpicture}[scale=0.8]
		%\filldraw[thick, color=lightgray!30] (-1,1.5)--(5.2,1.5) -- (5.2,5) --(-1,5) -- (-1,1.5);
		%\filldraw[thick, color=lightgray!10] (-1,-1)--(-1,4) -- (4,4) --(4,-1) -- (-1,-1);
		%\filldraw[thick, color=lightgray!30] (0,4.7)--(0,2) -- (4/3,1) --(4,3) -- (4,4.7) -- (0,4.7);
		%\filldraw[thick, color=lightgray!70] (0,4.7)--(0,2) -- (0.6,1.7) --(1,1.6) --(4/3,1.6) --(1.97,1.6) --(4,3) -- (4,4.7) -- (0,4.7);
		%%\draw[thick, color=black] (1.25,4) -- (3.2,1.5);
		\filldraw[thick, color=gray!30] (2,-1) -- (2,5)  -- (5.8,5) -- (5.8,-1);
		\filldraw[thick, color=gray!30] (5.8,-1) -- (5.8,5)  -- (8.5,5) -- (8.5,-1);
		\draw[->] (3,1) -- (2,1);
		\draw[->] (3,3) -- (2,3);
		\draw[->, thick] (7,-0.5) -- (6,-0.5);
		\draw[->, thick] (7,4.5) -- (6,4.5);
		%	\draw[thick, dashed] (0,0) -- (4,0) -- (4,4) -- (0,4) -- (0,0);
		%\draw[thick] (4,3) -- (4,5);
		\draw[thick,dashed] (2,-1) -- (2,5);
		\draw[thick,dashed] (5.8,-1) -- (5.8,5);
		%\draw[thick,dashed] (0,2)--(8/3,0);
		%\draw[thick] (0,2) -- (0,5);
		%\draw[thick,dashed] (3.2,-1)--(3.2,5.3);
		%\draw[thick,dashed] (-1,1.5)--(5.2,1.5);
		%\draw[thick,dashed] (-1,5)--(5.2,5);
		\draw[->] (-2,2) -- (9,2) node[below] {$x$};
		\draw[->] (1,-1) -- (1,5) node[right] {$y$};
		%\node at (1.9,-0.7){$ \mathcal{B}_2(b)$};
		%\node at (1.2,2.2){$- \frac{1+4c}{2(1-c)}$};
		%\node at (0,0){$\bullet$};
		%\node at (0,4){$\bullet$};
		%\node at (4,0){$\bullet$};
		\node at (3.9,2.5){$\mathcal{Q}_{\{\sigma,\epsilon+\beta-\nu t,\tau+\beta-\nu t\}} $};
		\node at (7,2.6){$\mathcal{H}_{\{\sigma,\epsilon+\beta-\nu t\}} $};
		%	\node at (2.5,3.7){$t^{br}$};
		%	\node at (3.7,2.5){$t^b$};
		%	\node at (1,1){$\Omega(t)$};
		%\node at (4/3,0){$\bullet$};
		%\node at (8/3,-0.4){$\frac23$};
		%\node at (4/3,-0.4){$\frac13$};
		%\node at (8/3,0){$\bullet$};
		%\node at (0,2){$\bullet$};
		%\node at (-0.3,2){$\frac 13$};
		%\node at (0,1){$\bullet$};
		%\node at (-0.3,1){$\frac16$};
		%%\node at (0,8/3){$\bullet$};
		%\node at (0,3){$\bullet$};
		%\node at (4,1.6){$\bullet$};
		%\node at (4.3,1.6){$\frac29$};
		%\node at (-0.3,3){$\frac12$};
		%\node at (-1,0.5){$x$};
		%\node at (1,0.5){$y$};
		%\draw[thick,dashed] (2,4) arc (180:360:0.6);
		%\draw[thick,dashed] (4,2) arc (90:270:0.6);
		%\draw[thick,dashed] (2.5,2) arc (0:360:0.6);
		%\node at (4.7,4.6){$-b+\frac12$};
		%\node at (4.7,3.3){$ \mathcal{B}_3(b)$};
		%\node at (-3.5,2.5){$V$};
		%\node at (3,2.7){$W$};
		%\draw (2,0.5) arc (90:145:0.5);
		%\node at (2.7,-1.9){$x_1=\beta t$};
		%\node at (4,-1.2){$x_1+(\tan \theta) \, x_2=\beta t $};
		%\node at (0.3,0.15){$\theta$};
		%\draw (0.5,0) arc (0:55:0.5);
		%\node at (0.1,1.1){$\theta$};
		%\draw (0,0.9) arc (270:325:0.4);
		%\node at (-0.7,3.5){$\mathcal{B}_0$};
		\end{tikzpicture}
		\qquad
	\end{center}
	\caption{\emph{Sense of propagation of regularity in the case $\sigma_{1}>0,\,\sigma_{2}=0.$}}\label{fig:1}
\end{figure}
 Unlike the previous situation, the case $\sigma=(\sigma_{1},\sigma_{2})$ with $\sigma_{1},\sigma_{2}>0$ is more involved,  since  the effects on the $y-$variable yield  a change in the geometry  of   the propagation. In this case, the dynamics   
  is carried out in a diagonal  sense  as is  indicated by the arrows in the figure \ref{FIGURE}. 
 \begin{figure}[h!]
 	\begin{center}
 		\begin{tikzpicture}[scale=0.7]
 		%%\draw[thick, color=black] (1.25,4) -- (3.2,1.5);
 		\filldraw[thick, color=gray!30](-4,3.5) -- (-4,5.6)--(8.5,-1) -- (4.5,-1); 
 		\filldraw[thick, color=gray!30] (-4,5.6) -- (1,5.6)  -- (8.5,5. 6) -- (8.5,-1);
 		%	\draw[->] (2,1) -- (3,1);
 		\draw[->] (4,1)--(3,-0.2) ;
 		\draw[->] (-2,4.3)--(-3,3) ;
 		\draw[->,thick] (4.9,4.4)--(4,3.2);
 		%	\draw[->, thick] (6,4.5) -- (7,4.5);
 		%	\draw[thick, dashed] (0,0) -- (4,0) -- (4,4) -- (0,4) -- (0,0);
 		%\draw[thick] (4,3) -- (4,5);
 		\draw[thick,dashed] (-4,5.6) -- (8.5,-1);
 		\draw[thick,dashed] (4.5,-1)--(-4,3.5) ; 	
 		%	\draw[thick,dashed,blue] (0.5,5.6) -- (8.5,1.5); 	
 		%\draw[thick,dashed] (0,2)--(8/3,0);
 		%\draw[thick] (0,2) -- (0,5);
 		%\draw[thick,dashed] (3.2,-1)--(3.2,5.3);
 		%\draw[thick,dashed] (-1,1.5)--(5.2,1.5);
 		%\draw[thick,dashed] (-1,5)--(5.2,5);
 		\draw[->] (-4.5,0) -- (9,0) node[below] {$x$};
 		\draw[->] (1,-1) -- (1,5) node[right] {$y$};
 		%\node at (1.9,-0.7){$ \mathcal{B}_2(b)$};
 		%\node at (1.2,2.2){$- \frac{1+4c}{2(1-c)}$};
 		%\node at (0,0){$\bullet$};
 		%\node at (0,4){$\bullet$};
 		%\node at (4,0){$\bullet$};
 		\node at (1.6,1.7){$\footnotesize{\mathcal{Q}_{\{\sigma, \eta_{1}(t), \eta_{2}(t)\}}} $};
 		\node at (5,2.5){$\mathcal{H}_{\{\sigma,\eta_{1}(t)\}} $};
 		%	\node at (2.5,3.7){$t^{br}$};
 		%	\node at (3.7,2.5){$t^b$};
 		%	\node at (1,1){$\Omega(t)$};
 		%\node at (4/3,0){$\bullet$};
 		%\node at (8/3,-0.4){$\frac23$};
 		%\node at (4/3,-0.4){$\frac13$};
 		%\node at (8/3,0){$\bullet$};
 		%\node at (0,2){$\bullet$};
 		%\node at (-0.3,2){$\frac 13$};
 		%\node at (0,1){$\bullet$};
 		%\node at (-0.3,1){$\frac16$};
 		%%\node at (0,8/3){$\bullet$};
 		%\node at (0,3){$\bullet$};
 		%\node at (4,1.6){$\bullet$};
 		%\node at (4.3,1.6){$\frac29$};
 		%\node at (-0.3,3){$\frac12$};
 		%\node at (-1,0.5){$x$};
 		%\node at (1,0.5){$y$};
 		%\draw[thick,dashed] (2,4) arc (180:360:0.6);
 		%\draw[thick,dashed] (4,2) arc (90:270:0.6);
 		%\draw[thick,dashed] (2.5,2) arc (0:360:0.6);
 		%\node at (4.7,4.6){$-b+\frac12$};
 		%\node at (4.7,3.3){$ \mathcal{B}_3(b)$};
 		%\node at (-3.5,2.5){$V$};
 		%\node at (3,2.7){$W$};
 		%\draw (2,0.5) arc (90:145:0.5);
 		%\node at (2.7,-1.9){$x_1=\beta t$};
 		%\node at (4,-1.2){$x_1+(\tan \theta) \, x_2=\beta t $};
 		%\node at (0.3,0.15){$\theta$};
 		%\draw (0.5,0) arc (0:55:0.5);
 		%\node at (0.1,1.1){$\theta$};
 		%\draw (0,0.9) arc (270:325:0.4);
 		%\node at (-0.7,3.5){$\mathcal{B}_0$};
 		\end{tikzpicture}
 		\qquad
 	\end{center}
 	\caption{\emph{Sense of propagation of regularity   in the  2-dimensional case with  $\sigma_{1},\sigma_{2}>0,$ whence  $\eta_{1}(t):=\epsilon+\beta-\nu t$ and $\eta_{2}(t):=\tau+\beta-\nu t.$}}\label{FIGURE}
 \end{figure}
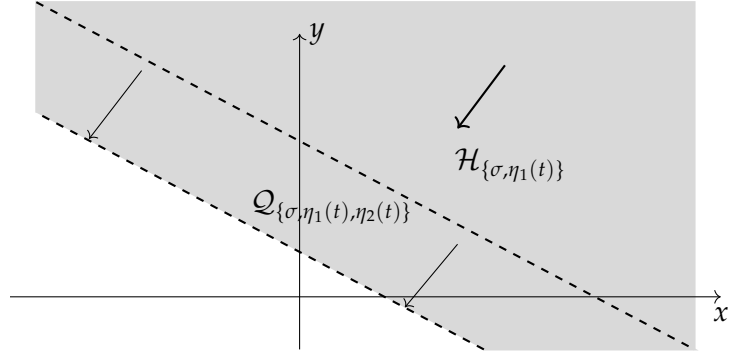
To give a  geometrical idea about the propagation of regularity phenomena in the three-dimensional case it  is even more involved  than in the $2d$ case since   it entails to describe    many more sub-cases that we do not intend to describe them by   entirely,   but instead  to fix the  ideas we   only focus on  describing  the case $\sigma=(\sigma_{1},\sigma_{2},\sigma_{3}),$ where $\sigma_{1},\sigma_{2},\sigma_{3}>0.$

The  consideration of a third variable   yield to  a more complex dynamics. In this situation, the sense of propagation  occurs  in a  diagonal sense  as  indicates  the dashed arrow  in the  figure  \ref{fig:M1}. As in the previous situations, the sets $\mathcal{Q}_{\{\sigma,\epsilon+\beta-\nu t,\tau+\beta-\nu t\}}$ and $\mathcal{H}_{\{\sigma,\epsilon+\beta-\nu t\}},$ denotes moving regions  with infinite speed of propagation indicating the  regularity propagated as well as the smoothing effect. Note that the  dashed triangles  enclose the  channel $\mathcal{Q}_{\{\sigma,\epsilon+\beta-\nu t,\tau+\beta-\nu t\}}$  and the upper triangle denotes the border of the region $\mathcal{H}_{\{\sigma,\epsilon+\beta-\nu t\}}$ when we restrict to the first octant. 

Additionally to the properties described  in Theorem \ref{zk9} we  also  provide information in the case 
when  we give  privilege  to a specific direction, this constitute our second main result  and it is summarized  in the following theorem.
 \begin{center}
	\begin{figure}
		\begin{tikzpicture}[x=0.5cm,y=0.5cm,z=0.3cm,>=stealth]
		% The axes
		\draw[->] (xyz cs:x=-5) -- (xyz cs:x=9) node[above] {$x$};
		\draw[->] (xyz cs:y=-2) -- (xyz cs:y=9) node[right] {$z$};
		\draw[->] (xyz cs:z=-2) -- (xyz cs:z=9) node[above] {$y$};
		\draw[->,dashed] (4,2,0) -- (4,2,-5) node[below] {};
		%	\filldraw[thick, color=darkgray!30] (-4,5.6) -- (1,5.6)  -- (8.5,5. 6) -- (8.5,-1);
		\filldraw[thick, color=lightgray!30](0,3,0) -- (4,0,0) -- (0,0,4)--(0,3,0); 
		\draw[thick,dashed] (0,3,0) -- (4,0,0)--(0,0,4)--(0,3,0);
		%%%
		\filldraw[thick, color=lightgray!30](0,6,0) -- (7,0,0) -- (0,0,7)--(0,6,0); 
		\draw[thick,dashed] (0,6,0) -- (7,0,0)--(0,0,7)--(0,6,0);
		% The thin ticks
		% \foreach \coo in {-9,-8,...,9}
		% {
		% 	\draw (\coo,-1.5pt) -- (\coo,1.5pt);
		% 	\draw (-1.5pt,\coo) -- (1.5pt,\coo);
		% 	\draw (xyz cs:y=-0.15pt,z=\coo) -- (xyz cs:y=0.15pt,z=\coo);
		%}
		% The thick ticks
		% \foreach \coo in {-10,-5,5,10}
		%{
		% 	\draw[thick] (\coo,-3pt) -- (\coo,3pt) node[below=6pt] {\coo};
		% 	\draw[thick] (-3pt,\coo) -- (3pt,\coo) node[left=6pt] {\coo};
		% 	\draw[thick] (xyz cs:y=-0.3pt,z=\coo) -- (xyz cs:y=0.3pt,z=\coo) node[below=8pt] {\coo};
		% }
		% Dashed lines for the points P, Q
		% \draw[dashed] 
		% (xyz cs:z=-5) -- 
		% +(0,7) coordinate (u) -- 
		% (xyz cs:y=7) -- 
		% +(-5,0) -- 
		% ++(xyz cs:x=-5,z=-5) coordinate (v) --
		% +(0,-7) coordinate (w) --
		% cycle;
		% \draw[dashed] (u) -- (v);
		% \draw[dashed] (-5,7) -- (-5,0) -- (w);
		% \draw[dashed] (3,0) |- (0,5);
		
		% Dots and labels for P, Q
		% \node[fill,circle,inner sep=1.5pt,label={left:$Q(-5,-5,7)$}] at (v) {};
		%\node[fill,circle,inner sep=1.5pt,label={above:$P(3,0,5)$}] at (3,5) {};
		% The origin
		\node[align=center] at (4,-4) (ori) {\\$\mathcal{Q}_{\{\sigma,\epsilon+\beta-\nu t,\tau+\beta-\nu t\}} $};
		\draw[->,help lines,shorten >=3pt] (ori) .. controls (1,-2) and (1.2,-1.5) .. (5,2,-1);
		\node[align=center] at (0,5,5) (ori) {\\$\mathcal{H}_{\{\sigma,\epsilon+\beta-\nu t\}} $};
		\draw[->,help lines,shorten >=3pt] (ori) .. controls (4,5,0) and (4,6,0) .. (4,2,2);
		%	\caption{\emph{Sense of propagation of regularity   in the  bidimensional case with  $\sigma_{1},\sigma_{2}>0.$}}\label{FIGURE 1}
		\end{tikzpicture}
		\caption{\emph{Sense of propagation of regularity   in the  3D case with  $\sigma_{1},\sigma_{2},\sigma_{3}>0.$}} \label{fig:M1}
	\end{figure}
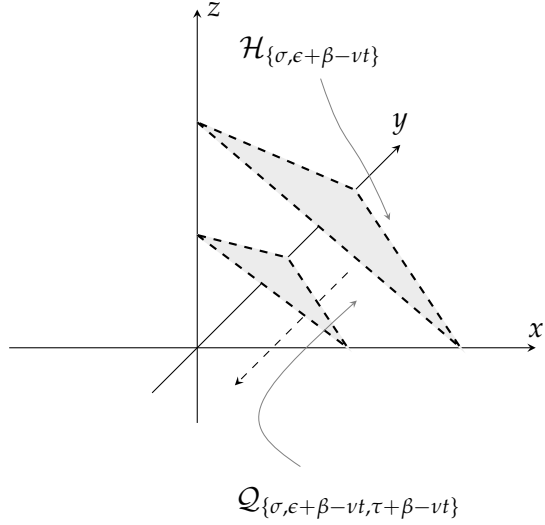
\end{center}
%Even when the  relevant dimensions for  physical interpretations corresponds to    $n=2$ and $n=3,$ the argument we provide here  holds for $n\geq 2,$ the only restriction in this sense comes  from the solution space, since our method of proof requires to have a priori bounds for  $\|u\|_{L^{\infty}_{T}H^{s}}$ and   $\|\nabla u\|_{L^{1}_{T}L^{\infty}},$  which in virtue  of Sobolev embedding is guaranteed whenever we consider initial data $u_{0}\in H^{s}(\mathbb{R}^{n}),\, s>s_{n}.$
%%\begin{center}
%%	\emph{ Suppose that the initial data $u_{0},$ satisfies $J^{s}_{x_{j}}u_{0}\in L^{2}(A)$ for some  $ j\in \{1,2,\cdots,n\};$  whence $s\in\mathbb{R},s>1+n/2$ and  $A$ is a certain class of subsets  in $\mathbb{R}^{n},n\geq 2.$ Also, assume that $u$ is the solution associated to the IVP \eqref{zk4}  in a suitable  Sobolev space $H^{r}(\mathbb{R}^{n}),\,r>1+n/2$. }
%
%\emph{Does implies  that $J^{s}_{x_{j}}u(t)\in L^{2}(B(t))$ where $B(t)\subset\mathbb{R}^{n}$ is a moving region  contained in  $A$ for all $t\in[0,T) $?}
%\end{center}
\begin{thmx}\label{zk10}
	Let $u_{0}\in H^{s_{n}^{+}}(\mathbb{R}^{n}).$  If for some $\sigma=(\sigma_{1},\sigma_{2},\dots,\sigma_{n})\in \mathbb{R}^{n}$  with $n\geq 2;$ 
	\begin{equation*}
	\sigma_{1}>0,\,\, \sigma_{2},\dots,\sigma_{n}\geq 0\qquad \mbox{and}\qquad \sqrt{\sigma_{2}^{2}+\sigma_{3}^{2}+\dots+\sigma_{n}^{2}}<\sqrt{3}\sigma_{1},
	\end{equation*}
	and for some $s\in \mathbb{R},s>s_{n}$
	\begin{equation*}
	\left\|J^{s}_{x_{j}}u_{0}\right\|_{L^{2}\left(\mathcal{H}_{\{\sigma,\beta\}}\right)}<\infty %=\int_{\mathcal{H}_{\{\sigma,\beta\}}}\left(J^{s}_{x_{1}}u_{0}(x)\right)^{2}\, \mathrm{d}x<\infty,
	\end{equation*}	
	for some $j\in \{1,2,\cdots,n\},$
	then the corresponding solution  $u=u(x,t)$ of the IVP  \eqref{zk4} satisfies:  For any $\nu \geq 0,\, \epsilon>0$ and $\tau \geq5\epsilon$ 
	\begin{equation*}
	\sup_{0\leq t\leq T}\int_{\mathcal{H}_{\{\sigma,\beta+\epsilon-\nu t\}}}\left(J^{r}_{x_{j}}u(x,t)\right)^{2}\mathrm{d}x\leq c^{*},
	\end{equation*} 	
	for any $r\in (0,s]$ with $c^{*}=c^{*}\left(\epsilon;\sigma; T; \nu; \|u_{0}\|_{H^{s_{n}+}_{x_{j}}}; \|J^{r}_{x_{j}}u_{0}\|_{L^{2}\left(\mathcal{H}_{  \{\sigma,\beta\}}\right)}\right)>0.$
	
	In addition, for any $\nu \geq 0,\, \epsilon>0$ and $\tau\geq5\epsilon,$
	\begin{equation*}
	\begin{split}
&\sum_{1\leq m\leq n,m \neq j}\int_{0}^{T}\int_{\mathcal{Q}_{\{\beta-\nu t+\epsilon,\tau-\nu t+\beta}\}}\left(\partial_{x_{m}}J^{s}_{x_{j}}u\right)^{2}\,\mathrm{d}x\,\mathrm{d}t\\
	&\quad +\int_{0}^{T}\int_{\mathcal{Q}_{\{\beta-\nu t+\epsilon,\tau-\nu t+\beta}\}}\left(J^{s+1}_{x_{j}}u(x,t)\right)^{2}\,\mathrm{d}x\,\mathrm{d}t\leq c
	\end{split}
	\end{equation*}
	with $c=c\left(\epsilon;\sigma;\tau; T; \nu; \|u_{0}\|_{H^{s_{n}+}}; \|J^{r}_{x_{j}}u_{0}\|_{L^{2}\left(\mathcal{H}_{\{\sigma,\beta\}}\right)}\right)>0.$
\end{thmx}
An argument  quite similar to the one given  for  the proof of  Theorem \ref{zk10}  also applies  for a proof of Theorem \ref{zk9}.

A quite  similar description of  the phenomena  presented in the Theorem \ref{zk10} can be given in geometrical terms  as we did for Theorem \ref{zk9} when we restrict ourselves to  dimension $n=2$ and $n=3$ see figure \ref{fig:1} and figure \ref{fig:M1} resp.  
  
Finally, as a by product  we present   a consequence of  Theorem \ref{zk9}  that describes   the behavior of the solution and its derivatives  in the remainder part of the half-space described in Theorem \ref{zk9}.
\begin{cor}\label{cor11}
	Let $u\in C\left([-T,T]: H^{s_{n}^{+}}(\mathbb{R}^{n})\right)$ be a solution of  the equation in \eqref{zk4} described by theorem A.
	Then, for  any $t\in (0,T)$ and $\delta>0,$ the following inequality holds:
	\begin{equation*}
	\int_{\mathbb{R}^{n}}\frac{1}{\langle \left(\sigma\cdot x-\beta\right)_{-}\rangle^{s+\delta}}\left(J^{s}u(x,t)\right)^{2}\mathrm{d}x\lesssim_{\delta,s,\sigma} \frac{1}{t},
	\end{equation*}
	where $x_{-}=\max\{0,-x\}$.
\end{cor}
\subsubsection{Organization of the paper}
 In the section 2  we  present a full description of the notation to be used throughout the document. In what   concerns  to section 3  we present a short review of several well known results about pseudo-differential operators. Additionally, in section 3 we deduce    new localization formulas for the operator $J^{s},s>0,$  that will be used extensively through the proof of Theorem \ref{zk9}. Finally, in the section 4 we present the proof of theorem \ref{zk9},  and  some of its  consequences.

\section{Notation}
%%%%%%%%%
In this section  we introduce the notation to be used throughout  all this document.

We adopt the following convention for the \emph{Fourier transform}
\begin{equation*}
\widehat{f}(\xi):=\int_{\mathbb{R}^{n}}e^{-2\pi ix\cdot \xi}f(x)\,\mathrm{d}x.
\end{equation*}
For $x\in\mathbb{R}^{n},$ we denote  $\langle x\rangle:=\left(1+|x|^{2}\right)^{1/2}.$  Additionally,  for any $s\in\mathbb{R}$ we define the operator $J^{s}$ via its Fourier transform as $\widehat{J^{s}f}(\xi)= \langle \xi\rangle ^{s}\widehat{f}(\xi).$  In the particular case that  be required to emphasize the action over a specific variable  we  write 
\begin{equation*}
\widehat{J^{s}_{x_{j}}f}(\xi):= \langle \xi_{j}\rangle^{s} \widehat{f}(\xi),\quad \xi=(\xi_{1},\xi_{2},\cdots,\xi_{n})\in\mathbb{R}^{n},
\end{equation*}
where  $j\in\mathbb{Z}^{+}$ with $1\leq j\leq n.$

The set of the \emph{Schwarz functions} will be denoted by $\mathcal{S}(\mathbb{R}^{n}),$ and its dual, the set of the  \emph{tempered distributions} will be denoted by $\mathcal{S}'(\mathbb{R}^{n}).$ The \emph{dual pair} between $\mathcal{S}'(\mathbb{R}^{n})$ and $\mathcal{S}(\mathbb{R}^{n})$ will be indicated as usual by  ${\displaystyle \langle\cdot,\cdot\rangle_{\mathcal{S}',\mathcal{S}}.}$

For $1\leq p\leq \infty,$\, $L^{p}(\mathbb{R}^{n})$  is the usual  Lebesgue space  with the norm  $\|\cdot\|_{L^{p}}.$  Additionally, for $s\in \mathbb{R},$  we consider the \emph{Sobolev space }$H^{s}(\mathbb{R}^{n})$  that is defined   as   
\begin{equation*}
H^{s}(\mathbb{R}^{n}):=\left\{f\in \mathcal{S}'(\mathbb{R}^{n})\,|\, \left\|J^{s}f\right\|_{L^{2}}<\infty\right\}.
\end{equation*}
% In this context,  we define   $${\displaystyle H^{\infty}(\mathbb{R}^{n})=\bigcap_{s\geq0} H^{s}(\mathbb{R}^{n}).}$$
Let $f=f(x,t)$  be a function  defined for $x=(x_{1},x_{2},\dots x_{n})\in \mathbb{R}^{n}$ and $t$ in the time interval $[0,T],$ with $T>0$  or in the hole line $\mathbb{R}$, then,  if $A$ denotes any of the spaces defined above, we define  the spaces  $L^{p}_{T}A$ and $L_{t}^{p}A$ by the norms 
\begin{equation*}
\|f\|_{L^{p}_{T}A}:=\left(\int_{0}^{T}\|f(\cdot,t)\|_{A}^{p}\,\mathrm{d}t\right)^{1/p}\quad \mbox{and}\quad \|f\|_{L^{p}_{t}A}:=\left(\int_{\mathbb{R}} \|f(\cdot,t)\|_{A}^{p}\,\mathrm{d}t\right)^{1/p},
\end{equation*}
for $1\leq p\leq \infty,$ with the  natural modification in the case $p=\infty.$

For $A,B$ operators  we will denote  the \emph{ commutator} between $A$ and $B$ by 
\begin{equation*}
[A; B]:=AB-BA.
\end{equation*}
%Moreover, we use similar definitions for the  mixed spaces $L_{x}^{q}L_{t}^{p}$ and $L_{x}^{q}L_{T}^{p}$  with $1\leq p,q\leq \infty.$

For two quantities  $A$ and $B$, we denote  $A\lesssim B$  if $A\leq cB$ for some constant $c>0.$ Similarly, $A\gtrsim B$  if  $A\geq cB$ for some $c>0.$  Also for two positive quantities, $A$  $B$  we say that are \emph{ comparable}  if $A\lesssim B$ and $B\lesssim A,$ when such condition be  satisfied we will indicate it by $A \equiv B.$   The dependence of the constant $c$  on other parameters or constants are usually clear from the context and we will often suppress this dependence whenever it be  possible.

For any real number $a,$ we denote  by $a+$ the quantity $a+\epsilon$ for  any $\epsilon>0.$

\section{Pseudo-differential Operators}
In the following section it is our intention to provide   a brief summary  about some well known facts about  pseudo-differential operators as well as  several  properties   in  Sobolev spaces,  that will be relevant in our analysis. 
\begin{defn}\label{zk18}
	Let $m\in \mathbb{R}.$  Let $\mathbb{S}^{m}(\mathbb{R}^{n}\times\mathbb{R}^{n})$ denote the set of functions $a\in C^{\infty}(\mathbb{R}^{n}\times \mathbb{R}^{n})$ such that 
for all $\alpha$ and all $\beta$ multi-index
	\begin{equation}\label{sym}
	\left|\partial_{x}^{\alpha}\partial_{\xi}^{\beta}a(x,\xi)\right|\lesssim_{\alpha,\beta}(1+|\xi|)^{m-|\beta|},\quad \mbox{for all}\quad x \in\mathbb{R}^{n}.
	\end{equation}
	An element $a\in \mathbb{S}^{m}(\mathbb{R}^{n}\times\mathbb{R}^{n})$ is called a \emph{symbol of order $m.$}
\end{defn}
\begin{rem}
For the sake of simplicity in the notation from here on we	will   suppress  the dependence of the space $\mathbb{R}^{n}$ when we make reference to  a symbol in  a particular class.
\end{rem}
\begin{defn}
	\emph{A pseudo-differential  operator} is a mapping $f\mapsto \Psi f$ given by 
	\begin{equation}\label{e1.1}
	(\Psi f)(x)=\int_{\mathbb{R}^{n}}e^{2\pi\mathrm{i}x\cdot\xi}a(x,\xi)\widehat{f}(\xi)\,\mathrm{d}\xi,
	\end{equation}
	where $a(x, \xi)$ is the symbol of $\Psi.$
\end{defn}
 \begin{rem}
 	In order to emphasize the role of the symbol $a$ we will often write  $\Psi_{a}.$
 \end{rem}
\begin{defn}
	If $a(x,\xi)\in \mathbb{S}^{m},$ the operator $\Psi_{a}$ is said to belong  to $\mathrm{OP\mathbb{S}^{m}.}$ More precisely,  if $\Sigma$ is  any symbol class and $a(x,\xi)\in \Sigma,$ we say that $\Psi_{a}\in \mathrm{OP}\Sigma.$ 
\end{defn}
A quite remarkable property that pseudo-differential operators enjoy  is the existence of the adjoint operator, that is described below in terms of  its asymptotic decomposition. 
\begin{thm}
	Let $a\in \mathbb{S}^{m}.$ Then, there exist $a^{*}\in \mathbb{S}^{m}$ such that  $\Psi_{a}^{*}=\Psi_{a^{*}},$ and for all $N\geq 0,$
	\begin{equation*}
	a^{*}(x,\xi)-\sum_{|\alpha|<N}\frac{(2\pi i)^{-|\alpha|}}{\alpha!}\partial_{\xi}^{\alpha}\partial_{x}^{\alpha}\overline{a}(x,\xi)\in\mathbb{S}^{m-N}.
	\end{equation*}
\end{thm}
\begin{proof}
	See Stein \cite{stein3} chapter VI.
\end{proof}
\begin{rem}
	In fact from the formula above it  follows  that $\Psi_{a}^{*}=\Psi_{\overline{a}}\,\,\mathrm{mod}\, \mathrm{OP}\mathbb{S}^{m-1}.$ 
\end{rem}
Also, the existence of  the adjoint  allow to extend the action of $\Psi_{a}$ over a  wider class of objects, such as tempered distributions.
\begin{thm}
	If $a\in \mathbb{S}^{m},$ then $\Psi_{a}$ defined  in \eqref{e1.1} is a continuous  operator \linebreak$ \Psi_{a}:\mathbb{S}(\mathbb{R}^{n})\longrightarrow C^{\infty}(\mathbb{R}^{n}).$ Additionally, the map can be extended to a continuous  map 
	\begin{equation*}
	\Psi_{a}:\mathcal{S}'(\mathbb{R}^{n})\longrightarrow \mathcal{S}'(\mathbb{R}^{n}).
	\end{equation*}
More precisely, for $u\in\mathcal{S}'$ and $v\in\mathbb{S}$
	\begin{equation*}
	\langle \Psi_{a} u,v\rangle_{\mathcal{S}',\mathcal{S}}:=	\langle  u,\overline{\Psi_{a^{*}}\overline{v}}\rangle_{\mathcal{S}',\mathcal{S}}.
	\end{equation*}
\end{thm}
In fact the proof of this theorem  is based in the following lemma  that condensates several properties to be used later.
\begin{lem}
	Let $a\in \mathbb{S}^{m},\, v\in \mathcal{S}(\mathbb{R}^{n}).$ Then for all $\xi,\eta\in\mathbb{R}^{n},$
	\begin{equation*}
	\left|\int_{\mathbb{R}^{n}}v(x)a(x,\xi)e^{2\pi ix\cdot \eta}\,\mathrm{d}x\right|\leq C_{N}\left(1+|\xi|\right)^{m}\left(1+|\eta| \right)^{-N},
\end{equation*}
for all $N\geq 0.$
\end{lem}
\begin{proof}
	The proof follows combining Leibniz rule and  definition \ref{sym}. 
\end{proof}
%Another property  is  the kernel representation  that is if $K\in \mathcal{D}'(\mathbb{R}^{n}\times \mathbb{R}^{n})$, then  there exist a kernel,  that we also denote by $K$ that    verifies  $K:C^{\infty}_{0}(\mathbb{R}^{n})\longrightarrow \mathcal{D}'(\mathbb{R}^{n}),$ and   it is defined as $ \langle Ku,v\rangle =\langle K,uv\rangle.$ The converse  is also true and  it is known in the literature as the \emph{Schwarz kernel Theorem,} for a more detailed information on this subject see Treves \ref{}. 

A quite interesting and useful remark about this extension  is related with  the representation of pseudo-differential operators. More precisely, for $u,v\in\mathcal{S}(\mathbb{R}^{n}),$ real valued functions 
\begin{equation}\label{e1.3}
\begin{split}
\langle \Psi_{a}u,v\rangle_{\mathcal{S}',\mathcal{S}}
&=\int_{\mathbb{R}^{n}}v(x)\Psi_{a}u(x)\,\mathrm{d}x\\
&=\int_{\mathbb{R}^{n}}\int_{\mathbb{R}^{n}}a(x,\xi)e^{2\pi ix\cdot\xi}v(x)\widehat{u}(\xi)\,\mathrm{d}\xi\,\mathrm{d}x\\
&=\int_{\mathbb{R}^{n}}\int_{\mathbb{R}^{n}}\int_{\mathbb{R}^{n}}a(x,\xi)e^{2\pi i (x-y)\cdot\xi}v(x)u(y)\,\mathrm{d}y\,\mathrm{d}\xi\mathrm{d}x\\
&=\langle K,uv\rangle_{\mathcal{S}',\mathcal{S}} 
\end{split}
\end{equation}
whence we get after  interpreting properly  as a distributional integral  that 
\begin{equation*}
K(x,x-y):=\int_{\mathbb{R}^{n}}a(x,\xi)e^{2\pi i (x-y)\cdot\xi}\,\mathrm{d}\xi.
\end{equation*}
This  brief description   of pseudo-differential operators in terms of kernels  is  summarized  in the following theorem. 
\begin{thm}[Realization of Pseudo-differential operators as singular integrals]\label{rea}
	Let $a\in \mathbb{S}^{m},\,m\in\mathbb{R}$ and $\Psi_{a}$ its corresponding pseudo-differential operator associated. Then, there exist  a kernel $k_{a}\in C^{\infty}\left(\mathbb{R}^{n}\times\mathbb{R}^{n}-\{0\}\right)$ satisfying the following properties:
	\begin{itemize}
		\item[(i)] The operator $\Psi_{a}$ admits the following representation
		\begin{equation}\label{k1}
		(\Psi_{a}f)(x)= \int_{\mathbb{R}^{n}} k_{a}(x,x-y)f(y)\,\mathrm{d}y,\qquad \mbox{if}\quad x\notin\supp(f);
		\end{equation}
		\item[(ii)] for all $\alpha,\beta$ multi-index and all $N\geq 0$,
		\begin{equation}\label{k2}
		\left|\partial_{x}^{\alpha}\partial_{\xi}^{\beta}k_{a}(x,z)\right|\lesssim_{\alpha,\beta,N,\delta}|z|^{-n-m-|\beta|-N},\qquad  |z|\geq \delta,
		\end{equation}
		if $n+N+m+|\beta|>0.$
	\end{itemize}
\end{thm}
\begin{proof}
	See Stein \cite{stein3},  chapter VI.
\end{proof}
Additionally   the product $\Psi_{a}\Psi_{b}$ of two  operators  with symbols $a(x,\xi)$ and $b(x,\xi)$  respectively  is a pseudo-differential operator $\Psi_{c}$  with symbol  $c(x,\xi).$ More precisely, the description of the symbol $c$ is summarized in the following theorem:
\begin{thm}
	Suppose $a$ and $b$ symbols belonging to $\mathbb{S}^{m}$ and  $ \mathbb{S}^{r}$
	respectively. Then, there is  a symbol $c$ in  $\mathbb{S}^{m+r}$ so that 
	\begin{equation*}
\Psi_{c}=\Psi_{a}\circ\Psi_{b}.
	\end{equation*}
	Moreover, 
	\begin{equation*}
	c\sim\sum_{\alpha}\frac{(2\pi i)^{-|\alpha|}}{\alpha!}\partial_{\xi}^{\alpha}a\partial^{\alpha}_{x}b,
	\end{equation*}
	in the sense  that 
	\begin{equation*}
	c-\sum_{|\alpha|<N}\frac{(2\pi i)^{-|\alpha|}}{\alpha!}\partial_{\xi}^{\alpha}a\,\partial^{\alpha}_{x}b\in \mathbb{S}^{m+r-N}, \quad \mbox{for all integer} \quad N,\, N\geq 0. 
	\end{equation*}
\end{thm}
\begin{proof}
	For the proof see Stein \cite{stein3} chapter VI.
\end{proof}
\begin{rem}
	Note that $c-ab\in  \mathbb{S}^{m+r-1}.$ Moreover,   each  symbol of the form   $\partial_{\xi}^{\alpha}a\,\partial_{x}^{\alpha}b$ lies  in the class $\mathbb{S}^{m+r-|\alpha|}.$
\end{rem}
A direct consequence of  the decomposition above  is that it allows to describe  explicitly   up to an error term,  operators such as  commutators between pseudo- differential operators as is described below:
\begin{prop}\label{prop1}
	For $a\in \mathbb{S}^{m}$ and $b\in \mathbb{S}^{r}$  we define the commutator $\left[\Psi_{a};\Psi_{b}\right]$ by  
	\begin{equation*}
	\left[\Psi_{a};\Psi_{b}\right]=\Psi_{a}\circ\Psi_{b}-\Psi_{b}\circ\Psi_{a}.
	\end{equation*}
	 Then, the operator  ${\displaystyle  \left[\Psi_{a};\Psi_{b}\right]\in \mathrm{OP}\mathbb{S}^{m+r-1},}$ has by principal   symbol the Poisson bracket, i.e, 
	\begin{equation*}
	 \sum_{|\alpha|=1}^{n}\frac{1}{2\pi i}\left(\partial_{\xi}^{\alpha}a\,\partial_{x}^{\alpha}b- \partial_{x}^{\alpha}a\,\partial_{\xi}^{\alpha}b\right)\,\, \mathrm{mod}\,\, \mathbb{S}^{m+r-2}.
	\end{equation*} 
\end{prop}
Also,  certain  class the pseudo-differential operators enjoy of some continuity properties  as the described below.
\begin{thm}\label{zkth1}
	Suppose that $a$ is symbol with  $a\in \mathbb{S}^{0}.$ Then,  the operator $\Psi_{a}$ given by 
		\begin{equation*}
	(\Psi_{a} f)(x)=\int_{\mathbb{R}^{n}}e^{2\pi\mathrm{i}x\cdot\xi}a(x,\xi)\widehat{f}(\xi)\,\mathrm{d}\xi,
	\end{equation*}
	  initially defined on $\mathcal{S}(\mathbb{R}^{n}), $ extends  to a bounded  operator from $L^{2}(\mathbb{R}^{n}) $ to itself.
\end{thm}
\begin{proof}
	The proof can be consulted in Stein \cite{stein3} Chapter VI, Theorem 1.
\end{proof}
 Also, thorough our analysis it  will be necessary to provide upper bound for a certain commutator expressions. More precisely,  we will require the use of the following result.
\begin{thm}[Kato \& Ponce \cite{KATOP2}]\label{KPC}
	Let $s>0,$ and $ p\in(1,\infty).$  Then, for  $f,g\in \mathcal{S}(\mathbb{R}^{n})$ the following inequalities hold:
	\begin{equation}\label{zk2}
	\left\|\left[J^{s};f\right]g\right\|_{L^{2}}\lesssim \|\nabla f\|_{L^{\infty}}\|J^{s-1}g\|_{L^{2}}+\|J^{s}f\|_{L^{2}}\|g\|_{L^{\infty}},
	\end{equation}
	and 
	\begin{equation}\label{zk8}
	\|J^{s}(fg)\|_{L^{2}}\lesssim \|f\|_{L^{\infty}}\|J^{s}g\|_{L^{2}}+\|J^{s}f\|_{L^{2}}\|g\|_{L^{\infty}}.
	\end{equation}
\end{thm}
\begin{proof}
	See \cite{KATOP2}.
	\end{proof}
\section{Localization Tools}
In this section we intend to provide the necessary tools  to describe our results. The representation of pseudo-differential operators as the displayed in  \eqref{k1}  together with the decay property  \eqref{k2}  will find a wide range of applicability in our analysis.  %More precisely, in situations as the summarized in the following subsection below. %where the support of functions involved are separated. 
\begin{lem}\label{lem1}
	Let $\Psi_{a}\in\mathrm{OP\mathbb{S}^{r}}.$ Let  $ \alpha=\left(\alpha_{1},\alpha_{2},\dots,\alpha_{n}\right)$ be a multi-index with $|\alpha|\geq 0.$  If $f\in L^{2}(\mathbb{R}^{n})$ and $g\in L^{p}(\mathbb{R}^{n}),\, p\in [2,\infty]$  with 
	\begin{equation}\label{e16}
	\dist\left(\supp(f),\supp(g)\right)\geq \delta>0,
	\end{equation}
	then, 
	\begin{equation*}
	\left\|g\partial_{x}^{\alpha}\Psi_{a}f\right\|_{L^{2}}\lesssim \|g\|_{L^{p}}\|f\|_{L^{2}},
	\end{equation*}
	where $\partial_{x}^{\alpha}:= \partial_{x_{1}}^{\alpha_{1}}\dots\partial_{x_{n}}^{\alpha_{n}},.$ 
\end{lem}
\begin{proof}
The proof follows as a direct application of the representation theorem \ref{rea}.  %So that,  for the sake of brevity  we will omit it. % and Young's inequality for convolutions.
In the following  we will consider $f\in \mathcal{S}(\mathbb{R}^{n}),$ the general case can be obtained by mollifying $f.$

%Notice that for any $r\in \mathbb{R},$   the operator  $J^{r}$  defines a pseudo-differential operator whose main symbol  belongs to the class $\mathbb{S}^{r}.$  
%  
  
  So that, in virtue of  \eqref{e16} and  Theorem \ref{rea}, the following representation holds 
  \begin{equation*}
  g(x)\Psi_{a}\partial_{x}^{\alpha}f(x)=\int_{\mathbb{R}^{n}}g(x) k_{a}(x,x-z)\partial_{z}^{\alpha}f(z)\,\mathrm{d}z,
  \end{equation*}
  where $k_{a}$ is the distributional kernel associated to $\Psi_{a}.$
  
 Next, integration by parts produce 
\begin{equation*}
\begin{split}
 g(x)\partial^{\alpha}_{x}\Psi_{a}f(x)&=(-1)^{|\alpha|}\int_{\{|x-z|\geq \delta\}}g(x)\partial_{z}^{\alpha} k_{a}(x,x-z)f(z)\,\mathrm{d}z.
 \end{split}
\end{equation*}
Finally, we  combine Young's inequality  to obtain 
\begin{equation*}
\begin{split}
\left\|g\partial^{\alpha}\Psi_{a}f\right\|_{L^{2}}
&\lesssim_{\alpha,r,N} \left\|g\left(\left(\frac{\mathbb{1}_{\{|\cdot|\geq \delta\}}}{|\cdot|^{n+r+|\alpha|+N}}\right)*f\right)\right\|_{L^{2}}\\
&\leq c_{\alpha,r,N}\left\|g\right\|_{L^{p_{1}}}\left\|\frac{\mathbb{1}_{\{|\cdot|\geq \delta\}}}{|\cdot|^{n+r+|\alpha|+N}}*f\right\|_{L^{p_{2}}}\\
&\lesssim_{\alpha,r,N}\left\|g\right\|_{L^{p_{1}}} \|f\|_{L^{2}}\left\|\frac{\mathbb{1}_{\{|\cdot|\geq \delta\}}}{|\cdot|^{n+r+|\alpha|+N}}\right\|_{L^{p_{3}}}
\end{split}
\end{equation*}
where the indexes $p_{1},p_{2},$ and $p_{3}$ satisfy:
\begin{equation*}
\frac{1}{p_{1}}+\frac{1}{p_{2}}=\frac{1}{2} \quad \mbox{and} \quad \frac{1}{2}+\frac{1}{p_{2}}=\frac{1}{p_{3}},
\end{equation*}
with $2\leq p_{1},p_{2}\leq \infty$ and $1\leq p_{3}\leq 2.$

Hence, after choosing $N$  properly  we get finally
\begin{equation*}
\left\|g\partial_{x}^{\alpha}\Psi_{a}f\right\|_{L^{2}}\lesssim_{n,r,\delta,\alpha} \|g\|_{L^{p_{1}}}\|f\|_{L^{2}}.
\end{equation*}
\end{proof}
The first  result that incorporate  the use of the previous  considerations becomes summarized in the following lemma.
\begin{lem}\label{zk19}
	Let $\Psi_{a}\in \mathrm{OP\mathbb{S}^{0}}.$ Assume that 
$ f\in H^{s}(\mathbb{R}^{n}),\, s<0.$   If   $\theta_{1}f\in L^{2}(\mathbb{R}^{n}),$ then 
	\begin{equation*}
	\theta_{2}\Psi_{a} f\in L^{2}(\mathbb{R}^{n}).
	\end{equation*}
\end{lem}
\begin{proof}
	Let   $\theta_{1},\theta_{2}$ be $ C^{\infty}(\mathbb{R}^{n})$  with bounded derivatives of all orders, such that:  $0 \leq \theta_{1},\theta_{2}\leq 1,$
	and  their respective supports satisfy:
	\begin{equation}\label{e1.2}
	\dist\left(\supp\left(1-\theta_{1}\right),\supp\left(\theta_{2}\right)\right)\geq \delta,
	\end{equation}
	for some $\delta>0$.
	
	First, we decompose   $f$ by incorporating  $\theta_{1}$ and $\theta_{2}$  as follows:
\begin{equation}\label{zk3}
\begin{split}
\langle \theta_{2}\Psi_{a} f,g\rangle_{\mathcal{S}',\mathcal{S}}&=\langle \theta_{2}\Psi_{a} \left(\theta_{1}f\right) +\theta_{2}\Psi_{a}\left(\left(1-\theta_{1}\right)f\right), g\rangle_{\mathcal{S}',\mathcal{S}} \\
&=\langle \theta_{2}\Psi_{a} \left(\theta_{1}f\right), g\rangle_{\mathcal{S}',\mathcal{S}}+ \langle \theta_{2}\Psi_{a}\left(\left(1-\theta_{1}\right)f\right), g\rangle_{\mathcal{S}',\mathcal{S}}
\end{split}
\end{equation}
for all $g\in \mathcal{S}(\mathbb{R}^{n}).$

In view that $\Psi_{a}\in \mathrm{OP} \mathbb{S}^{0}$  and $\theta_{1}f\in L^{2}(\mathbb{R}^{n}),$ it is clear that combining  Theorem \ref{zkth1} and hypothesis,  it follows that  
 \begin{equation*}
 \left\|\theta_{2}\Psi_{a}\left(\theta_{1}f\right)\right\|_{L^{2}}\lesssim \left\|\theta_{1}f\right\|_{L^{2}}<\infty.
\end{equation*}
Notice that the second term in  \eqref{zk3}  condensates all the information  of $f$  that does not behaves as a function.  In this sense, we proceed as  we   described previously  in \eqref{e1.3}, that is, for $g\in\mathcal{S}(\mathbb{R}^{n})$  
\begin{equation*}
\begin{split}
&\langle \theta_{2}\Psi_{a}\left(\left(1-\theta_{1})\right)f\right),g\rangle_{\mathcal{S}',\mathcal{S}}\\
&=\langle \Psi_{a}\left(\left(1-\theta_{1})\right)f\right),  \theta_{2}g\rangle_{\mathcal{S}',\mathcal{S}}\\
&=\int_{\mathbb{R}^{n}}\int_{\mathbb{R}^{n}}\int_{\mathbb{R}^{n}}a(x,\xi)e^{2\pi i (x-y)\cdot\xi}\theta_{2}(x)g(x)(1-\theta_{1})(y)f(y)\,\mathrm{d}y\,\mathrm{d}\xi\,\mathrm{d}x,
\end{split}
\end{equation*}
 so that,  after interpreting properly  in the distributional sense  together with  Theorem \ref{rea} yield  
\begin{equation*}
\begin{split}
\theta_{2}(x)\Psi_{a}\left(\left(1-\theta_{1}\right)f\right)(x) %&= \int_{\mathbb{R}^{n}}\ex^{2\pi\iu x\cdot\xi}a(x,\xi)\theta_{2}(x) \widehat{\left(\left(1-\theta_{1}\right)f\right)}(\xi)\,\mathrm{d}\xi\\
%&= \int_{\mathbb{R}^{n}}\int_{\mathbb{R}^{n}}\theta_{2}(x)\ex^{2\pi\iu x\cdot \xi} a(x,\xi)\widehat{(1-\theta_{1})}(\eta)\widehat{f}(\eta)\,\mathrm{d}\xi\\
&= \int_{\mathbb{R}^{n}}\theta_{2}(x)k_{a}(x,x-z)(1-\theta_{1})(z)f(z)\,\mathrm{d}z\\
&=\int_{\{|x-z|\geq \delta\}}\theta_{2}(x)k_{a}(x,x-z)(1-\theta_{1})(z)f(z)\,\mathrm{d}z,
\end{split}
\end{equation*}
being the last  equality above  a consequence of \eqref{e1.2}.

Next, we choose $m\in \mathbb{Z}$   such that $s>-2m,$ then  it is clear that    $f\in H^{-2m}(\mathbb{R}^{n}).$ So that,  after applying integration by parts we get 
\begin{equation*}
\begin{split}
&\int_{B_{x}(\delta)^{c}}\theta_{2}(x)k_{a}(x,x-z)(1-\theta_{1})(z)f(z)\,\mathrm{d}z\\
&=\int_{B_{x}(\delta)^{c}}J^{2m}\left(\theta_{2}(x)k_{a}(x,x-z)(1-\theta_{1})(z)\right)J^{-2m}f(z)\,\mathrm{d}z\\
%&=\sum_{j=0}^{m}\sum_{|\beta|\leq 2j}c_{\beta,m,j}\int_{\{|x-z|\geq \delta\}}\theta_{2}(x)\partial_{z}^{\beta}\left(k_{a}(x,x-z)(1-\theta_{1})(z)\right)J^{-2m}f(z)\,\mathrm{d}z\\
&=%\sum_{j=0}^{m}\sum_{|\beta|\leq 2j}\sum_{\gamma\leq \beta}
\sum_{j,\beta,\gamma}c_{\beta,m,j,\beta,\gamma}\int_{B_{x}(\delta)^{c}}\!\!\theta_{2}(x)\partial_{z}^{\gamma}k_{a}(x,x-z)\partial_{z}^{\beta-\gamma}\left(\left(1-\theta_{1})(z)\right)\right)J^{-2m}f(z)\,\mathrm{d}z,
\end{split}
\end{equation*}
whence $B_{x}(\delta)$ denotes the open ball with center at $x$ and radius $\delta>0.$

The proof finish after  combining  Young's convolution  inequality  and the representation theorem \ref{rea}, whence we get  finally that 
$\theta_{2}\Psi_{a}\left(\left(1-\theta_{1}\right)f\right)\in L^{2}(\mathbb{R}^{n}).$
\end{proof}
In the next part   we will describe the operator  $J^{s},\, s>0$  when we restrict  on a certain class of half-spaces.%$\mathbb{R}^{n}.$

%Also another kind of distinguished subset  of the euclidean space  where our analysis will take place  are  cones. More precisely, we define the \emph{ conic neighborhood}  with  vertex $p\in\mathbb{R}^{n},$  axis $ q$ and opening $\gamma\in(0,\pi),$  as the set 
%\begin{equation}\label{cone1}
%\mathfrak{C}_{\left\{p,q,\gamma\right\}}=\left\{p\in \mathbb{R}^{n}\,|\, \left(p-p_{0}\right)\cdot q\leq \left\|p-p_{0}\right\|	\cos \gamma \right\}
%% \end{equation}
%Next, we present a tool  that allow us  to recover  information about smoothness when we restrict to a special kind of domains in $\mathbb{R}^{n}.$   
\begin{lem}\label{zk37}
	Let $f\in L^{2}(\mathbb{R}^{n})$  and  $\sigma=(\sigma_{1},\sigma_{2},\dots,\sigma_{n})\in \mathbb{R}^{n}$ such that $\sigma_{1}>0,$ for $j=1,2,\dots,n.$  Also assume   that   
	\begin{equation*}
	J^{s}f\in L^{2}\left(\mathcal{H}_{\{\sigma,\alpha\}}\right),\quad s>0.
	\end{equation*}
	Then, for any $\epsilon>0$ and any $r\in (0,s]$
	\begin{equation*}
	J^{r}f\in L^{2}\left(\mathcal{H}_{\{\sigma,\,\alpha+\epsilon\}}\right).
	\end{equation*}
\end{lem}
%The proof of this  lemma  follows by using the same argument of the proof provided by Kenig. et al \cite{KLPV}, and just some few  minor changes have to be made. Nevertheless, for the interested reader, we will  present here the main steps in the argument. 
\begin{proof}
Let $\epsilon>0.$ In the following  $\theta_{1},\theta_{2}$ are  smooth functions with bounded derivatives satisfying:
 $0\leq \theta_{1},\theta_{2}\leq 1,$  
	\begin{equation*}
	\begin{split}
	\theta_{1}(x)= 
	\begin{cases} 
	1 & x\in \overline{\mathcal{H}_{\{\sigma,\beta+\frac{\epsilon}{4}\}}} \\
	 0 & x\in \mathcal{H}_{\{\sigma,\beta\}}^{c} 
	\end{cases}
	\qquad \mbox{and}\qquad
	\theta_{2}(x)= 
	\begin{cases} 
	1 & x\in \overline{\mathcal{H}_{\{\sigma,\beta+\epsilon\}}} \\
	0 & x\in \mathcal{H}_{\{\sigma,\beta+\frac{\epsilon}{2}\}}^{c} ,
	\end{cases}
	\end{split}
	\end{equation*}
so that, by construction it follows  that 
\begin{equation*}
\dist\left(\supp\left(1-\theta_{1}\right),\supp\left( \theta_{2}\right)\right)\geq\frac{\epsilon}{4|\sigma|}>0.
\end{equation*}
	Next, for $z\in \mathbb{C},$ we define the function
	\begin{equation*}
	F(z):= \theta_{2}J^{sz}f,\qquad z=\alpha+i\tau,\quad \alpha, \tau\in \mathbb{R}.
	\end{equation*}
 The function $F$ defines  a continuous function  at $\Omega:=\left\{z\in\mathbb{C}\,: 0<Re(z)<1\right\},$ as well as analytic in its interior.
	
	First, for $z=i\tau,\,\tau\in \mathbb{R},\, F(i\tau)=\theta_{2}J^{si\tau }f,$
that combined with theorem  \ref{zkth1} implies  that $ F(i\tau)\in L^{2}.$

Instead in the case $z=1+i\tau,\,\tau\in \mathbb{R}$  we have that 
$
F(1+i\tau)=\theta_{2}J^{i\tau }J^{s}f,$  then, by lemma \ref{zk19}  we get that 
$
F(1+i\tau)\in L^{2}$.

Finally, by  the three lines lemma  we obtain 
\begin{equation*}
\theta_{2}J^{\tau s}f\in L^{2}(\mathbb{R}^{n}),\quad  \mbox{for any} \quad \tau\in (0,1).
\end{equation*}	
\end{proof}
\begin{rem}
	The Lemma \ref{zk37} represents  a extension  to the $n-$ dimensional case, $n\geq 2$ of a one-dimensional  version  proved in \cite{KLPV}. 
\end{rem}
   In our analysis we  will encounter repeatedly   the operators $J^{s},\,s>0$ and $\partial_{x}^{\alpha},\ \alpha\in (\mathbb{Z}^{+})^{n}$  and it will be  essential for us  to establish  a relationship between them  when we restrict ourselves to  an specific class of subsets of $\mathbb{R}^{n}.$
\begin{lem}[Localization formulas]\label{A}
	Let $f\in L^{2}(\mathbb{R}^{n}).$ Let   $\sigma=(\sigma_{1},\sigma_{2},\dots,\sigma_{n})\in  \mathbb{R}^{n}$ a non-null vector  such that $\sigma_{j}\geq 0,\, j=1,2,\dots,n.$  Let $\epsilon>0,$   we consider  the function $\varphi_{\sigma,\epsilon}\in C^{\infty}(\mathbb{R}^{n})$   to satisfy: $0\leq\varphi_{\sigma,\epsilon} \leq 1,$  
	\begin{equation*}
	\varphi_{\sigma,\epsilon}(x)=
	\begin{cases}
	0\quad \mbox{if}\quad & x\in \mathcal{H}_{\left\{
	\sigma,\frac{\epsilon}{2}\right\}}^{c}\\
	1\quad \mbox{if}\quad & x\in\mathcal{H}_{\{\sigma,\epsilon\}}
	\end{cases}	
	\end{equation*}
	and the following increasing property:
	 for every multi-index $\alpha$ with $|\alpha|=1$
	\begin{equation*} 
	\partial^{\alpha}_{x}\varphi_{\sigma,\epsilon}(x)\geq 0,\quad x\in\mathbb{R}^{n}. 
	\end{equation*}
	\begin{itemize}
		\item[(I)] If  $m\in \mathbb{Z}^{+}$ and $\varphi_{\sigma,\epsilon}J^{m}f\in L^{2}(\mathbb{R}^{n}),$  then for  all $\epsilon'>2\epsilon$  and all multi-index $\alpha$ with   $0 \leq|\alpha|\leq m,$ the derivatives of  $f$ satisfy
		\begin{equation*}
		\varphi_{\sigma,\epsilon'}\partial^{\alpha}_{x}f\in L^{2}(\mathbb{R}^{n}).
		\end{equation*} 
		 
		\item[(II)]If $m\in \mathbb{Z}^{+}$ and $\varphi_{\sigma,\epsilon}\,\partial^{\alpha}_{x}f\in L^{2}(\mathbb{R}^{n})$ for all multi-index $\alpha$ with $ 0\leq |\alpha|\leq m,$  then for all $\epsilon'>2\epsilon$
		\begin{equation*}
		\varphi_{\sigma,\epsilon'}J^{m}f\in L^{2}(\mathbb{R}^{n}).
		\end{equation*}
		\item[(III)] If $s>0,$  and $J^{s}(\varphi_{\sigma,\epsilon}f)\in L^{2}(\mathbb{R}^{n}),$ then for any  $\epsilon'>2\epsilon$
		\begin{equation*}
		\varphi_{\sigma,\epsilon'}\,J^{s}f\in L^{2}(\mathbb{R}^{n}).
		\end{equation*}
		\item[(IV)] If $s>0,$  and  $\varphi_{\sigma,\epsilon}J^{s}f\in L^{2}(\mathbb{R}^{n}),$ then for any $\epsilon'>2\epsilon$
		\begin{equation*}
		J^{s}\left(\varphi_{\sigma,\epsilon'}f\right)\in L^{2}(\mathbb{R}^{n}).
		\end{equation*}
	\end{itemize}
\end{lem}
%The proof of these results requires the following 
%\begin{lem}
%	Suppose  $1<p<\infty,$ and $s\geq 1.$  Then  $f\in L^{p}_{s}(\mathbb{R}^{n})$ if and only if  $f\in L^{p}_{s-1}(\mathbb{R}^{n})$ and for each $\alpha$ multi-index with $|\alpha|=1,$ the functions $\partial^{\alpha}f\in L^{p}_{s-1}(\mathbb{R}^{n}).$ Moreover,
%	\begin{equation*}
%	\|f\|_{L^{p}_{s}} \simeq\|f\|_{L^{p}_{s-1}}+\sum_{|\alpha|=1}\left\|\partial^{\alpha}f\right\|_{L^{p}_{s-1}}.
%	\end{equation*}
%\end{lem}
\begin{proof}[Proof of Lemma \ref{A}]		
		In the following we  will assume  that $f\in \mathcal{S}(\mathbb{R}^{n}),$ and the general case  follows by using a regularization device.
		
%	Since  the lemma \ref{A} condensates several properties into one enunciate, it is plausible for the better understanding of the proof,  to present it here   into  several stages that we pass to describe
	\begin{flushleft}
		{\sc Proof of (i):}
	\end{flushleft}
	 Notice that for $\alpha-$ multi-index with $|\alpha|\leq m,$  the partial derivatives satisfy
	\begin{equation*}
	\begin{split}
	\partial_{x}^{\alpha}f(x)&=\left((2\pi i\xi)^{\alpha}\widehat{f}(\xi)\right)^{\check{}}(x)\\
%	&\textcolor{blue}{=\left(\frac{(2\pi i\xi)^{\alpha}}{\left(1+|\xi|^{2}\right)^{\frac{m}{2}}}\left(\left(1+|\xi|^{2}\right)^{\frac{m}{2}}\widehat{f}(\xi)\right)\right)^{\check{}}(x)}\\
%	&\textcolor{blue}{=\left(\frac{(2\pi i\xi)^{\alpha}}{\left(1+|\xi|^{2}\right)^{\frac{m}{2}}}\widehat{J^{m}f}(\xi)\right)^{\check{}}}\\
	&=\Psi_{m,\alpha}J^{m}f(x),
	\end{split}
	\end{equation*}
	where 
	\begin{equation*}
	\Psi_{m,\alpha}f(x):= \int_{\mathbb{R}^{n}}e^{2\pi i x\cdot\xi}\left(\frac{(2\pi i)^{\alpha}}{\left(1+|\xi|^{2}\right)^{\frac{m}{2}}}\right)\widehat{f}(\xi)\,\mathrm{d}\xi;
	\end{equation*}
whence 	$\Psi_{m,\alpha}\in    \mathrm{OP}\mathbb{S}^{|\alpha|-m}\subset \mathrm{OP}\mathbb{S}^{0}.$ 

	Since, 
	  \begin{equation*}
	\dist\left(\supp\left(\varphi_{\sigma, \epsilon'}\right),\supp\left(1-\varphi_{\sigma,\epsilon}\right)\right)\geq \frac{\epsilon'-2\epsilon}{|\sigma|}>0,
	\end{equation*}
	we rewrite $	\varphi_{\sigma,\epsilon'}\partial_{x}^{\alpha}f$ as follows:
	\begin{equation*}
	\begin{split}
	\varphi_{\sigma,\epsilon'}\partial_{x}^{\alpha}f %&\textcolor{blue}{=\varphi_{\sigma,\epsilon'}\Psi_{m,\alpha}J^{m}f}\\
%	&\textcolor{blue}{=\Psi_{m,\alpha}\left(\varphi_{\sigma, \epsilon'}J^{m}f\right)-\left[\Psi_{m,\alpha}; \varphi_{\sigma, \epsilon'}\right]J^{m}f}\\
%	&\textcolor{blue}{=\Psi_{m,\alpha}\left(\varphi_{\sigma, \epsilon'}J^{m}f\right)-\left[\Psi_{m,\alpha}; \varphi_{\sigma, \epsilon'}\right]\varphi_{\sigma,\epsilon}J^{m}f-\left[\Psi_{m,\alpha}; \varphi_{\sigma, \epsilon'}\right]\left(1-\varphi_{\sigma,\epsilon}\right)J^{m}f}\\
	&=\Psi_{m,\alpha}\left(\varphi_{\sigma, \epsilon'}J^{m}f\right)-\left[\Psi_{m,\alpha}; \varphi_{\sigma, \epsilon'}\right]\varphi_{\sigma,\epsilon}J^{m}f+ \varphi_{\sigma, \epsilon'}\Psi_{m,\alpha}\left(\left(1-\varphi_{\sigma,\epsilon}\right)J^{m}f\right)\\
	&=I+II+III.
	\end{split}
	\end{equation*}
	It is  straightforward to obtain from the hypothesis that
	\begin{equation*}
	\varphi_{\sigma, \epsilon'}J^{m}f\in L^{2}(\mathbb{R}^{n}),\quad \mbox{for all}\quad \epsilon'>2\epsilon.
	\end{equation*}
%	
%	{\color{blue}
%		\begin{description}
%			\item[Justification]
%			Since  
%			\begin{equation}
%			\varphi_{\sigma,\epsilon'}(x)\varphi_{\sigma,\epsilon}=\varphi_{\sigma, \epsilon'}(x),\quad \mbox{for all}\quad x\in\mathbb{R}^{n},
%			\end{equation}	
%			then it is clear that 
%			\begin{equation}
%			\begin{split}
%			\varphi_{\sigma, \epsilon'}J^{m}f&=\varphi_{\sigma, \epsilon'}\varphi_{\sigma,\epsilon}J^{m}f\\
%			&=\varphi_{\sigma,\epsilon'}\left(\varphi_{\sigma,\epsilon}J^{m}f\right).
%			\end{split}
%			\end{equation}
%			Therefore,
%			\begin{equation}
%			\left\|\varphi_{\sigma, \epsilon'}J^{m}f\right\|_{L^{2}}\leq \left\|\varphi_{\sigma, \epsilon'}\right\|_{L^{\infty}}\left\|\varphi_{\sigma,\epsilon}J^{m}f\right\|_{L^{2}}\leq\left\|\varphi_{\sigma,\epsilon}J^{m}f\right\|_{L^{2}}<\infty.
%			\end{equation}
%	\end{description}}
Hence,  in virtue of Theorem \ref{zkth1}
	\begin{equation*}
	\begin{split}
	\left\|I\right\|_{L^{2}}&=\left\|\Psi_{m,\alpha}\left(\varphi_{\sigma, \epsilon'}J^{m}f\right)\right\|_{L^{2}}\lesssim \left\|\varphi_{\sigma, \epsilon'}J^{m}f\right\|_{L^{2}}<\infty.
	\end{split}
	\end{equation*}
	With respect to $II$  we shall remark  that $\left[\Psi_{m,\alpha}; \varphi_{\sigma, \epsilon'}\right]\in \mathrm{OP}\mathbb{S}^{|\alpha|-m-1}\subset\mathrm{OP}\mathbb{S}^{0},$ thus  by Theorem \ref{zkth1} it is clear that 
	\begin{equation*} 
	\left\|II\right\|_{L^{2}}\lesssim \left\|\varphi_{\sigma,\epsilon}J^{m}f\right\|_{L^{2}},
	\end{equation*}
	being the term in the r.h.s bounded by hypothesis.
	
	  To estimate $III$ we notice that after  applying  an argument similar to the one used in the proof of lemma \ref{zk19}  we get   $\|III\|_{L^{2}}<\infty,$  so that, for the sake of brevity we will omit the details.
	  
	  Finally, we gather the estimates above  whence we obtain 
	\begin{equation*}
	\varphi_{\sigma,\epsilon'}\partial^{\alpha}_{x}f\in L^{2}(\mathbb{R}^{n}),\quad \mbox{for all}\quad  \epsilon'>2\epsilon,
	\end{equation*} 
	for all multi-index $\alpha$ satisfying $|\alpha|\leq m.$
\begin{flushleft}
	{\sc Proof of (ii):}
\end{flushleft}
First, notice that  for $f\in \mathcal{S}(\mathbb{R}^{n})$  the operator $J^{m}$ is  defined as   
 \begin{equation*}
 J^{m}f(x)=\int_{\mathbb{R}^{n}}e^{2\pi \mathrm{i}x\cdot\xi} \left(1+ |\xi|^{2}\right)^{\frac{m}{2}}\widehat{f}(\xi)\,\mathrm{d}\xi, \quad x\in \mathbb{R}^{n}.
 	\end{equation*} 
% This operator has the peculiarity that for $m\in 2\mathbb{N}$ it becomes  a local operator, more precisely it involves powers of the  Laplace operator. Instead, for $m\in 2\mathbb{N}+1$ it turns out into a nonlocal operator,  making this a difficult property to handle specially in situations like these where the properties we are studying are locals.
 Since  the symbol  associated to the   operator $J^{m}$ in the Fourier space corresponds  to  $\langle \xi\rangle ^{m},$ it is clear that 
% \begin{equation*}
% \left(1+|\xi|^{2}\right)^{\frac{m}{2}}=\sum_{|\alpha|\leq m} \left(\frac{m!}{\alpha_{1}!\alpha_{2}!\dots\alpha_{n}! (m-|\alpha|)!}\right)\xi^{\alpha} \left(\frac{\xi^{\alpha}}{\left(1+|\xi|^{2}\right)^{\frac{m}{2}}}\right),\qquad  \xi\in \mathbb{R}^{n},
% \end{equation*}  
% in terms of operators it may be represented as 
\begin{equation}\label{z1}
\begin{split}
J^{m}f(x)%&\textcolor{blue}{=\sum_{|\alpha|\leq m} \underbrace{\frac{(-1)^{|\alpha|}m!}{(4\pi^{2})^{|\alpha|}\alpha_{1}!\alpha_{2}!\dots\alpha_{n}! (m-|\alpha|)!}}_{c_{m,\alpha}}\Psi_{m,\alpha}\partial_{x}^{\alpha}f(x)}\\
&=\sum_{|\alpha|\leq m}c_{m,\alpha}\Psi_{m,\alpha}\partial_{x}^{\alpha}f (x),
\end{split}
\end{equation}
where 
\begin{equation*}
\Psi_{m,\alpha}g(x):=\int_{\mathbb{R}^{n}} e^{2\pi\mathrm{i}x\cdot\xi}\left(\frac{(2\pi i\xi)^{\alpha}}{\left(1+|\xi|^{2}\right)^{\frac{m}{2}}}\right) \widehat{g}(\xi)\,\mathrm{d}\xi, \quad g\in \mathcal{S}(\mathbb{R}^{n}).
\end{equation*}
Notice that for any multi-index $\alpha\in (\mathbb{Z}^{+})^{n},$  the condition $|\alpha|-m\leq0$ implies that $\Psi_{m,\alpha}\in\mathrm{OP}\mathbb{S}^{|\alpha|-m}\subset \mathrm{OP}\mathbb{S}^{0}.$ 

So that, in virtue of Theorem \ref{zkth1} it is clear that $\Psi_{m,\alpha}:L^{2}(\mathbb{R}^{n})\longrightarrow L^{2}(\mathbb{R}^{n}).$

Now, we combine \eqref{z1} with the hypothesis whence we obtain 
for $\epsilon'>2\epsilon,$
\begin{equation}\label{z3}
\begin{split}
\varphi_{\sigma, \epsilon'}J^{m}f %&\textcolor{blue}{=\sum_{|\alpha|\leq m} c_{\alpha,m}\varphi_{\sigma, \epsilon'}\Psi_{m,\alpha}\partial_{x}^{\alpha}f}\\
%&\textcolor{blue}{=\sum_{|\alpha|\leq m} c_{\alpha,m}\Psi_{m,\alpha}(\varphi_{\sigma, \epsilon'}\partial_{x}^{\alpha}f)-\sum_{|\alpha|\leq m} c_{\alpha,m}\left[\Psi_{m,\alpha};\varphi_{\sigma, \epsilon'}\right]\partial_{x}^{\alpha}f}\\
&=\sum_{|\alpha|\leq m} c_{\alpha,m}\Psi_{m,\alpha}(\varphi_{\sigma, \epsilon'}\partial_{x}^{\alpha}f)-\sum_{|\alpha|\leq m} c_{\alpha,m}\left[\Psi_{m,\alpha};\varphi_{\sigma, \epsilon'}\right]\varphi_{\sigma,\epsilon}\partial_{x}^{\alpha}f\\
&\quad +\sum_{|\alpha|\leq m} c_{\alpha,m}\varphi_{\sigma, \epsilon'}\Psi_{m,\alpha}\left(\left(1-\varphi_{\sigma,\epsilon}\right)\partial_{x}^{\alpha}f\right).
\end{split}
\end{equation}
By hypothesis $\varphi_{\sigma, \epsilon'}\partial_{x}^{\alpha}f\in L^{2}(\mathbb{R}^{n}),$   then by $L^{2}-$continuity of  $\Psi_{m,\alpha}$    we get 
 \begin{equation*}
\left\|\Psi_{m,\alpha}\left(\varphi_{\sigma, \epsilon'}\partial_{x}^{\alpha}f\right)\right\|_{L^{2}}\lesssim \left\|\varphi_{\sigma, \epsilon'}\partial_{x}^{\alpha}f\right\|_{L^{2}}<\infty.
\end{equation*}
%Instead, as regard  the commutator expression in the second term in the r.h.s of  \eqref{z3},
Since   $\left[\Psi_{m,\alpha}; \varphi_{\sigma, \epsilon'}\right]\in \mathrm{OP}\mathbb{S}^{|\alpha|-m-1}\subset \mathrm{OP}\mathbb{S}^{0},$ then  by theorem \ref{zkth1}  it follows that 
\begin{equation*}
\left\|\left[\Psi_{m,\alpha}; \varphi_{\sigma, \epsilon'}\right]\varphi_{\sigma, \epsilon}\partial_{x}^{\alpha}f\right\|_{L^{2}}\lesssim \left\|\varphi_{\sigma, \epsilon}\partial_{x}^{\alpha}f\right\|_{L^{2}}<\infty,
\end{equation*} 
being this last inequality  a consequence of our hypothesis.

Finally, we focus our attention in  to bound the third term in the r.hs of \eqref{z3}, for that  we will take hand of the  following fact
\begin{equation}\label{z5}
\dist\left(\supp\left(\varphi_{\sigma, \epsilon'}\right),\supp\left(1-\varphi_{\sigma,\epsilon}\right)\right)\geq \frac{\epsilon'-2\epsilon}{|\sigma|}>0.
\end{equation}
Hence, an argument similar to the one used in the proof of lemma \ref{zk19} allow us to obtain that $\varphi_{\sigma, \epsilon'}\Psi_{m,\alpha}\left(\left(1-\varphi_{\sigma,\epsilon}\right)\partial_{x}^{\alpha}f\right)\in L^{2}(\mathbb{R}^{n}).$ 

In summary we have proved that for all $\epsilon'>2\epsilon,$
\begin{equation*}
\varphi_{\sigma, \epsilon'}J^{m}f\in L^{2}(\mathbb{R}^{n}).
\end{equation*}
\begin{flushleft}
	{\sc Proof of (iii):}
	\end{flushleft}
	Since   $J^{s}(\varphi_{\sigma,\epsilon}f)\in L^{2}(\mathbb{R}^{n}),$  it is clear that 
	\begin{equation}\label{eq1.1}
	\begin{split}
	\varphi_{\sigma,\epsilon'}J^{s}f&=\varphi_{\sigma,\epsilon'}J^{s}\left(\varphi_{\sigma,\epsilon}f+(1-\varphi_{\sigma,\epsilon})f\right)\\
	&=\varphi_{\sigma,\epsilon'}J^{s}\left(\varphi_{\sigma,\epsilon}f\right)+\varphi_{\sigma,\epsilon'}J^{s}\left((1-\varphi_{\sigma,\epsilon})f\right).\\
	\end{split}
	\end{equation}
So that, to handle the remainder term in (\ref{eq1.1}) it is sufficiently to consider  the relation between the supports of the function involved. More precisely,
\begin{equation*}
\supp(\varphi_{\sigma,\epsilon'})\subset \mathcal{H}_{\left\{\sigma,\frac{\epsilon'}{2}\right\}}\quad \mbox{and}\quad \supp(f(1-\varphi_{\sigma,\epsilon}))\subset \mathcal{H}_{\{\sigma,\epsilon\}}^{c},
\end{equation*}
which  implies 
\begin{equation*}
\dist\left(\supp\left(\varphi_{\sigma,\epsilon'}\right), \supp\left(1-\varphi_{\sigma,\epsilon}\right)\right)\geq\frac{\epsilon'-2\epsilon}{2|\sigma|}>0
\end{equation*}
Thus, by means of Lemma \ref{lem1} we obtain that
\begin{equation*}
\begin{split}
 \left\|\varphi_{\sigma,\epsilon'}J^{s}\left((1-\varphi_{\sigma,\epsilon})f\right)\right\|_{L^{2}}&\lesssim \left\|\varphi_{\sigma,\epsilon'}\right\|_{L^{\infty}}\left\|\left(1-\varphi_{\sigma,\epsilon}\right)f\right\|_{L^{2}}\lesssim \|f\|_{L^{2}}.
 \end{split}
\end{equation*}
  Finally, gathering the bounds above we get 
  \begin{equation*}
  \varphi_{\sigma,\epsilon'}J^{s}f,\quad \mbox{ for all}\quad  \epsilon'>2\epsilon, 
  \end{equation*}
  which finish the proof of {\sc (iii)}.
\begin{flushleft}
	{\sc Proof of (iv):}
	\end{flushleft}
Without loss of generality we can assume that  $s\in [m,m+1)$ where $m\in \mathbb{N}_{0}.$% and  $f\in \mathcal{S}(\mathbb{R}^{n}).$
 
First, we rewrite $J^{s}(\varphi_{\sigma, \epsilon'}f)$ as  
\begin{equation}\label{e7}
\begin{split}
J^{s}(\varphi_{\sigma, \epsilon'}f)&=\varphi_{\sigma, \epsilon'}J^{s}f+ \widetilde{\Psi_{\zeta_{s,\sigma,\epsilon'}}}f\\
\end{split}
\end{equation}
where  ${\displaystyle \widetilde{\Psi_{\zeta_{s,\sigma,\epsilon'}}}:=\left[J^{s}; \varphi_{\sigma, \epsilon'}\right] }.$

The decomposition above shows that    in order to obtain a bound for $ J^{s}(\varphi_{\sigma, \epsilon'}f)$ we   require  to estimate the commutator term. % however 
%this task presents several drawbacks. %that are worth mentioning. For example, in  the one-dimensional case, the  Kato-Ponce commutator  estimate \ref{zk2}  allows  to  establish a recurrent process to estimate the commutator for higher order derivatives, of course all this method  depends strongly on the fact that  $\varphi_{\sigma, \epsilon'}$ have $s$ derivatives, $s\geq 1$ compactly supported, which makes easy to control $\|\varphi_{\sigma,\epsilon'}\|_{H^{s}(\mathbb{R})}.$  
% Instead for higher dimensions this is no longer the case, mainly because   $\varphi_{\sigma,\epsilon'}$ and its derivatives are not  bounded in $L^{2},$ but only in $L^{\infty}(\mathbb{R}^{n}).$   So that, the method used for the one-dimensional case does not adapt to the higher dimensional case and it is required a new approach to face this problem.
 
 %To face these issues we present a new way of dealing with these types of problems regardless the dimension of the space.

In the following  we will decompose the operator $\widetilde{\Psi_{\zeta_{s,\sigma,\epsilon'}}}$ whence  
 \begin{equation*}
 \widetilde{\Psi_{\zeta_{s,\sigma,\epsilon'}}}g(x)=\int_{\mathbb{R}^{n}}e^{2\pi ix\cdot \xi} \zeta_{s,\sigma,\epsilon'}(x,\xi)\, \widehat{g}(\xi)\,\mathrm{d}\xi,\qquad g\in \mathcal{S}(\mathbb{R}^{n}).
 \end{equation*}
 Next,  by   proposition \ref{prop1} we get 
  \begin{equation*}\label{}
\begin{split}
&\zeta_{s,\sigma,\epsilon'}(x,\xi)\\
&=\sum_{1\leq|\alpha|\leq m}\frac{(2\pi i)^{-|\alpha|}}{\alpha!}\left\{\partial^{\alpha}_{\xi}\left(\langle \xi\rangle^{s}\right)\partial_{x}^{\alpha}\varphi_{\sigma, \epsilon'}\right\}+\kappa_{s-m-1}(x,\xi)\\
&=\sum_{j=1}^{m}\sum_{|\alpha|= j}\frac{(2\pi i)^{-|\alpha|}}{\alpha!}\left\{\partial^{\alpha}_{\xi}\left(\langle \xi\rangle^{s}\right)\partial_{x}^{\alpha}\varphi_{\sigma, \epsilon'}\right\}+\kappa_{s-m-1}(x,\xi)\\
&=-\frac{s}{4\pi^{2}}\sum_{|\alpha|=1}(2\pi i\xi)^{\alpha}\langle\xi\rangle^{s-2}\partial_{x}^{\alpha}\varphi_{\sigma, \epsilon'}-\frac{s}{4\pi^{2}}\sum_{ \mathclap{\substack{|\alpha|=2\\\alpha=\alpha_{1}+\alpha_{2}\\|\alpha_{1}|=|\alpha_{2}|=1}}}\delta_{\alpha_{1},\alpha_{2}}\frac{1}{\alpha!}\partial_{x}^{\alpha}\varphi_{\sigma, \epsilon'}\langle \xi\rangle^{s-2}\\
&\quad +\frac{s(s-2)}{(2\pi)^{4}}\sum_{|\alpha|=2}\frac{1}{\alpha!}\partial_{x}^{\alpha}\varphi_{\sigma,\epsilon'}\left(2\pi i\xi\right)^{\alpha}\langle \xi \rangle^{s-4}\\
&\quad +
 \frac{s(s-2)}{(2\pi)^{4}}\sum_{|\alpha|=2}\frac{1}{\alpha!}\partial_{x}^{\alpha}\varphi_{\sigma, \epsilon'}(x)\left(2\pi i\xi\right)^{\alpha}\langle \xi \rangle^{s-4}\\
%&\, +\frac{s(s-2)}{16\pi^{4}}\sum_{\mathclap{\substack{|\alpha|=3\\\alpha=\alpha_{1}+\alpha_{2}+\alpha_{3}\\|\alpha_{1}|=\cdots=|\alpha_{3}|=1}}}\frac{1}{\alpha!}\left\{\left(\delta_{\alpha_{1},\alpha_{2}}(2\pi i \xi )^{\alpha_{3}}+\delta_{\alpha_{3},\alpha_{1}}(2\pi i\xi)^{\alpha_{2}}+\delta_{\alpha_{3},\alpha_{2}}(2\pi i\xi)^{\alpha_{1}}\right)\langle \xi \rangle^{s-4}\right\}\partial_{x}^{\alpha}\varphi_{\sigma, \epsilon'}(x)\\
&%\, +\frac{s(s-2)(s-4)}{16\pi^{4}}\sum_{\mathclap{\substack{|\alpha|=3\\\alpha=\alpha_{1}+\alpha_{2}+\alpha_{3}\\|\alpha_{1}|=\cdots=|\alpha_{3}|=1}}}\frac{1}{\alpha!}\left((2\pi i\xi )^{\alpha_{3}}\langle \xi \rangle ^{s-6}\right)\partial_{x}^{\alpha}\varphi_{\sigma, \epsilon'}(x)
\quad + \cdots +\sum_{|\alpha|= m}\frac{(2\pi i)^{-|\alpha|}}{\alpha!}\left\{\partial^{\alpha}_{\xi}\left(\langle \xi\rangle^{s}\right)\partial_{x}^{\alpha}\varphi_{\sigma, \epsilon'}\right\}+\kappa_{s-m-1}(x,\xi),
\end{split}
\end{equation*}
where $\kappa_{s-m-1}\in \mathbb{S}^{m+1-s}\subset \mathbb{S}^{0},$ and for  this symbol we consider    the operator 
\begin{equation*}
\Theta_{\kappa_{s-m-1}}g(x):= \int_{\mathbb{R}^{n}}e^{2\pi i x\cdot \xi} \kappa_{s-m-1}(x,\xi)\, \widehat{g}(\xi)\,\mathrm{d}\xi,\qquad g\in \mathcal{S}(\mathbb{R}^{n}).
\end{equation*}
In view that $\kappa_{s-m-1}\in \mathbb{S}^{0},$ then by  Theorem \ref{zkth1}  is clear that 
$
\left\|\Theta_{\kappa_{s-m-1}}f\right\|_{L^{2}}\lesssim \|f\|_{L^{2}}.$
Also, for multi-index $\alpha,\beta $ with $\beta\leq \alpha$ we define  the symbol 
\begin{equation*}
\eta_{\alpha,\beta}(x,\xi)=\frac{(2\pi i\xi)^{\beta}}{\left(1+|\xi|^{2}\right)^{\frac{|\alpha|}{2}}},\qquad x,\xi\in\mathbb{R}^{n}.
\end{equation*}
Then, it is clear that   $\Psi_{\eta_{\alpha,\beta}}\in\mathrm{OP}\mathbb{S}^{0},$ whence
\begin{equation}\label{e11}
\Psi_{\eta_{\alpha,\beta}}g(x):=\int_{\mathbb{R}^{n}}e^{2\pi ix\cdot \xi}\eta_{\alpha,\beta}(x,\xi)\widehat{g}(\xi)\,\mathrm{d}\xi,\quad g\in \mathcal{S}(\mathbb{R}^{n});
\end{equation}
with this notation at hand,  we have 
\begin{equation}\label{e8}
\begin{split}
 \widetilde{\Psi_{\zeta_{s,\sigma,\epsilon'}}}f(x)=\sum_{j=1}^{m}\sum_{|\alpha|=j}\sum_{\beta\leq\alpha} c_{\alpha,\beta} \,\partial_{x}^{\alpha}\varphi_{\sigma,\epsilon'}(x)\Psi_{\eta_{\alpha,\beta}}J^{s-|\alpha|}f(x)+ \Theta_{\kappa_{s-m-1}}f(x).
\end{split}
\end{equation}
% Reaching  this  decomposition  requires to describe explicitly  a formula for   higher order  derivatives, that  can be obtained by using F\`{a}a di Bruno's formula (see \cite{Faa}, Theorem 2.1),  nevertheless  for the sake of brevity we do not present it here. 
Notice that from \eqref{e8} the problem becomes  reduced  to   localize   $J^{s-|\alpha|}f,$ for every $\alpha\in (\mathbb{N}_{0})^{n}.$  In this sense,  we claim that there exist  a smooth  function $\theta_{\epsilon,\epsilon'}$  such that:
\begin{equation}\label{e1}
\supp\left(\theta_{\epsilon,\epsilon'}\right)\subset \mathcal{H}_{\left\{\sigma, \frac{\epsilon'+14\epsilon}{16}\right\}},
\end{equation}
as well as 
\begin{equation}\label{e2}
0\leq \theta_{\epsilon,\epsilon'} \leq 1 \qquad \mbox{and}\quad  \theta_{\epsilon,\epsilon'}\equiv 1\quad \mbox{on}\quad  \mathcal{H}_{\left\{\sigma, \frac{7\epsilon'+2\epsilon}{16}\right\}}.
\end{equation}
The interested reader can verify that it is enough  to consider  $\rho\in C^{\infty}_{0}(\mathbb{R}),\, \rho\geq 0,$  even,  with  $\supp(\rho)\subseteq (-1,1)$ and $\|\rho\|_{L^{1}}=1.$  Then, by  defining

%\begin{equation*}
%\theta(x):=\varphi_{\sigma, \epsilon'+2\epsilon}\left(4x+\left(\frac{\epsilon'-2\epsilon}{2|\sigma|^{2}}\right)\sigma\right),\quad \mbox{for all} \quad x\in \mathbb{R}^{n}
%\end{equation*}
	\begin{equation*}
\nu_{\epsilon,\epsilon'}(y)=
\begin{cases}
0\quad \mbox{if}\quad &  y\leq \frac{\epsilon'+6\epsilon}{8}\\
4\left(\frac{8 y-\epsilon'-6\epsilon}{\epsilon'-2\epsilon}\right)\quad \mbox{if}\quad & \frac{\epsilon'+6\epsilon}{8}< y<\frac{3\epsilon'+2\epsilon}{8}\\
1\quad\mbox{if}&  y\geq\frac{3\epsilon'+2\epsilon}{8},
\end{cases}	
\end{equation*}
and    $\rho_{\epsilon,\epsilon'}(y):=\frac{16}{\epsilon'-2\epsilon}\rho\left(\frac{16y}{\epsilon'-2\epsilon}\right)$ for all $y\in\mathbb{R}$ and $\epsilon'>2\epsilon,$  
we see that 
\begin{equation*}
\theta_{\epsilon,\epsilon'}(x)=\left(\rho_{\epsilon,\epsilon'}*\nu_{\epsilon,\epsilon'}\right)(\sigma\cdot x),\, x\in \mathbb{R}^{n}
\end{equation*}  
satisfy the  claimed properties.

Hence,
\begin{equation*}
\begin{split}
&\partial_{x}^{\alpha}\varphi_{\sigma, \epsilon'} \Psi_{\eta_{\alpha,\beta}}J^{s-|\alpha|}f\\
&=-\left[\Psi_{\eta_{\alpha,\beta}};\partial_{x}^{\alpha}\varphi_{\sigma, \epsilon'}\right]\theta_{\epsilon,\epsilon'
} J^{s-|\alpha|}f+\partial_{x}^{\alpha}\varphi_{\sigma, \epsilon'}\Psi_{\eta_{\alpha,\beta}}\left(\left(1-\theta_{\epsilon,\epsilon'}
\right)J^{s-|\alpha|}f\right)\\
&\quad +\Psi_{\eta_{\alpha,\beta}}\left( \partial_{x}^{\alpha}\varphi_{\sigma, \epsilon'}J^{s-|\alpha|}f\right)\\
&=I+II+III.
\end{split}
\end{equation*}
The first and the third term in the r.h.s above are  bounded in the $L^{2}-$norm. More precisely,  
\begin{equation*}
\|I\|_{L^{2}}=\left\|\left[\Psi_{\eta_{\alpha,\beta}};\partial_{x}^{\alpha}\varphi_{\sigma, \epsilon'}\right]\theta_{\epsilon,\epsilon'} J^{s-|\alpha|}f\right\|_{L^{2}}\lesssim \left\| \theta_{\epsilon,\epsilon'}J^{s-|\alpha|}f\right\|_{L^{2}}<\infty,
\end{equation*}
being the last inequality a consequence of  combining  lemma \ref{zk37} and Theorem \ref{zkth1}.

The third term  is handled  by using  that  $\Psi_{\eta_{\alpha,\beta}}\in \mathrm{OP}\mathbb{S}^{0},$ so that, combining  lemma \ref{zk37}, \eqref{e1}-\eqref{e2} and Theorem \ref{zkth1} we obtain 
\begin{equation*}
\begin{split}
\|III\|_{L^{2}}&=\left\|\Psi_{\eta_{\alpha,\beta}}\left( \partial_{x}^{\alpha}\varphi_{\sigma, \epsilon'}J^{s-|\alpha|}f\right)\right\|_{L^{2}}
%&\textcolor{blue}{=\left\|\Psi_{\eta_{\alpha,\beta}}\left( \partial_{x}^{\alpha}\varphi_{\sigma, \epsilon'}\theta J^{s-|\alpha|}f\right)\right\|_{L^{2}}}\\
\lesssim \left\| \partial_{x}^{\alpha}\varphi_{\sigma, \epsilon'}\right\|_{L^{\infty}}\left\| J^{s-|\alpha|}f\right\|_{L^{2}\left(\mathcal{H}_{\{\sigma,\epsilon\}}\right)} < \infty.
\end{split}
\end{equation*}
Notice that by construction, for any multi-index  $\alpha$ with $|\alpha|\geq 1$ the following relationship  holds
\begin{equation}\label{e4}
\dist\left(\supp\left(\left(1-\theta_{\epsilon,\epsilon'
}\right)\right), \supp\left(\partial_{x}^{\alpha}\varphi_{\sigma, \epsilon'}\right)\right)\geq \underbrace{\frac{(\epsilon'-2\epsilon)}{16|\sigma|}}_{=:\mu(\epsilon)}>0;
\end{equation}
this property will so that  by using a similar argument to the used in the proof of lemma \ref{zk19}  it is possible to prove    

For any multi-index $\alpha$ with $|\alpha|\leq s,$ we set $g_{s,\alpha}:=J^{s-|\alpha|}f,$ %We shall remark that by hypothesis    $f\in L^{2}(\mathbb{R}^{n}),$ so that by definition it is clear that 
then, it is clear that   $g_{s,\alpha}\in H^{-(s-|\alpha|)}(\mathbb{R}^{n}).$

Thus, combining \eqref{e4} and  the Fourier transform properties we obtain 
\begin{equation*}
\begin{split}
&\partial_{x}^{\alpha}\varphi_{\sigma, \epsilon'}(x)\Psi_{\eta_{\alpha,\beta}}\left(\left(1-\theta_{\epsilon,\epsilon'
}\right)g_{s,\alpha}\right)(x)\\
&=\int_{\mathbb{R}^{n}}e^{2\pi ix\cdot \xi}\partial_{x}^{\alpha}\varphi_{\sigma, \epsilon'}(x) \eta_{\alpha,\beta}(x,\xi)\widehat{\left(\left(1-\theta_{\epsilon,\epsilon'}\right)g_{s,\alpha}\right)}(\xi)\,\mathrm{d}\xi\\
&=\int_{B_{x}(\mu(\epsilon))^{c}}\partial_{x}^{\alpha}\varphi_{\sigma, \epsilon'}(x)k_{\alpha,\beta}(x,x-y)\left(1-\theta_{\epsilon,\epsilon'}(y)\right) J^{2q}J^{-2q}g_{s,\alpha}(y)\,\mathrm{d}y,
\end{split}
\end{equation*}
where $q$ is chosen in such a way that $q\in \mathbb{Z}^{+}$ and
\begin{equation}\label{e3}
q\geq \max\left\{1,\ceil[\Big]{\frac{s-|\alpha|}{2}}\right\},
\end{equation}
thus, integration by parts yield
\begin{equation*}
\begin{split}
&\int_{B_{x}(\mu(\epsilon))^{c}}\partial_{x}^{\alpha}\varphi_{\sigma, \epsilon'}(x)k_{\alpha,\beta}(x,x-y)\left(1-\theta_{\epsilon,\epsilon'}(y)\right) J^{2q}J^{-2q}g_{s,\alpha}(y)\,\mathrm{d}y\\
%&{\textcolor{blue}{=\int_{\left\{|x-y|\geq \mu(\epsilon)\right\}}\partial_{x}^{\alpha}\varphi_{\sigma, \epsilon'}(x)J^{2q}\left(k(x,x-y)(1-\theta(y))\right) J^{-2q}g_{s,\alpha}(y)\,\mathrm{d}y}}\\
%&\textcolor{blue}{=\sum_{j=1}^{2q}c_{q,j}\int_{\left\{|x-y|\geq \mu(\epsilon)\right\}}\partial_{x}^{\alpha}\varphi_{\sigma, \epsilon'}(x)\left(-\Delta\right)^{j}\left(\left(k(x,x-y)(1-\theta(y))\right)\right) J^{-2q}g_{s,\alpha}(y)\,\mathrm{d}y}\\
%&\textcolor{blue}{=\sum_{j=1}^{2q}\sum_{|\beta|\leq 2j }c_{\beta,j}c_{q,j}\int_{\left\{|x-y|\geq \mu(\epsilon)\right\}}\partial_{x}^{\alpha}\varphi_{\sigma, \epsilon'}(x)\partial_{y}^{\beta}\left(\left(k(x,x-y)(1-\theta(y))\right)\right) J^{-2q}g_{s,\alpha}(y)\,\mathrm{d}y}\\
&=%\sum_{j=1}^{2q}\sum_{|\beta|\leq 2j }\sum_{0\leq\gamma\leq \beta}
\sum_{j,\beta,\gamma}c_{\beta,\gamma,q,j} \partial_{x}^{\alpha}\varphi_{\sigma, \epsilon'}(x)\left( k_{\alpha,\beta,\gamma,\mu(\epsilon)}(x,\cdot)\ast\left(\left(\partial_{y}^{\beta-\gamma}(1-\theta_{\epsilon,\epsilon'})\right) J^{-2q}g_{s,\alpha}\right)\right)(x),
\end{split}
\end{equation*}
where ${\displaystyle  k_{\alpha,\beta,\gamma,\mu(\epsilon)}(x,\cdot):=\left(\mathbb{1}_{\left\{|\cdot|\geq \mu(\epsilon)\right\}}\partial_{y}^{\gamma}k_{\alpha,\beta}(x,\cdot)\right)}.$

Nevertheless, at this point two cases have to be  distinguished.
\begin{flushleft}
	\textbf{Case : $\gamma=\beta$}
\end{flushleft}
For this case we combine  Young's inequality  and \eqref{e3} to obtain
\begin{equation}\label{e5}
\begin{split}
\left\|\partial_{x}^{\alpha}\varphi_{\sigma, \epsilon'}\Psi_{\eta_{\alpha,\beta}}\left(\left(1-\theta_{\epsilon,\epsilon'}\right)g_{s,\alpha}\right)\right\|_{L^{2}}
%&\textcolor{blue}{\lesssim \sum_{j=1}^{2q}\sum_{|\beta|\leq 2j }c_{\beta,\beta,q,j}\left\|\partial_{x}^{\alpha}\varphi_{\sigma, \epsilon'}\right\|_{L^{\infty}}\left\|\frac{\mathbb{1}_{\left\{|y|\geq \mu(\epsilon)\right\}}}{|y|^{n+|\beta|+N}}*\left(1-\theta\right)J^{-2q}g_{s,\alpha} \right\|_{L^{2}}}\\
&\lesssim_{\mu(\epsilon)}\sum_{j=1}^{2q}\sum_{|\beta|\leq 2j }\widetilde{c_{\beta,\beta,q,j,n,N}}\left\|\partial_{x}^{\alpha}\varphi_{\sigma, \epsilon'}\right\|_{L^{\infty}}\|f\|_{L^{2}}\\
&<\infty.
\end{split}
\end{equation}
\begin{flushleft}
	\textbf{Case : $\gamma\neq \beta$}
\end{flushleft}
Combining  \eqref{e3}, Young's inequality and theorem  \ref{zkth1}  imply that 
\begin{equation}\label{e6}
\begin{split}
&\left\|\partial_{x}^{\alpha}\varphi_{\sigma, \epsilon'}\Psi_{\eta_{\alpha,\beta}}\left(\left(1-\theta_{\epsilon,\epsilon'}\right)g_{s,\alpha}\right)\right\|_{L^{2}}\\
%&\textcolor{blue}{\lesssim\sum_{j=1}^{2q}\sum_{|\beta|\leq 2j }\sum_{0\leq\gamma\leq \beta}c_{\beta,\gamma,q,j}\left\| \partial_{x}^{\alpha}\varphi_{\sigma, \epsilon'}\right\|_{L^{\infty}}\left\|\left(\mathbb{1}_{\left\{|\cdot|\geq \mu(\epsilon)\right\}}\partial_{y}^{\gamma}k(x,\cdot)\right)*\left(\left(\partial_{y}^{\beta-\gamma}(1-\theta)\right) J^{-2q}g_{s,\alpha}\right)\right\|_{L^{2}}}\\
%&\textcolor{blue}{\lesssim\sum_{j=1}^{2q}\sum_{|\beta|\leq 2j }\sum_{0\leq\gamma\leq \beta}c_{\beta,\gamma,q,j}\left\| \partial_{x}^{\alpha}\varphi_{\sigma, \epsilon'}\right\|_{L^{\infty}}\left\|\mathbb{1}_{\left\{|\cdot|\geq \mu(\epsilon)\right\}}\partial_{y}^{\gamma}k(x,\cdot)\right\|_{L^{1}}\left\|\left(\partial_{y}^{\beta-\gamma}(1-\theta)\right) J^{-2q}g_{s,\alpha}\right\|_{L^{2}}}\\
&\lesssim_{\mu(\epsilon)}\sum_{j=1}^{2q}\sum_{|\beta|\leq 2j }\sum_{0\leq\gamma\leq \beta}\widetilde{c_{\beta,\gamma,q,j,n,N}}\left\| \partial_{x}^{\alpha}\varphi_{\sigma, \epsilon'}\right\|_{L^{\infty}}\left\|\partial_{y}^{\beta-\gamma}(1-\theta_{\epsilon,\epsilon'})\right\|_{L^{\infty}}\left\| f\right\|_{L^{2}}.
\end{split}
\end{equation}
Gathering the estimates in \eqref{e5} and \eqref{e6} we   get
\begin{equation*}
\left\|II\right\|_{L^{2}}\lesssim_{\mu(\epsilon),n,N,s}\left\| \partial_{x}^{\alpha}\varphi_{\sigma, \epsilon'}\right\|_{L^{\infty}}\left\| f\right\|_{L^{2}}.
\end{equation*}
We shall remark that we are looking for  an upper  bound in the $L^{2}-$norm  of   $\widetilde{\Psi_{\zeta_{s,\sigma,\epsilon'}}}f.$ In this sense, we get  after going   back into \eqref{e8}  the following:
\begin{equation*}
\begin{split}
&\left\| \widetilde{\Psi_{\zeta_{s,\sigma,\epsilon'}}}f\right\|_{L^{2}}\\
&\leq \left\|\sum_{j=1}^{m}\sum_{|\alpha|=j}\sum_{\beta\leq\alpha} c_{\alpha,\beta}\partial_{x}^{\alpha}\varphi_{\sigma,\epsilon'}(x)\Psi_{\eta_{\alpha,\beta}}J^{s-|\alpha|}f\right\|_{L^{2}}+\left\| \Theta_{\kappa_{s-m-1}}f\right\|_{L^{2}}\\
%&\textcolor{blue}{\leq\sum_{j=1}^{m}\sum_{|\alpha|=j}\sum_{\beta\leq\alpha} |c_{\alpha,\beta}|\left\|\partial_{x}^{\alpha}\varphi_{\sigma,\epsilon'}(x)\Psi_{\eta_{\alpha,\beta}}J^{s-|\alpha|}f\right\|_{L^{2}}+\left\| \Theta_{\kappa_{s-m-1}}f\right\|_{L^{2}}}\\
&\lesssim_{\mu(\epsilon),n,N,s} \sum_{j=1}^{m}\sum_{|\alpha|=j}\sum_{\beta\leq\alpha} |c_{\alpha,\beta}|\left(\left\| \theta_{\epsilon,\epsilon'}J^{s-|\alpha|}f\right\|_{L^{2}}+\|f\|_{L^{2}}\right)\left(\left\| \partial_{x}^{\alpha}\varphi_{\sigma, \epsilon'}\right\|_{L^{\infty}}+1\right)\\
&\quad +\|f\|_{L^{2}}\\
&<\infty,
\end{split}
\end{equation*}
then, 
\begin{equation*}
\begin{split}
&\left\|J^{s}(\varphi_{\sigma,\epsilon'}f)\right\|_{L^{2}}\\
&\lesssim_{\mu(\epsilon),n,N,s} \left\|\varphi_{\sigma,\epsilon'}J^{s}f\right\|_{L^{2}}+\|f\|_{L^{2}}\\
&\quad +\sum_{j=1}^{m}\sum_{|\alpha|=j}\sum_{\beta\leq\alpha} |c_{\alpha,\beta}|\left(\left\| \theta_{\epsilon,\epsilon'}J^{s-|\alpha|}f\right\|_{L^{2}}+\|f\|_{L^{2}}\right)\left(\left\| \partial_{x}^{\alpha}\varphi_{\sigma, \epsilon'}\right\|_{L^{\infty}}+1\right)\\
&<\infty.
\end{split}
\end{equation*}
So that, we have proved that $\varphi_{\sigma,\epsilon}J^{s}f\in L^{2}$ implies that 
\begin{equation*}
J^{s}(\varphi_{\sigma, \epsilon'}f)\in L^{2}(\mathbb{R}^{n}), \quad \mbox{for all}\quad \epsilon'>2\epsilon.
\end{equation*}
\end{proof}
\begin{rem}
	At this point several  comments about this lemma should be presented due to its relevance.
	\begin{itemize}
		\item [(i)]To our knowledge the first version of this "localization formulas" were presented in \cite{KLPV} in the study of propagation of regularity in solutions of the $k-$generalized  KdV  equation. % and therefore these one's are  restricted to the one-dimensional case. %Nevertheless, in \cite{KLPV} is not presented a proof  of such results and these one's are considered as a remark from a previous lemma, for more details see \cite{KLPV}.
		\item[(ii)] Also 	we shall emphasize that the lemma  above  also holds when we consider a  " direction "  $\sigma$  with  different conditions  to the emphasized above. As we shall see later a more general version of this lemma also holds  when we consider different   regions of the space instead of just merely half-spaces.
	\end{itemize}
\end{rem}
\begin{rem}
	A more general version of this lemma also hold  when we consider  a wider class of domains where is required to localize the regularity. More precisely, the following  localization result is also  true.
\end{rem}
\begin{lem}\label{lemm}
	Let $f\in L^{2}(\mathbb{R}^{n}).$ If   $\theta_{1}, \theta_{2}\in C^{\infty}(\mathbb{R}^{n})$  are functions such that:  $0\leq \theta_{1},\theta_{2}\leq 1,$   their  respective supports satisfy
	\begin{equation*}
	\dist\left(\supp\left(1-\theta_{1}\right), \supp\left(\theta_{2}\right)\right)\geq \delta,
	\end{equation*}
	for some positive number  $\delta,$ and for all multi-index $\beta,$   the functions $\partial_{x}^{\beta}\theta_{1},\partial_{x}^{\beta}\theta_{2}\in L^{\infty}(\mathbb{R}^{n}).$
	
	Then, the following identity  holds:
	\begin{itemize}
		\item[(I)] If  $m\in \mathbb{Z}^{+}$ and $\theta_{1}J^{m}f\in L^{2}(\mathbb{R}^{n}),$  then for  all multi-index $\alpha$ with   $0 \leq|\alpha|\leq m,$ the derivatives of  $f$ satisfy
		\begin{equation*}
		\theta_{2}\partial^{\alpha}_{x}f\in L^{2}(\mathbb{R}^{n}).
		\end{equation*} 
		
		\item[(II)]If $m\in \mathbb{Z}^{+}$ and $\theta_{1}\,\partial^{\alpha}_{x}f\in L^{2}(\mathbb{R}^{n})$ for all multi-index $\alpha$ with $ 0\leq |\alpha|\leq m,$  then 
		\begin{equation*}
		\theta_{2}J^{m}f\in L^{2}(\mathbb{R}^{n}).
		\end{equation*}
		\item[(III)] If $s>0,$  and $J^{s}(\theta_{1}f)\in L^{2}(\mathbb{R}^{n}),$ then 
		\begin{equation*}
		\theta_{2}\,J^{s}f\in L^{2}(\mathbb{R}^{n}).
		\end{equation*}
		\item[(IV)] If $s>0,$  and  $\theta_{1}J^{s}f\in L^{2}(\mathbb{R}^{n}),$ then 
		\begin{equation*}
		J^{s}\left(\theta_{2}f\right)\in L^{2}(\mathbb{R}^{n}).
		\end{equation*}
	\end{itemize}
\end{lem}
\begin{proof}
	The proof  follows by using an argument similar to the one used in the proof of lemma \ref{A}, so that for the sake of brevity we  omit the details.
\end{proof}

\section{Kato's smoothing effect}

In the following section we present he Kato's smoothing effect version satisfied by the solutions of the ZK equation in the $n-$ dimensional setting, $n\geq 2.$  By using the Kato's approach  for the KdV model \cite{KATO1}, we are able to prove that the solutions associated to the ZK equation   gain one local spatial derivative  for each direction. 
%It is noteworthy to mention that this phenomenon manifests itself with a strong dependence on the dispersion in the variable $ x_ {1}.$ Roughly speaking, the variable where the dispersion is higher tends to governor the dynamics of this phenomena.
\begin{lem}\label{se}
	Let  $\psi$ be  a smooth function   such that $\psi,\psi'\geq 0.$  Let  $u$ be a function such that $u\in C\left([0,T]: H^{\infty}(\mathbb{R}^{n})\right)$ is a  solution of the IVP \eqref{zk4}. Assume also that $s\geq 0,\,r>n/2$ and $\sigma=(\sigma_{1},\sigma_{2},\dots,\sigma_{n})\in \mathbb{R}^{n},$  with 
	\begin{equation}\label{zk27}
	\sigma_{1}>0\qquad\mbox{and}\qquad \sqrt{3}\sigma_{1}>\sqrt{\sigma_{2}^{2}+\sigma_{3}^{2}+\cdots+\sigma_{n}^{2}}.
	\end{equation}
	 Then,
	\begin{equation}\label{zk26}
	\begin{split}
	&\sum_{m=1}^{n}\int_{0}^{T}\int_{\mathbb{R}^{n}} \left(\partial_{x_{m}}J^{s}u(x,t)\right)^{2}\,\psi'(\sigma\cdot x)\,\mathrm{d}x\,\mathrm{d}t\\
	&\lesssim_{n,\sigma}\left(1+T+\left\|\nabla u\right\|_{L^{1}_{T}L^{\infty}_{x}}+T\|u\|_{L^{\infty}_{T}H^{r}_{x}}\right) \left\|u\right\|_{L^{\infty}_{T}H^{s}_{x}}^{2}.
	\end{split}
	\end{equation}
	\end{lem}
	\begin{proof}		 
		For $\sigma$ as in \eqref{zk27} we define the function $$\psi_{\sigma}(x):=\psi(\sigma\cdot x),\quad x\in \mathbb{R}^{n}.$$
		
To obtain \eqref{zk26} we apply $J^{s}$ to the equation in \eqref{zk4} followed by a multiplication by  $J^{s}u\psi_{\sigma},$ so that after integrating  in the spatial variable yield
		\begin{equation}\label{zk25}
		\begin{split}
		&\frac{\mathrm{d}}{\mathrm{d}t}\int_{\mathbb{R}^{n}}\left(J^{s}u(x,t)\right)^{2} \psi_{\sigma}(x)\,\mathrm{d}x +\sigma_{1}\sum_{m=1}^{n}\int_{\mathbb{R}^{n}}\left(\partial_{x_{m}}J^{s}u(x,t)\right)^{2}\,\psi'(\sigma\cdot x)\,\mathrm{d}x\\
		&+2\sum_{m=1}^{n}\sigma_{m}\int_{\mathbb{R}^{n}}\left(\left(\partial_{x_{m}}J^{s}u\right)\left(\partial_{x_{1}}J^{s}u\right)\right)(x,t)\,\psi'(\sigma \cdot x)\,\mathrm{d}x\\
		&  -\sigma_{1}\sum_{m=1}^{n}\sigma_{m}^{2}\int_{\mathbb{R}^{n}}\left(J^{s}u(x,t)\right)^{2} \psi'''(\sigma \cdot x)\,\mathrm{d}x \\
		&  + \int_{\mathbb{R}^{n}}\left(J^{s}\left(u\partial_{x_{1}}u\right)J^{s}u\right)(x,t)\,\psi_{\sigma}(x)\,\mathrm{d}x=0.
		\end{split}
		\end{equation}
		At this point  we use  the approach used in   \cite{LPZK}, to obtain the smoothing effect. More precisely, 		for  $m$  a positive integer we define   
		\begin{equation*}
		\mu_{m}^{2}:=\int_{\mathbb{R}^{n}} \left(\partial_{x_{m}}J^{s}u\right)^{2}\,\psi'(\sigma\cdot x)\,\mathrm{d}x.
		\end{equation*}
	We claim that 
		\begin{equation}\label{zk22}
		\sigma_{1}\sum_{m=1}^{n}\mu_{m}^{2}+2\sum_{m=1}^{n}\sigma_{m}\int_{\mathbb{R}^{n}}\left(\left(\partial_{x_{m}}J^{s}u\right)\left(\partial_{x_{1}}J^{s}u\right)\right)(x,t)\psi'(\sigma\cdot x
		)\,\mathrm{d}x\equiv_{\sigma} \sum_{m=1}^{n} \mu_{m}^{2}.
		\end{equation}
		For our proposes it is necessary to decouple the term  involving the interactions of the derivatives in \eqref{zk22}, this can be achieved by using Cauchy-Schwarz inequality. More precisely, 
		  due to   $\psi'\geq 0,$ it is clear that   
		\begin{equation}\label{interaction}
	\left|	\int_{\mathbb{R}^{n}}\left(\left(\partial_{x_{m}}J^{s}u\right)\left(\partial_{x_{1}}J^{s}u\right)\right)(x,t)\psi'(\sigma\cdot x
		)\,\mathrm{d}x\right|\leq\mu_{m}\mu_{1}, \quad\mbox{for} \quad m\in \{1,2\dots,n\}.
		\end{equation}
		So that, for  $\sigma\in \mathbb{R}^{n}$ with 
		\begin{equation*}
		\sigma_{1}>0\qquad\mbox{and}\qquad \sqrt{3}\sigma_{1}>\sqrt{\sigma_{2}^{2}+\sigma_{3}^{2}+\cdots+\sigma_{n}^{2}},
		\end{equation*}
		there exist $\lambda=\lambda(\sigma)>0,$ such that 
		\begin{equation}\label{zk23}
		\sigma_{1}\sum_{m=1}^{n}\mu_{m}^{2}+2\sum_{m=1}^{n}\sigma_{m}\int_{\mathbb{R}^{n}}\left(\partial_{x_{m}}J^{s}u\right)\left(\partial_{x_{1}}J^{s}u\right)\psi'(\sigma\cdot x)
		\,\mathrm{d}x\geq \sum_{m=1}^{n} \lambda(\sigma)\mu_{m}^{2}.
		\end{equation}
		The opposite  inequality follows as  an straightforward  application of    Young's inequality in \eqref{interaction},  that is
		\begin{equation}\label{interaction2}
		\sigma_{1}\sum_{m=1}^{n}\mu_{m}^{2}+2\sum_{m=1}^{n}\sigma_{m}\int_{\mathbb{R}^{n}}\left(\partial_{x_{m}}J^{s}u\right)\left(\partial_{x_{1}}J^{s}u\right)\psi'(\sigma\cdot x)\,\mathrm{d}x\lesssim_{\sigma} \sum_{m=1}^{n} \mu_{m}^{2}.
		\end{equation}
	Finally, we  gather the  inequalities \eqref{interaction} and \eqref{interaction2} to obtain 
		\begin{equation*}
		\sigma_{1}\sum_{m=1}^{n}\mu_{m}^{2}+2\sum_{m=1}^{n}\sigma_{m}\int_{\mathbb{R}^{n}}\left(\partial_{x_{m}}J^{s}u\right)\left(\partial_{x_{1}}J^{s}u\right)\psi'(\sigma\cdot x
		)\,\mathrm{d}x\equiv_{\sigma} \sum_{m=1}^{n} \mu_{m}^{2}.
		\end{equation*}
		Instead, the nonlinear part is handled by means of integration by parts, H\"{o}lder's inequality and   \eqref{zk2}, whence we get 
		\begin{equation}\label{zk24}
		\begin{split}
		\int_{\mathbb{R}^{n}}J^{s}\left(u\partial_{x_{1}}u\right)J^{s}u\,\psi_{\sigma}\,\mathrm{d}x
		&= \int_{\mathbb{R}^{n}} \psi_{\sigma}J^{s}u\left[J^{s}; u\right]\partial_{x_{1}}u\,\mathrm{d}x-\frac{1}{2}\int_{\mathbb{R}^{n}}\partial_{x_{1}}u\left(J^{s}u\right)^{2}\psi_{\sigma}\,\mathrm{d}x \\
		&\quad -\frac{\sigma_{1}}{2}\int_{\mathbb{R}^{n}}u \left(J^{s}u\right)^{2}\psi'(\sigma\cdot x)\,\mathrm{d}x\\
		&\lesssim_{\sigma}\left(\left\|\nabla u\right\|_{L^{\infty}}+\|u\|_{L^{\infty}}\right)\|u\|_{H^{s}}^{2}.
		\end{split}
		\end{equation} 
		Therefore, we deduce after  gathering \eqref{zk25}, \eqref{zk23}, \eqref{zk24} and integrating in time   that
		\begin{equation*}
		\begin{split}
		&\sum_{m=1}^{n}\int_{0}^{T}\int_{\mathbb{R}^{n}}\left|\partial_{x_{m}}J^{s}u(x,t)\right|^{2}\,\psi'(\sigma\cdot x)\,\mathrm{d}x\,\mathrm{d}t\\
		&\lesssim_{n,\sigma} \left(1+T+\left\|\nabla u\right\|_{L^{1}_{T}L^{\infty}} +T\|u\|_{L^{\infty}_{T}L^{\infty}}\right) \left\|u\right\|_{L^{\infty}_{T}H^{s}}^{2}.
		\end{split}
		\end{equation*} 
		From the last inequality above  it is straightforward  to obtain \eqref{zk25} by means of  Sobolev embedding.
	\end{proof}
%\begin{rem}
%	A quite interesting property that will be useful in the proof of Theorem \ref{zk9} is to obtain the smoothing effect explicitly in terms of the Bessel operator $J^{s+1}$ instead of merely $\nabla J^{s},$ as was pointed out in \eqref{zk26}. However, when we restrict ourselves   to the hypothesis of theorem \ref{zk9}, this modification   is not straightforward to do, mainly because  the restrictions imposed in  the hypothesis of such theorem.
%\end{rem}
	  \section{Proof of Theorem \ref{zk9}}
In this section we provide  a proof of theorem \ref{zk9}, nevertheless before provide the proof scheme we will describe the weighted functions to be used in our analysis.

More precisely, for $\epsilon>0$ and $\tau\geq 5\epsilon$ we define the following  families  of functions
\begin{equation*}
\chi_{\epsilon, \tau},\widetilde{\phi_{\epsilon,\tau}}, \phi_{\epsilon,\tau},\psi_{\epsilon}\in C^{\infty}(\mathbb{R}),
\end{equation*} 
 satisfying the   conditions indicated below:
\begin{itemize}
	\item[(i)] $
	\chi_{\epsilon, \tau}(x)=
	\begin{cases} 
	1 & x\geq \tau \\
	0 & x\leq \epsilon 
	\end{cases},
$
\item[(ii)] $\supp(\chi_{\epsilon, \tau}')\subset[\epsilon,\tau],$ 
\item[(iii)] $\chi_{\epsilon, \tau}'(x)\geq 0,$
\item[(iv)] $
\chi_{\epsilon, \tau}'(x)\geq\frac{1}{10(\tau-\epsilon)}\mathbb{1}_{[2\epsilon,\tau-2\epsilon]}(x),$
\item[(v)] $\supp\left(\widetilde{\phi_{\epsilon,\tau}}\right),\supp(\phi_{\epsilon,\tau})\subset \left[\frac{\epsilon}{4},\tau\right],$
\item[(vi)] $\phi_{\epsilon,\tau}(x)=\widetilde{\phi_{\epsilon,\tau}}(x)=1,\,\mbox{if}\quad  x\in \left[\frac{\epsilon}{2},\epsilon\right],$
\item[(vii)] $\supp(\psi_{\epsilon})\subset \left(-\infty,\frac{\epsilon}{4}\right].$
\item[(viii)] For all $x\in\mathbb{R}$ the following quadratic partition of the unity holds
$$
\chi_{\epsilon, \tau}^{2}(x)+\widetilde{\phi_{\epsilon,\tau}}^{2}(x)+\psi_{\epsilon}(x)=1,$$
\item[(ix)] also for all $x\in\mathbb{R}$
 $$
\chi_{\epsilon, \tau}(x)+\phi_{\epsilon,\tau}(x)+\psi_{\epsilon}(x)=1,\quad x\in\mathbb{R}.
$$
\end{itemize}
For a more detailed construction of these families of weighted  functions see  \cite{ILP1}.
\begin{proof}[Proof of theorem \ref{zk9}]
Formally we apply the operator $J^{s}$ to the equation  in \eqref{zk4}, followed by a multiplication by $J^{s}u\,\chi_{\epsilon,\tau,\sigma}^{2}(x,t)$   to obtain 
\begin{equation*}
\left(J^{s}\partial_{t}u J^{s}u+ J^{s}\partial_{x_{1}}\Delta uJ^{s}u+J^{s}\left(u\partial_{x_{1}}u\right)J^{s}u\right)\chi_{\epsilon,\tau,\sigma}^{2}(x,t)=0,
\end{equation*}
where 
\begin{equation*}
\chi_{\epsilon, \tau,\sigma }(x,t):=\chi_{\epsilon, \tau}\left(\sigma\cdot x+\nu t-\beta\right),\quad x\in \mathbb{R}^{n},\,t\geq 0,
\end{equation*}
similarly are  defined the functions  $\phi_{\epsilon, \tau,\sigma }(x,t),\widetilde{\phi_{\epsilon, \tau,\sigma}}(x,t)$ and $\psi_{\epsilon,\sigma}(x,t).$

Without loss of generality we will assume  that  $\beta=0.$

So that,  after integrating by parts  we obtain the following identity:
\begin{equation}\label{zkwee}
\begin{split}
&\frac{1}{2}\frac{\mathrm{d}}{\mathrm{d}t}\int_{\mathbb{R}^{n}}\left(J^{s}u(x,t)\right)^{2} \chi^{2}_{\epsilon,\tau,\sigma}(x,t)\,\mathrm{d}x\underbrace{-\nu\int_{\mathbb{R}^{n}} \left(J^{s}u(x,t)\right)^{2} \left(\chi_{\epsilon,\tau}\,\chi_{\epsilon, \tau}'\right)(\sigma\cdot x+\nu t)\,\mathrm{d}x}_{A_{1}(t)}\\
& +\sigma_{1}\sum_{m=1}^{n}\underbrace{\int_{\mathbb{R}^{n}}\left(\partial_{x_{m}}J^{s}u(x,t)\right)^{2}\left(\chi_{\epsilon,\tau}\,\chi_{\epsilon,\tau}'\right)(\sigma\cdot x+\nu t)\,\mathrm{d}x}_{A_{2,m}(t)}\\
&+2\sum_{m=1}^{n}\sigma_{m}\underbrace{\int_{\mathbb{R}^{n}}\left(\partial_{x_{m}}J^{s}u(x,t)\right)\left(\partial_{x_{1}}J^{s}u(x,t)\right)\left(\chi_{\epsilon,\tau}\,\chi_{\epsilon,\tau}'\right)(\sigma\cdot x+\nu t)\,\mathrm{d}x}_{A_{3,m}(t)}\\
& \underbrace{-3\sigma_{1}|\sigma|^{2}\int_{\mathbb{R}^{n}}\left(J^{s}u(x,t)\right)^{2} \left(\chi_{\epsilon,\tau}'\,\chi_{\epsilon,\tau}''\right)(\sigma\cdot x+\nu t)\,\mathrm{d}x}_{A_{4}(t)}\\
&    \underbrace{-\sigma_{1}|\sigma|^{2}\int_{\mathbb{R}^{n}}\left(J^{s}u(x,t)\right)^{2}\left(\chi_{\epsilon,\tau}\,\chi_{\epsilon,\tau}'''\right)(\sigma\cdot x+\nu t)\,\mathrm{d}x}_{A_{5}(t)}\\
& +\underbrace{\int_{\mathbb{R}^{n}}J^{s}\left(u(x,t)\partial_{x_{1}}u(x,t)\right)J^{s}u(x,t)\,\chi_{\epsilon,\tau,\sigma }^{2}(x,t)\,\mathrm{d}x}_{A_{6}(t)}=0.
\end{split}
\end{equation}

\begin{flushleft}
{\sc Case : $s\in \left(s_{n}, s_{n}+1\right)$.}
\end{flushleft}
For $\epsilon>0$ and $\tau\geq 5\epsilon,$  there exist $\delta>0$ and $c>0$ such that 
 \begin{equation*}
 \chi_{\epsilon,\tau,\sigma}(x, t)\,\chi_{\epsilon, \tau}'(\sigma\cdot x+\nu t  )\leq c\mathbb{1}_{\mathcal{Q}_{\{\sigma, -\delta, \delta\}}}(x) \quad x\in \mathbb{R}^{n},\, t\in [0,T].
 \end{equation*}
 Hence, 
 \begin{equation}\label{zk36}
\begin{split}
\int_{0}^{T}|A_{1}(t)|\,\mathrm{d}t&\leq|\nu|\int_{0}^{T}\int_{\mathbb{R}^{n}} \left(J^{s}u(x,t)\right)^{2}\left( \chi_{\epsilon,\tau
	,}\,\chi_{\epsilon, \tau}'\right)(\sigma\cdot x+\nu t)\,\mathrm{d}x\,\mathrm{d}t\\
&\leq c|\nu| \int_{0}^{T}\int_{\mathcal{Q}_{\{\sigma,-\delta, \delta\}}} \left(J^{s}u(x,t)\right)^{2} \,\mathrm{d}x\,\mathrm{d}t.\\
\end{split}
\end{equation}
Nevertheless, by  constructing  $\psi$  properly in  lemma \ref{se} it can be deduced that 
there exist $\delta_{1}>0,$ such that 
\begin{equation}\label{zk35}
\int_{0}^{T}\int_{\mathcal{Q}_{\left\{\sigma, -\delta_{1}, \delta_{1}\right\}}}\!\!\!\!\! \left(J^{r}u(x,t)\right)^{2}\,\mathrm{d}x\,\mathrm{d}t\leq c,\quad \mbox{for any}\quad r\in \left[0, \beta\right],
\end{equation}
where $\beta=s_{n}+1.$  

Thus,  as a particular case we get that  $A_{1}\in L^{1}_{T}.$

By using a similar argument  we deduce  that
\begin{equation*}
\begin{split}
\int_{0}^{T}|A_{4}(t)|\,\mathrm{d}t&\leq 3\sigma_{1}|\sigma|^{2}\int_{0}^{T}\int_{\mathbb{R}^{n}}\left(J^{s}u\right)^{2} \left(\chi_{\epsilon,\tau
}'\,\chi_{\epsilon,\tau}''\right)(\sigma\cdot x+\nu t)\,\mathrm{d}x\,\mathrm{d}t\\
&\leq c_{\sigma,\epsilon},
\end{split}
\end{equation*}
and 
\begin{equation*}
\begin{split}
\int_{0}^{T}|A_{5}(t)|\,\mathrm{d}t &\leq 3\sigma_{1}|\sigma|^{2}\int_{0}^{T}\int_{\mathbb{R}^{n}}\left(J^{s}u(x,t)\right)^{2} \left(\chi_{\epsilon,\tau}'\,|\chi_{\epsilon,\tau}''|\right)(\sigma\cdot x+\nu t)\,\mathrm{d}x\,\mathrm{d}t\\
&\leq c_{\sigma,\epsilon}.
\end{split}
\end{equation*}

Next, for $m$  a positive  integer with $1\leq m\leq n,$  we set
\begin{equation}\label{interec}
\mu_{m}^{2}(t):=\int_{\mathbb{R}^{n}} \left(\partial_{x_{m}}J^{s}u(x,t)\right)^{2}\left(\chi_{\epsilon,\tau}\,\chi_{\epsilon,\tau}'\right)(\sigma\cdot x+\nu t)\,\mathrm{d}x, \quad  t\in (0,T).
\end{equation}
By using an argument similar to that one used in the proof of lemma \ref{se},   we obtain that  there exist a positive number  $\lambda=\lambda(\sigma),$  such that    
\begin{equation}\label{interaction3}
\begin{split}
       &\sum_{m=1}^{n}\lambda \mu_{m}^{2}\\
       &\lesssim 2\sum_{m=1}^{n}\sigma_{m}\int_{\mathbb{R}^{n}}\left(\partial_{x_{m}}J^{s}u(x,t)\right)\left(\partial_{x_{1}}J^{s}u(x,t)\right)\left(\chi_{\epsilon,\tau}\,\chi_{\epsilon,\tau}'\right)(\sigma\cdot x +\nu t)\,\mathrm{d}x\\
       &\quad  +\sigma_{1}\sum_{m=1}^{n}\mu_{m}^{2}.
\end{split}
\end{equation}
%where $\lambda$ satisfies the inequality   $$0<\lambda<2\sigma_{1}-|\sigma|.$$ 

%The inequality  \eqref{interaction3} will be fundamental to proceed in our inductive argument, 
%mainly, because  it  allows us to decouple the interaction between the derivatives  involved into one term  only containing the smoothing effect.
%
%We shall  stress that the sum in the r.h.s of  \eqref{interaction3} corresponds to the smoothing effect and it  will be bounded after apply Gronwall's inequality and integration in time.
%%For this, notice that the terms that measures in some sense the interactions can be handled by  Cauchy-Schwarz inequality  
%%\begin{equation}\label{interaction}
%%\begin{split}
%%\left|\int_{\mathbb{R}^{n}}\left(\partial_{x_{m}}J^{s}u\right)\left(\partial_{x_{1}}J^{s}u\right)\chi_{\epsilon,\tau, \sigma}\,\chi_{\epsilon,\tau,\sigma}'\,\mathrm{d}x\right|&\leq \mu_{m}\mu_{1}.
%%\end{split}
%%\end{equation}
%% 
%% 
%% 
%%
 Finally, we show how to control  the nonlinear part. First, notice that by  Cauchy-Schwarz inequality  it is enough to estimate $ \chi_{\epsilon, \tau,\sigma}J^{s}(u\partial_{x_{1}}u)$ in the $L^{2}-$norm. In this sense,  the following   relationships are very useful
 \begin{equation*}
\chi_{\epsilon, \tau,\sigma}(x,t)+\phi_{\epsilon, \tau,\sigma}(x,t)+\psi_{\epsilon,\sigma}(x,t)=1,\quad x\in \mathbb{R}^{n},t\in\mathbb{R}
\end{equation*}
and 
\begin{equation*}
\chi_{\epsilon, \tau,\sigma}^{2}(x,t)+\widetilde{\phi_{\epsilon, \tau,\sigma}}^{2}(x,t)+\psi_{\epsilon,\sigma}(x,t)=1,\quad x\in \mathbb{R}^{n},t \in \mathbb{R}.
\end{equation*}
So that,  in order to integrate  these decomposition into the non-linear part we write
\begin{equation}\label{nlp1}
\begin{split}
&\chi_{\epsilon, \tau,\sigma} J^{s}(u\partial_{x_{1}}u)\\
&= -\frac{1}{2}\left[J^{s}; \chi_{\epsilon,\tau,\sigma}\right]\partial_{x_{1}}\left((u\chi_{\epsilon,\tau,\sigma})^{2}+\left(u\widetilde{\phi_{\epsilon,\tau,\sigma}}\right)^{2}+u^{2}\psi_{\epsilon,\sigma}\right)\\
&\quad +\left[J^{s};u\chi_{\epsilon,\tau,\sigma}\right]\partial_{x_{1}}\left(u\chi_{\epsilon,\tau,\sigma}\right)+ \left[J^{s};u\chi_{\epsilon,\tau,\sigma}\right]\partial_{x_{1}}\left(u\phi_{\epsilon,\tau,\sigma}\right)\\ 
&\quad +\left[J^{s};u\chi_{\epsilon,\tau,\sigma}\right]\partial_{x_{1}}\left(u\psi_{\epsilon,\sigma}\right)+ \chi_{\epsilon,\tau,\sigma }uJ^{s}\partial_{x_{1}}u\\
&=B_{6,1}+B_{6,2}+B_{6,3}+B_{6,4}+B_{6,5}+B_{6,6}+B_{6,7}.
\end{split}
\end{equation}
%The terms    $B_{6,1}$ and $B_{6,2}$   cannot be bounded directly by using the Kato-Ponce commutator estimate \eqref{zk2}. This is basically because the derivatives of  $\chi_{\epsilon,\tau,\sigma}$  are  not in $L^{p}(\mathbb{R}^{n})$ for $1\leq p<\infty,$  but on the contrary all are in $L^{\infty}(\mathbb{R}^{n}).$  So, unlike the one-dimensional case, another approach is required to avoid this problem. These limitations, can be avoided   by using pseudo-differential calculus, that is,  breaking up the commutator term   
%until we get an expression that does not require to be localized.

First, we denote by $\Psi_{\zeta}$ the operator
${\displaystyle 
\Psi_{\zeta}:=\left[J^{s}; \chi_{\epsilon,\tau,\sigma}\right]},$
 that according to proposition \ref{prop1}
\begin{equation}\label{Commutator1}
\begin{split}
\zeta(x,\xi)\sim \sum_{\alpha} \frac{(2\pi i)^{-|\alpha|}}{\alpha!}\left(\partial^{\alpha}_{\xi}\left(\langle \xi\rangle^{s}\right)\partial_{x}^{\alpha}\left(\chi_{\epsilon,\tau,\sigma}(x,t)\right)\right)\quad \mbox{for any}\quad t\geq 0.
\end{split}
\end{equation}
Nevertheless,  we have  previously  obtained similar decomposition for this operator in the proof of lemma \ref{A}, so that, for the sake of brevity we will omit the details.

Therefore, we claim that there exist a pseudo-differential operator $\Psi_{\kappa_{s-m-1}}\in\mathrm{OP}\mathbb{S}^{s-m-1},$  such that for any  $f\in\mathcal{S}$  the following identity holds
\begin{equation}\label{e10}
\begin{split}
\left[J^{s}; \chi_{\epsilon, \tau,\sigma}\right]f(x)&=\sum_{j=1}^{m}\sum_{|\alpha|=j}\sum_{\beta\leq\alpha} c_{\alpha,\beta}\sigma^{\alpha} \,\partial_{x}^{\alpha}\chi_{\epsilon, \tau,\sigma}(x, t)\Psi_{\eta_{\alpha,\beta}}J^{s-|\alpha|}f(x)\\
&\quad + \Psi_{\kappa_{s-m-1}}f(x), 
\end{split}
\end{equation}
where $m=\ceil[\Big]{\frac{n}{2}}+1,$ and  $\Psi_{\eta_{\alpha,\beta}}\in \mathrm{OP}\mathbb{S}^{|\beta|-|\alpha|}\subset\mathrm{OP}\mathbb{S}^{0}$ is defined  as in  \eqref{e11}. 

Hence,   combining  theorem \ref{zkth1}, interpolation and  inequality \eqref{zk8},  we obtain 
\begin{equation}\label{e12}
\begin{split}
\|B_{6,1}\|_{L^{2}}&=\left\|\left[J^{s};\chi_{\epsilon, \tau,\sigma}\right]\partial_{x_{1}}\left(\left(u\chi_{\epsilon,\tau,\sigma}\right)^{2}\right)\right\|_{L^{2}}\\
&\leq \sum_{j=1}^{m}\sum_{|\alpha|=j}\sum_{\beta\leq\alpha} c_{\alpha,\beta}\sigma^{\alpha} \,\left\|\partial_{x}^{\alpha}\chi_{\epsilon, \tau,\sigma}(\cdot,t)\right\|_{L^{\infty}_{x}}\left\|\Psi_{\eta_{\alpha,\beta}}J^{s-|\alpha|}\left(\partial_{x_{1}}\left(\left(u\chi_{\epsilon,\tau,\sigma}\right)^{2}\right)\right)\right\|_{L^{2}}\\
&\quad +\left\| \Psi_{\kappa_{s-m-1}}\left(\partial_{x_{1}}\left(\left(u\chi_{\epsilon,\tau,\sigma}\right)^{2}\right)\right)\right\|_{L^{2}}\\
%&\textcolor{blue}{\lesssim_{\sigma,\epsilon,\tau}\sum_{j=1}^{m}\sum_{|\alpha|=j}\sum_{\beta\leq\alpha} c_{\alpha,\beta}\sigma^{\alpha} \,\left\|J^{s-|\alpha|}\left(\partial_{x_{1}}\left(\left(u\chi_{\epsilon,\tau,\sigma}\right)^{2}\right)\right)\right\|_{L^{2}}+\left\| \partial_{x_{1}}\left(\left(u\chi_{\epsilon,\tau,\sigma}\right)^{2}\right)\right\|_{L^{2}}}\\
%&\textcolor{blue}{= c_{\sigma,\epsilon,\tau}\sum_{j=1}^{m}\sum_{|\alpha|=j}\sum_{\beta\leq\alpha} c_{\alpha,\beta}\sigma^{\alpha} \,\left\|J^{s-1}\left(\partial_{x_{1}}\left(\left(u\chi_{\epsilon,\tau,\sigma}\right)^{2}\right)\right)\right\|_{L^{2}}+\left\| J^{1}\left(\left(u\chi_{\epsilon,\tau,\sigma}\right)^{2}\right)\right\|_{L^{2}}}\\
%&\textcolor{blue}{\lesssim \left(1+c_{\sigma,\epsilon,\tau}\sum_{j=1}^{m}\sum_{|\alpha|=j}\sum_{\beta\leq\alpha} c_{\alpha,\beta}\sigma^{\alpha} \right)\left\|J^{s}\left(\left(u\chi_{\epsilon,\tau,\sigma}\right)^{2}\right)\right\|_{L^{2}}}\\
&\lesssim_{\epsilon,\tau,\sigma,m}\left\|u\right\|_{L^{\infty}}\left\|J^{s}\left(u\chi_{\epsilon,\tau,\sigma}\right)\right\|_{L^{2}}.
\end{split}
\end{equation}
Similarly, 
\begin{equation}\label{e13}
\begin{split}
\|B_{6,2}\|_{L^{2}}=&\left\|\left[J^{s}; \chi_{\epsilon,\tau,\sigma }\right]\partial_{x_{1}}\left(\left(u\widetilde{\phi_{\epsilon,\tau,\sigma}}\right)^{2}\right)\right\|_{L^{2}}\\
%&\textcolor{blue}{\leq \sum_{j=1}^{m}\sum_{|\alpha|=j}\sum_{\beta\leq\alpha} c_{\alpha,\beta}\sigma^{\alpha} \,\left\|\partial_{x}^{\alpha}\chi_{\epsilon, \tau,\sigma}(\cdot,t)\right\|_{L^{\infty}_{x}}\left\|\Psi_{\eta_{\alpha,\beta}}J^{s-|\alpha|}\left(\partial_{x_{1}}\left(\left(u\widetilde{\phi_{\epsilon,\tau,\sigma}}\right)^{2}\right)\right)\right\|_{L^{2}}}\\
%&\quad \textcolor{blue}{+\left\| \Psi_{\kappa_{s-m-1}}\left(\partial_{x_{1}}\left(\left(u\widetilde{\phi_{\epsilon,\tau,\sigma}}\right)^{2}\right)\right)\right\|_{L^{2}}}\\
%&\textcolor{blue}{\lesssim_{\sigma,\epsilon,\tau}\sum_{j=1}^{m}\sum_{|\alpha|=j}\sum_{\beta\leq\alpha} c_{\alpha,\beta}\sigma^{\alpha} \,\left\|J^{s-|\alpha|}\left(\partial_{x_{1}}\left(\left(u\widetilde{\phi_{\epsilon,\tau,\sigma}}\right)^{2}\right)\right)\right\|_{L^{2}}+\left\| \partial_{x_{1}}\left(\left(u\widetilde{\phi_{\epsilon,\tau,\sigma}}\right)^{2}\right)\right\|_{L^{2}}}\\
%&\textcolor{blue}{= c_{\sigma,\epsilon,\tau}\sum_{j=1}^{m}\sum_{|\alpha|=j}\sum_{\beta\leq\alpha} c_{\alpha,\beta}\sigma^{\alpha} \,\left\|J^{s-1}\left(\partial_{x_{1}}\left(\left(u\widetilde{\phi_{\epsilon,\tau,\sigma}}\right)^{2}\right)\right)\right\|_{L^{2}}+\left\| J^{1}\left(\left(u\widetilde{\phi_{\epsilon,\tau,\sigma}}\right)^{2}\right)\right\|_{L^{2}}}\\
%&\textcolor{blue}{\lesssim\left(1+c_{\sigma,\epsilon,\tau}\sum_{j=1}^{m}\sum_{|\alpha|=j}\sum_{\beta\leq\alpha} c_{\alpha,\beta}\sigma^{\alpha}\right)\left\|J^{s}\left(\left(u\widetilde{\phi_{\epsilon,\tau,\sigma}}\right)^{2}\right)\right\|_{L^{2}}}\\
&\lesssim_{\epsilon,\tau,\sigma,m}\left\|u\right\|_{L^{\infty}}\left\|J^{s}\left(u\widetilde{\phi_{\epsilon,\tau,\sigma}}\right)\right\|_{L^{2}}.
\end{split}
\end{equation}
For  the term $B_{6,3}$ it is not   required the use of the previous machinery to obtain some upper  bounds. Instead, we  notice  that  for any $t\geq 0,$
\begin{equation}\label{e18}
\dist\left(\supp \left(\chi_{\epsilon, \tau,\sigma}(\cdot,t)\right), \supp \left(\psi_{\epsilon,\sigma}(\cdot,t)\right)\right)\geq \frac{\epsilon}{2|\sigma|}.
\end{equation}
So, we fall in the hypothesis of lemma \ref{lem1}, therefore it is clear that 
\begin{equation}\label{e19} 
\|B_{6,3}\|_{L^{2}}=\left\|\chi_{\epsilon, \tau,\sigma}J^{s}\partial_{x_{1}}\left(u^{2}\psi_{\epsilon,\sigma}\right)\right\|_{L^{2}}\lesssim_{\epsilon,\sigma} \|u_{0}\|_{L^{2}}\|u\|_{L^{\infty}}.
\end{equation}
Also, from the remark indicated in \eqref{e18} and lemma \ref{lem1} it also follows that
\begin{equation}\label{e20}
\left\|B_{6,6}\right\|_{L^{2}}=\left\|u\chi_{\epsilon, \tau,\sigma}J^{s}\left(u\psi_{\epsilon,\sigma}\right)\right\|_{L^{2}_{x}}
\lesssim _{\epsilon,\sigma}\|u\|_{L^{\infty}}\|u_{0}\|_{L^{2}}.
\end{equation} 
Concerning to the  terms $B_{6,4}$ and $B_{6,5},$ the scenario is quite different,  and the argument used above  can be avoided by  using the Kato-Ponce commutator estimate \eqref{zk2}. More precisely,
\begin{equation}\label{e14}
\left\|\left[J^{s}; u\chi_{\epsilon, \tau,\sigma}\right]\partial_{x_{1}}\left(u\chi_{\epsilon,\tau,\sigma}\right)\right\|_{L^{2}}\lesssim \left\|J^{s}\left(u\chi_{\epsilon, \tau,\sigma}\right)\right\|_{L^{2}}\left\|\partial\left(u \chi_{\epsilon, \tau,\sigma }\right)\right\|_{L^{\infty}},
\end{equation}
and 
\begin{equation}\label{e15}
\begin{split}
\left\|\left[J^{s}; u\chi_{\epsilon, \tau,\sigma }\right]\partial_{x_{1}}\left(u\phi_{\epsilon, \tau,\sigma }\right)\right\|_{L^{2}}
&\lesssim \left\|J^{s}\left(u\phi_{\epsilon,\tau,\sigma }\right)\right\|_{L^{2}}\left\|\nabla\left(u\chi_{\epsilon, \tau,\sigma }\right)\right\|_{L^{\infty}}\\
&\quad  +\left\|J^{s}\left(u\chi_{\epsilon, \tau,\sigma }\right)\right\|_{L^{2}}\left\|\partial_{x_{1}}\left(u\phi_{\epsilon, \tau,\sigma }\right)\right\|_{L^{\infty}}.
\end{split}
\end{equation}
 Next, we bound the term $B_{6,7},$ for this  we notice that up to constants,  this term  is in essence  upper bounded by  $\|u(t)\|_{L^{\infty}}$ times  $A_{1}(t),$  more precisely,   
\begin{equation}
\begin{split}
|B_{6,7}(t)|&\lesssim_{\sigma} \|u(t)\|_{L^{\infty}}|A_{1}(t)|.
\end{split}
\end{equation}
Since it was proved previously that $A_{1}\in L^{1}_{T}$  is bounded,  for the sake of brevity we will omit the proof,  so that,  we  shall  only guarantee that $\|u(t)\|_{L^{\infty}}<\infty,$ for almost all $t\in [0,T].$ Nevertheless, a combination of the local theory and the Sobolev  embedding clearly imply that $\|u(t)\|_{L^{\infty}}<\infty.$

Instead, to handle the term  $B_{6,8},$  we only require to use Sobolev's embedding,  this in order to guarantee    that   $\nabla  u\in L^{1}\left( [0,T]:L^{\infty}(\mathbb{R}^{n})\right).$  More precisely, 

%Since our argument depends strongly on the applicability of Gronwall's inequality   we require a local theory that in the Sobolev space $H^{s}(\mathbb{R}^{n})$ that guarantee $\partial_{x_{1}}u\in L^{1}\left([0,T]:L^{\infty}(\mathbb{R}^{n})\right)$
\begin{equation*}
\begin{split}
|B_{6,8}(t)| %&\lesssim\left\|\partial_{x_{1}}u(t)\right\|_{L^{\infty}}\int_{\mathbb{R}^{n}} \left(J^{s}u(x,t)\right)^{2} \chi_{\epsilon,\tau}^{2}(\sigma\cdot x+\nu t)\,\mathrm{d}x\\
&\lesssim \left\|\nabla u(t)\right\|_{L^{\infty}}\int_{\mathbb{R}^{n}} \left(J^{s}u(x,t)\right)^{2} \chi_{\epsilon,\tau,\sigma }^{2}(x, t)\,\mathrm{d}x,
\end{split}
\end{equation*}
notice that the quantity to be estimated by Gronwall's inequality corresponds to the integral expression in the r.h.s  above.

Finally, we turn our attention to several terms that remains  to estimate in several inequalities above e.g.   the term $\left\|J^{s}\left(u\chi_{\epsilon, \tau,\sigma }\right)\right\|_{L^{2}} $ present in   \eqref{e12},  \eqref{e14}, \eqref{e15}. Nevertheless, after taking into consideration  the following  decomposition 
\begin{equation*}
\begin{split}
J^{s}\left(u\chi_{\epsilon, \tau,\sigma}\right)
&=\chi_{\epsilon, \tau,\sigma }J^{s}u+\left[J^{s}; \chi_{\epsilon, \tau,\sigma }\right]\left(u\chi_{\epsilon, \tau,\sigma }+u\phi_{\epsilon, \tau,\sigma }+u\psi_{\epsilon,\sigma}\right)\\
&=I +II+III+IV.
\end{split}
\end{equation*}
The  term $I$ in the r.h.s above  is the quantity to be estimated by Gronwall's inequality; instead the term $IV$ can be easily  handled by using lemma \ref{lem1}. More precisely, we get 
\begin{equation*}
\|IV\|_{L^{2}}=\left\|\chi_{\epsilon, \tau,\sigma }J^{s}\left(\psi_{\epsilon,\sigma}u\right)\right\|_{L^{2}}\lesssim_{\epsilon,\sigma} \|u_{0}\|_{L^{2}}.
\end{equation*}
Instead the terms, $\left\|J^{s}\left(u\phi_{\epsilon, \tau,\sigma }\right)\right\|_{L^{2}},\, \left\|J^{s}\left(u\widetilde{\phi_{\epsilon, \tau,\sigma }}\right)\right\|_{L^{2}}$   can be estimated by using lemma \ref{A}.

However, to bound the terms $II$ and $III$ we  require to  decompose the commutator in a    similar manner  to that one in \eqref{Commutator1}-\eqref{e10}.  So that, we get   
\begin{equation*}
\|II\|_{L^{2}}\lesssim_{\epsilon,\tau,\sigma,s} \left\|J^{s-1}(u\chi_{\epsilon, \tau,\sigma })\right\|_{L^{2}}
\end{equation*}
and
\begin{equation*}
\|III\|_{L^{2}}\lesssim_{\epsilon,\tau,\sigma,s} \left\|J^{s-1}(u\phi_{\epsilon, \tau,\sigma })\right\|_{L^{2}}.
\end{equation*}
Later, after  integrating in time, we get that 
\begin{equation*}
\|II\|_{L^{2}_{T}L^{2}}^{2}\lesssim_{\epsilon,\tau,\sigma,s} T\left\|J^{s-1}(u\chi_{\epsilon, \tau,\sigma })\right\|_{L^{2}_{T}L^{2}}^{2}<\infty,
\end{equation*}
the last inequality above follows as a direct application of  theorem \ref{lwp}.  

Similarly, 
\begin{equation*}
\|III\|_{L^{2}_{T}L^{2}}^{2}\lesssim_{\epsilon,\tau,\sigma,s} T\left\|J^{s-1}(u\phi_{\epsilon, \tau,\sigma })\right\|_{L^{2}_{T}L^{2}}^{2}<\infty.
\end{equation*}
Finally, this step conclude  by gathering the estimates above combined with Gronwall's inequality and integration in time  whence we obtain that: \\
for any $ \nu\geq 0, \epsilon>0, \, \tau \geq 5 \epsilon,$ 
\begin{equation}\label{p1}
\sup_{0\leq t\leq T}\left\|J^{s}u\right\|_{L^{2}\left(\mathcal{H}_{\{\sigma,\epsilon-\nu t\}}\right)}^{2}+\sum_{m=1}^{n}\lambda(\sigma)\left\|\partial_{x_{m}}J^{s}u\right\|_{L^{2}_{T}L^{2}\left(\mathcal{Q}_{\{\sigma,\epsilon-\nu t, \tau-\nu t \}}\right)}^{2}\leq c^{*}_{(s)},
\end{equation}
where   $c^{*}_{(s)}$ is a positive constant depending on the following quantities  \\${\displaystyle c^{*}_{(s)}=c^{*}_{(s)}\left(\epsilon;\tau;\sigma;\lambda;\nu;n;s;T;\left\|J^{s}u_{0}\right\|_{L^{2}\left(\mathcal{H}_{\{\sigma,\epsilon\}}\right)};\left\|u\right\|_{L^{\infty}_{T}H^{s}}\right)>0.}$

%By way of comment it is important to highlight the meaning of the terms  
%of the terms that make up the previous inequality. In particular, the   first term in the left hand side  describes the propagation of regularity  of the function $u$

Notice that the second term in the left hand side above  provides a gain of one local derivative in all directions, being this  indicated  by the presence of   the operator $\nabla J^{s}$. However, our method of proof requires   a  modified version of the smoothing effect obtained in  \eqref{p1}. Roughly speaking,    we shall  indicate the  smoothing effect in terms of  the operator $J^{s+1}$ instead of $\nabla J^{s}.$ 

So that,  we will undertake this modification by means of a chains o claims that will imply the smoothing effect required.
%Before proceed with the inductive argument, we  shall rewrite the smoothing effect term above  in such  a way that it could be  hand-able  we  require some minors changes that will be required in our analysis.
\begin{flushleft}
{\sc Claim 1:}
	\end{flushleft}
If for any $\nu\geq 0,\, \epsilon_{1}>0$ and $\tau_{1}\geq 5\epsilon_{1}$
\begin{equation}\label{e22}
\sum_{m=1}^{n}\lambda(\sigma)\left\|\partial_{x_{m}}J^{s}u\right\|_{L^{2}_{T}L^{2}\left(\mathcal{Q}_{\{\sigma,\epsilon_{1}-\nu t, \tau_{1}-\nu t \}}\right)}^{2}\leq c^{*}_{(s)},
\end{equation}
for some positive constant $c^{*}_{(s)},$ then   there exist a  constant $c^{**}_{(s)}>0,$  such that 
for any $ \nu \geq 0,\epsilon>0,$ and $\tau\geq \epsilon,$
\begin{equation}\label{e24}
\left\|J^{s+1}u\right\|_{L^{2}_{T}L^{2}\left(\mathcal{Q}_{\{\sigma,\epsilon-\nu t, \tau-\nu t \}}\right)}\leq c^{**}_{(s)},
\end{equation}
where ${\displaystyle c^{**}_{(s)}=c^{**}_{(s)}\left(\epsilon;\tau; T;n; c^{*}_{(s)};\lambda(\sigma);\|u_{0}\|_{L^{2}}\right)}.$
\begin{proof}[{\sc Proof of Claim 1}]
We claim that there exist functions $\theta_{1},\theta_{2}\in C^{\infty}\left(\mathbb{R}^{n}\right),$ with bounded derivatives such that:
%	\begin{equation}
%	\dist\left(\supp\left(\theta_{1}\right),\supp\left(\theta_{2}\right)\right)\geq\frac{\epsilon}{2|\sigma|},
%	\end{equation} 
	\begin{equation*}
	\supp\left(\theta_{1}\right)\subseteq \mathcal{Q}_{\left\{\sigma,\frac{\epsilon}{2},\tau+\epsilon\right\}}\quad\mbox{with}\quad \theta_{1}\equiv 1 \quad \mbox{on}\quad \overline{\mathcal{Q}_{\left\{\sigma,\epsilon,\tau\right\}}},
	\end{equation*}
	and 
		\begin{equation*}
	\supp\left(\theta_{2}\right)\subseteq \mathcal{Q}_{\left\{\sigma,\frac{3\epsilon}{2},\tau-\epsilon\right\}}\quad\mbox{with}\quad \theta_{2}\equiv 1 \quad \mbox{on}\quad \overline{\mathcal{Q}_{\left\{\sigma,2\epsilon,\tau-2\epsilon\right\}}}.
	\end{equation*}
%%	\textcolor{blue}{The interested reader can verify that the functions indicated below satisfy such conditions:
%%	\begin{equation}
%	\theta_{1}(x):=\varphi_{\sigma, \epsilon}(x)\left(1-\varphi_{\sigma,2\epsilon}\left(x+\frac{\left(\epsilon-\tau\right)\sigma}{|\sigma|^{2}}\right)\right), \quad x\in \mathbb{R}^{n}
%	\end{equation}
%	and 
%	\begin{equation}
%	\theta_{2}(x):=\varphi_{\sigma, \epsilon}\left(x-\frac{\epsilon\sigma}{|\sigma|^{2}}\right)\left(1-\varphi_{\sigma,2\epsilon}\left(x+\frac{\left(3\epsilon-\tau\right)\sigma}{|\sigma|^{2}}\right)\right), \quad x\in \mathbb{R}^{n}.
%	\end{equation}}
	Next, notice  that	for any $s>0,$ the following representation holds
	\begin{equation*}
	J^{s+1}f=J^{s-1}f+\sum_{m=1}^{n}c_{m}\Psi_{m}\partial_{x_{m}}J^{s}f.
	\end{equation*} 
%	\begin{equation}\label{e22}
%	\begin{split}
%	\sum_{m=1}^{n}\left\|\partial_{x_{m}}J^{s}u\right\|_{L^{2}_{T}L^{2}\left(\mathcal{Q}_{\{\sigma,\epsilon-\nu t, \tau-\nu t \}}\right)}^{2}&=\sum_{m=1}^{n}|c_{m}|\left\|\Psi_{m}J^{s+1}u\right\|_{L^{2}_{T}L^{2}\left(\mathcal{Q}_{\{\sigma,\epsilon-\nu t, \tau-\nu t \}}\right)}^{2}\\
%	%&\lesssim \sum_{m=1}^{n}\left\|J^{s+1}u\right\|_{L^{2}_{T}L^{2}\left(\mathcal{Q}_{\{\sigma,\epsilon-\nu t, \tau-\nu t \}}\right)}^{2}\\
%	&\lesssim_{n}\sum_{m=1}^{n}\gamma_{m} \left\|\theta_{1}\Psi_{m}J^{s+1}u\right\|_{L^{2}_{T}L^{2}}^{2}.
%	\end{split}
%	\end{equation}
Roughly speaking, up to constant the operator $ \Psi_{m}$ is nothing more than $\partial_{x_{m}}J^{-1}, $ for $m=1,2,\cdots,n.$ Besides,  $\Psi_{m}\in \mathrm{OP}\mathbb{S}^{0},$ so that, in virtue of theorem \ref{zkth1} it maps $L^{2}(\mathbb{R}^{n})$ into $L^{2}(\mathbb{R}^{n}).$

  Notice that notwithstanding  we recover in theory all the derivatives  in the channel of propagation, it is  not so useful for our proposes, due to   the perturbation by the  non-local pseudo-differential operator $\Psi_{m}.$ So that,   in order to reconcile all these issues, we introduce into our analysis the  functions $\theta_{1},\theta_{2},$ by means of the    following decomposition:
 \begin{equation*}
 \begin{split}
 \theta_{2}\Psi_{m}\partial_{x_{m}}J^{s}u&=\left[\Psi_{m}; \theta_{2}\right]\theta_{1}\partial_{x_{m}}J^{s}u+\theta_{2}\Psi_{m}\left(\left(1-\theta_{1}\right)\partial_{x_{m}}J^{s}f\right)+\Psi_{m}\left(\theta_{2}\partial_{x_{m}}J^{s}u\right).\\
 \end{split}
 \end{equation*}
 Nevertheless, in order to incorporate this decomposition  properly into the  argument, it is necessary to define the following functions
 \begin{equation*}
  \vartheta_{1}(x,t):=\theta_{1}\left(x+\frac{\nu t\sigma}{|\sigma|^{2}}\right)\quad \mbox{and}\quad  \vartheta_{2}(x,t):=\theta_{2}\left(x+\frac{\nu t\sigma}{|\sigma|^{2}}\right),\quad x\in \mathbb{R}^{n},t\in \mathbb{R}.
 \end{equation*}
 Therefore, combining lemma \ref{zk19}, the continuity of $\Psi_{m}$ and the properties of the weighted functions $\theta_{1},\theta_{2},$ we get 
 \begin{equation*}
 \begin{split}
&\left\|J^{s+1}u\right\|_{L^{2}\left(\mathcal{Q}_{\{\sigma,\epsilon-\nu t, \tau-\nu t \}}\right)}\\
& \leq\left\|\vartheta_{2}J^{s+1}u\right\|_{L^{2}}\\
&\leq \sum_{m=1}^{n}|c_{m}|\left\|\vartheta_{2}\Psi_{m}\partial_{x_{m}}J^{s}u\right\|_{L^{2}}+\left\|\vartheta_{2}J^{s-1}u\right\|_{L^{2}}\\
 & =\sum_{m=1}^{n}\gamma_{m} \left\{ \left\|\left[\Psi_{m}; \vartheta_{1}\right]\vartheta_{2}\partial_{x_{m}}J^{s}u\right\|_{L^{2}}+\left\|\vartheta_{2}\Psi_{m}\left(\left(1-\vartheta_{1}\right)\partial_{x_{m}}J^{s}f\right)\right\|_{L^{2}} \right.\\
 & \left. \quad + \left\|\Psi_{m}\left(\vartheta_{2}\partial_{x_{m}}J^{s}u\right)\right\|_{L^{2}}\right\} +\left\|\vartheta_{2}J^{s-1}u\right\|_{L^{2}}\\
 &\lesssim_{\epsilon,\tau,n}  \sum_{m=1}^{n}\gamma_{m}\left\{\left\|\vartheta_{2}\partial_{x_{m}}J^{s}u\right\|_{L^{2}} +\|u_{0}\|_{L^{2}}+\left\|\vartheta_{2}\partial_{x_{m}}J^{s}u\right\|_{L^{2}}\right\}+\left\|\vartheta_{2}J^{s-1}u\right\|_{2}\\
 &\lesssim_{\epsilon,\tau,n}\sum_{m=1}^{n}\gamma_{m}\left\{\left\|\partial_{x_{m}}J^{s}u\right\|_{L^{2}\left(\mathcal{Q}_{\left\{\sigma,\frac{3\epsilon}{2}-\nu t, \tau-\epsilon-\nu t \right\}}\right)} +\|u_{0}\|_{L^{2}}\right.\\
 &\quad \left. +\left\|\partial_{x_{m}}J^{s}u\right\|_{L^{2}\left(\mathcal{Q}_{\left\{\sigma,\frac{\epsilon}{2}-\nu t, \tau+\epsilon-\nu t \right\}}\right)}\right\}+ \left\|\vartheta_{2}J^{s-1}u\right\|_{L^{2}}.
 \end{split}
 \end{equation*}
 Hence, in virtue of \eqref{p1} and  \eqref{e22}, it  is clear that 
 \begin{equation}\label{e23}
 \begin{split}
\left\|J^{s+1}u\right\|_{L^{2}_{T}L^{2}\left(\mathcal{Q}_{\{\sigma,\epsilon-\nu t, \tau-\nu t \}}\right)}^{2}
%&\textcolor{blue}{\lesssim_{\epsilon,\tau,n} \sum_{m=1}^{n}\gamma_{m}\left\|\partial_{x_{m}}J^{s}u\right\|_{L^{2}_{T}L^{2}\left(\mathcal{Q}_{\left\{\sigma,\frac{\epsilon}{2}-\nu t, \tau+\epsilon-\nu t \right\}}\right)} ^{2} +T\|u_{0}\|_{L^{2}}^{2}+ Tc^{*}_{(s-1)}}\\
 &\lesssim_{\epsilon,\tau,n} T\left(\|u_{0}\|_{L^{2}}^{2}+c^{*}_{(s-1)}\right)+\left(\frac{c^{*}_{(s)}}{\lambda(\sigma)}\right).
 \end{split}
 \end{equation}
 Gathering the estimates above we get that: for all $\nu\geq 0, \epsilon>0$ and $\tau\geq 5\epsilon,$ the following inequality holds
 \begin{equation*}
 \left\|J^{s+1}u\right\|_{L^{2}_{T}L^{2}\left(\mathcal{Q}_{\{\sigma,\epsilon-\nu t, \tau-\nu t \}}\right)}\lesssim_{\epsilon,\tau,n} c^{**}_{(s)},
 \end{equation*}
whence $c^{**}_{(s)}:=c\left(T\left(\|u_{0}\|_{L^{2}}^{2}+c^{*}_{(s-1)}\right)+\left(\frac{c^{*}_{(s)}}{\lambda(\sigma)}\right)\right)^{1/2}$\, and \,$c=c(\epsilon,\tau,n)$ is a positive constant depending on the parameters indicated.
\end{proof}
\begin{flushleft}
	{\sc Claim 2:}
\end{flushleft}
If for any $\nu\geq 0,\, \epsilon>0$ and $\tau\geq 5\epsilon$
\begin{equation}\label{smoothing1}
\left\|J^{s+1}u\right\|_{L^{2}_{T}L^{2}\left(\mathcal{Q}_{\{\sigma,\epsilon-\nu t, \tau-\nu t \}}\right)}\leq c^{**}_{(s)},
\end{equation}
then for any $\epsilon>0,\, \nu \geq 0 $ and $\tau\geq 5\epsilon,$ there exist a positive  constant such that 
\begin{equation}\label{smoothing2}
 \left\|J^{r}u\right\|_{L^{2}_{T}L^{2}\left(\mathcal{Q}_{\{\sigma,\epsilon-\nu t, \tau-\nu t \}}\right)}\leq c^{***}_{(r)},\quad \mbox{for any}\quad r\in (0,s+1].
 \end{equation}
 
 \begin{proof}[\sc Proof of claim 2:]
 The proof follows by using an argument quite similar to the one used in the proof of lemma \ref{zk37}. Nevertheless, we will indicate some details about the proof.
 
First that all we  shall give us  enough room to handle the operators involved in the smoothing effect \eqref{smoothing1}. For that, notice that the inequality \eqref{smoothing1} holds for any $\epsilon>0$ and $\tau\geq5\epsilon$ in particular if we choose $(\epsilon,\tau)=\left(\frac{2\epsilon}{3},2\tau+\frac{2\epsilon}{3}\right).$

Next, we consider  $\theta_{1},\theta_{2}$  smooth functions satisfying: $0\leq \theta_{1},\theta_{2}\leq 1,$ with   bounded  derivatives, verifying the  conditions indicated below
\begin{equation*}
\supp\left(\theta_{1}\right)\subseteq \mathcal{Q}_{\left\{\sigma,\frac{\epsilon}{3},\tau+\frac{2\epsilon}{3}\right\}}\quad\mbox{with}\quad \theta_{1}\equiv 1 \quad \mbox{on}\quad \overline{\mathcal{Q}_{\left\{\sigma,\frac{2\epsilon}{3},\tau+\frac{\epsilon}{3}\right\}}},
\end{equation*}
and 
\begin{equation*}
\supp\left(\theta_{2}\right)\subseteq \mathcal{Q}_{\left\{\sigma,\frac{2\epsilon}{3},\tau+\frac{\epsilon}{3}\right\}}\quad\mbox{with}\quad \theta_{2}\equiv 1 \quad \mbox{on}\quad \overline{\mathcal{Q}_{\left\{\sigma,\epsilon,\tau\right\}}}.
\end{equation*}
From these conditions it is clear that the following chain of inequalities holds
\begin{equation}\label{ine1}
\begin{split}
%\int_{0}^{T}\int_{\mathcal{Q}_{\left\{\sigma,\epsilon-\nu t,\tau-\nu t\right\}}}\left|J^{s+1}u(x,t)\right|^{2}\mathrm{d}x\,\mathrm{d}t&\leq 
%
\int_{0}^{T}\left\|J^{s+1}u(\cdot,t)\theta_{2}\left(\sigma\cdot +\nu t\right)\right\|^{2}_{L^{2}}\mathrm{d}t
%&\leq\int_{0}^{T}\int_{\mathcal{Q}_{\left\{\sigma,\frac{2\epsilon}{3}-\nu t,\tau+\frac{\epsilon}{3}-\nu t\right\}}}\left|J^{s+1}u(x,t)\right|^{2}\mathrm{d}x\,\mathrm{d}t\\
&\leq \int_{0}^{T}\left\|J^{s+1}u(\cdot,t)\theta_{1}\left(\sigma\cdot +\nu t\right)\right\|^{2}_{L^{2}}\mathrm{d}t\\
&\leq \int_{0}^{T}\int_{\mathcal{Q}_{\left\{\sigma,\frac{\epsilon}{3}-\nu t,\frac{2\epsilon}{3}+2\tau-\nu t\right\}}}\left|J^{s+1}u(x,t)\right|^{2}\mathrm{d}x\,\mathrm{d}t\\
&\leq c^{**}_{(s)}.
\end{split}
\end{equation}
Next,  for $t\in (0,T)$ (fixed) we define the function
\begin{equation*}
F_{t}(z)=\theta_{2}\left(x+ \frac{\nu t\sigma}{|\sigma|^{2}}\right)J^{z}u(x,t),\quad z=\alpha+i\beta\in \mathbb{C}, \,\alpha\in [0,s+1]\quad\mbox{and}\quad \beta\in \mathbb{R}.
\end{equation*}
Notice that from   theorem \ref{zkth1} it is clear that 
\begin{equation*}
F_{t}(i\beta)=\theta_{2}\left(\cdot+ \frac{\nu t\sigma}{|\sigma|^{2}}\right)J^{i\beta}u(\cdot,t)\in L^{2}(\mathbb{R}^{n}).
\end{equation*}
Instead, in the case we evaluate $F_{t}$ at $z=\alpha+i\beta,$  we notice that it  falls on the scope     of lemma \ref{zk19}, so 
 that, 
 \begin{equation*}
 F_{t}(\alpha+i\beta)=\theta_{2}\left(\cdot+ \frac{\nu t\sigma}{|\sigma|^{2}}\right)J^{\alpha+i\beta}u(\cdot,t)=\theta_{2}\left(\cdot+ \frac{\nu t\sigma}{|\sigma|^{2}}\right)J^{i\beta}J^{s}u(\cdot,t)\in L^{2}(\mathbb{R}^{n}).
 \end{equation*}
 Hence,  by the three lines lemma  and  Young's inequality, we get  for  fixed $t,$ the following: 
\begin{equation}\label{sm1}
\left\|\theta_{2}\left(\cdot+ \frac{\nu t\sigma}{|\sigma|^{2}}\right)J^{r}u(\cdot,t)\right\|_{L^{2}}\lesssim \left\|u_{0}\right\|_{L^{2}}+\left\|\theta_{2}\left(\cdot+ \frac{\nu t\sigma}{|\sigma|^{2}}\right)J^{s}u(\cdot,t)\right\|_{L^{2}},
\end{equation}
for $r\in (0,s].$

A careful inspection of the constant involved in the inequalities \eqref{sm1} show that these one's do not depend on the temporal variable. So that, after  integrating both sides in \eqref{sm1} we finally obtain: For any $\epsilon>0,\, \tau\geq 5\epsilon$ and $\nu\geq 0,$
\begin{equation*}
\int_{0}^{T}\int_{\mathcal{Q}_{\left\{\sigma,\epsilon-\nu t,\tau-\nu t\right\}}}\left|J^{r}u(x,t)\right|^{2}\mathrm{d}x\,\mathrm{d}t\lesssim \underbrace{T^{\frac{1}{2}}\left\|u_{0}\right\|_{L^{2}}+c^{**}_{(s)}}_{=: \, c^{***}_{(r)}},\quad r\in (0,s+1].
\end{equation*}
\end{proof}
Since  the technical details necessary has been clarified,   we turn back  our attention to the inductive process.
 \begin{flushleft}
	{\sc Case: ${\displaystyle s\in\left(l,l+1\right),\, l\in \mathbb{N}, \, l> s_{n}+1.}$}
 \end{flushleft}
As usual, our starting point is the identity \eqref{zkwee},  so that, to follow with the inductive argument we estimate the  corresponding terms  coming from such identity. 

Firstly,  we  handle $A_{1}.$ So,  the  main idea is to use the smoothing effect obtained from the former  case i.e,  $s\in(l-1,l],$ that as we have seen it provides one extra local derivative. 

 Nevertheless,  combining the inductive hypothesis, claim 1 and claim 2 we obtain that:
For  $\nu \geq 0, \, \epsilon>0,$ and   $\tau\geq 5\epsilon$ 
% \begin{equation}
% \begin{split}
%\lambda(\sigma)\sum_{m=1}^{n} \int_{0}^{T}\int_{\mathcal{Q}_{\{\sigma,\epsilon-\nu t ,\tau-\nu t\}}}\left(\partial_{x_{m}}J^{r}u(x,t)\right)^{2}\,\mathrm{d}x\,\mathrm{d}t\leq c^{*}, \quad \mbox{for any}\quad  r\in [0,l] 
% \end{split}
% \end{equation}
% or what is more 
 \begin{equation}\label{e17}
\int_{0}^{T}\int_{\mathcal{Q}_{\{\sigma,\epsilon-\nu t ,\tau-\nu t\}}}\left(J^{r}u(x,t)\right)^{2}\,\mathrm{d}x\,\mathrm{d}t\leq c^{***}_{(r)}, \quad \mbox{for any}\quad  r\in (0,l+1]. 
 \end{equation}
So that, with this remark at hand it is enough to combine the properties of the weighted function  and \eqref{e17} to finally obtain 
\begin{equation*}	
\begin{split}
&\int_{0}^{T}\int_{\mathcal{Q}_{\{\sigma,\epsilon-\nu t ,\tau-\nu t\}}}\left(J^{r}u(x,t)\right)^{2}\,\mathrm{d}x\,\mathrm{d}t\\
%&=\int_{0}^{T}\int_{\mathbb{R}^{n}}\mathbb{1}_{\mathcal{Q}_{\{\sigma,\epsilon ,\tau\}}}\left( x+\frac{\nu t\sigma}{|\sigma|^{2}}\right)\left(J^{r}u(x,t)\right)^{2}\,\mathrm{d}x\,\mathrm{d}t\\
%&=\int_{0}^{T}\int_{\mathbb{R}^{n}}\mathbb{1}_{[\epsilon,\tau]}\left(\sigma\cdot x+\nu t\right)\left(J^{r}u(x,t)\right)^{2}\,\mathrm{d}x\,\mathrm{d}t\\
%&\leq \int_{0}^{T}\int_{\mathbb{R}^{n}}\chi'_{\epsilon_{1},\tau_{1}}\left(\sigma\cdot x+\nu t\right)\left(J^{r}u(x,t)\right)^{2}\,\mathrm{d}x\,\mathrm{d}t\\
&\lesssim_{\epsilon,\tau}\int_{0}^{T}\int_{\mathbb{R}^{n}}\left(\chi_{\epsilon/3, \tau+\epsilon}\chi'_{\epsilon/3,\tau+\epsilon}\right)\left(\sigma\cdot x+\nu t \right)\left(J^{r}u(x,t)\right)^{2}\,\mathrm{d}x\,\mathrm{d}t.
\end{split}
\end{equation*}		
%where $\left(\epsilon_{1},\tau_{1}\right)=\left(\frac{\epsilon}{3},\frac{2\epsilon}{3}+\tau\right)$ and $r\in(0,l+1].$

Therefore, for $\epsilon>0$ and $\tau\geq 5\epsilon$
we obtain that 
\begin{equation*}
\begin{split}
\int_{0}^{T}|A_{1}(t)|\,\mathrm{d}t  %&\textcolor{blue}{\leq\int_{0}^{T}\int_{\mathbb{R}^{n}} \left(\chi_{\epsilon, \tau}\chi_{\epsilon,\tau }'\right)\left(\sigma\cdot x+\nu t\right)\left(J^{s}u(x,t)\right)^{2}\,\mathrm{d}x\,\mathrm{d}t}\\
%&\textcolor{blue}{\leq c_{2}\int_{0}^{T}\int_{\mathbb{R}^{n}} \mathbb{1}_{[\epsilon,\tau]}\left(\sigma\cdot x+\nu t\right)\left(J^{s}u(x,t)\right)^{2}\,\mathrm{d}x\,\mathrm{d}t}\\
%&\textcolor{blue}{=c_{2}\int_{0}^{T}\int_{\mathbb{R}^{n}}\mathbb{1}_{\mathcal{Q}_{\{\sigma,\epsilon ,\tau\}}}\left( x+\frac{\nu t\sigma}{|\sigma|^{2}}\right)\left(J^{s}u(x,t)\right)^{2}\,\mathrm{d}x\,\mathrm{d}t}\\
&=c\int_{0}^{T}\int_{\mathcal{Q}_{\{\sigma,\epsilon-\nu t ,\tau-\nu t\}}}\left(J^{s}u(x,t)\right)^{2}\,\mathrm{d}x\,\mathrm{d}t\\
&\lesssim_{\epsilon,\tau,\nu} \left(\frac{c^{***}_{(s-1)}}{\lambda(\sigma)}\right).
\end{split}
\end{equation*}
%being the last inequality a consequence of \eqref{e17} with $r=s-1<l.$

Next, we have that 
\begin{equation*}
\begin{split}
\int_{0}^{T} |A_{4}(t)|\,\mathrm{d}t %&   \textcolor{blue}{\leq  3\sigma_{1}|\sigma|^{2}\int_{0}^{T} \int_{\mathbb{R}^{n}} \left(J^{s}u(x,t)\right)^{2} \left(\chi_{\epsilon,\tau}'\,\chi_{\epsilon,\tau}''\right)\left(\sigma \cdot x+\nu t\right)\,\mathrm{d}x\,\mathrm{d}t}\\
%&  \textcolor{blue}{\leq  3c_{1}\sigma_{1}|\sigma|^{2}\int_{0}^{T} \int_{\mathbb{R}^{n}} \left(J^{s}u(x,t)\right)^{2} \left(\chi_{\epsilon,\tau}'\,\chi_{\epsilon/3,\tau+\epsilon}'\right)\left(\sigma \cdot x+\nu t\right)\,\mathrm{d}x\,\mathrm{d}t}\\
%&\textcolor{blue}{\leq  3cc_{1}\sigma_{1}|\sigma|^{2}\int_{0}^{T} \int_{\mathbb{R}^{n}} \left(J^{s}u(x,t)\right)^{2} \mathbb{1}_{[\epsilon,\tau]}\left(\sigma \cdot x+\nu t\right)\,\mathrm{d}x\,\mathrm{d}t}\\
&= 3cc_{1}\sigma_{1}|\sigma|^{2} \int_{0}^{T}\int_{\mathcal{Q}_{\{\sigma,\epsilon-\nu t ,\tau-\nu t\}}}\left(J^{s}u(x,t)\right)^{2}\,\mathrm{d}x\,\mathrm{d}t\\
&\lesssim_{\epsilon,\tau,\sigma}\left( \frac{c^{***}_{(s-1)}}{\lambda(\sigma)}\right),%\left(\frac{3cc_{1}\sigma_{1}|\sigma|^{2}}{\lambda(\sigma)}\right)c^{*},
\end{split}
\end{equation*}
being the last inequality a consequence of \eqref{e17}   with $r=s-1<l.$

Next, by combining the properties of the weighted functions  we obtain
\begin{equation*}
\begin{split}
\int_{0}^{T}|A_{5}(t)|\,\mathrm{d}t %&  \textcolor{blue}{\leq \sigma_{1}|\sigma|^{2}\int_{0}^{T}\int_{\mathbb{R}^{n}}\left(J^{s}u(x,t)\right)^{2}\left(\chi_{\epsilon,\tau}\,\chi_{\epsilon,\tau}'''\right)\left(\sigma \cdot x+\nu t\right)\,\mathrm{d}x\,\mathrm{d}t}\\
%&\textcolor{blue}{\leq \frac{ \sigma_{1}|\sigma|^{2}}{\tau-3\epsilon}\int_{0}^{T}\int_{\mathbb{R}^{n}}\left(J^{s}u(x,t)\right)^{2}\mathbb{1}_{[\epsilon,\tau]}\left(\sigma \cdot x+\nu t\right)\,\mathrm{d}x\,\mathrm{d}t}\\
&\lesssim_{\epsilon,\tau,\sigma}  \int_{0}^{T}\int_{\mathcal{Q}_{\{\sigma,\epsilon-\nu t ,\tau-\nu t\}}}\left(J^{s}u(x,t)\right)^{2}\,\mathrm{d}x\,\mathrm{d}t\\
&\lesssim_{\epsilon,\tau,\sigma} \left(\frac{c^{***}_{(s-1)}}{\lambda(\sigma)}\right),
\end{split}
\end{equation*}
being the last inequality a consequence of \eqref{e17}   with $r=s-1<l.$

In what concerns to the terms $A_{2,m}$ and $A_{3,m},$  an analysis similar to the used in \eqref{interec}-\eqref{interaction3} imply  the existence of  a positive constant  $\lambda=\lambda(\sigma),$  such that    
\begin{equation}\label{interaction3.1}
\begin{split}
&\lambda(\sigma)\sum_{m=1}^{n} \int_{\mathbb{R}^{n}} \left(\partial_{x_{m}}J^{s}u(x,t)\right)^{2}\left(\chi_{\epsilon,\tau}\,\chi_{\epsilon,\tau}'\right)\left(\sigma\cdot x+\nu t\right)\,\mathrm{d}x\\
&\lesssim 
\sigma_{1}\sum_{m=1}^{n} \int_{\mathbb{R}^{n}} \left(\partial_{x_{m}}J^{s}u(x,t)\right)^{2}\left(\chi_{\epsilon,\tau}\,\chi_{\epsilon,\tau}'\right)(\sigma\cdot x+\nu t)\,\mathrm{d}x\\
&\quad +2\sum_{m=1}^{n}\sigma_{m}\int_{\mathbb{R}^{n}}\left(\partial_{x_{m}}J^{s}u(x,t)\right)\left(\partial_{x_{1}}J^{s}u(x,t)\right)\left(\chi_{\epsilon,\tau}\,\chi_{\epsilon,\tau}'\right)(\sigma\cdot x+\nu t)\,\mathrm{d}x.
\end{split}
\end{equation}
%where $\lambda$ satisfies the inequality  
% \begin{equation*}
%0<\lambda<2\sigma_{1}-|\sigma|.
%\end{equation*}
%Notice that the constant $\lambda$ does not depends on the parameter $s,$ so that it is independent of the inductive variable.

%On the other hand we have that 
%In fact, as was remarked in \eqref{zk22}, a more accurate   equivalence for the terms involving the smoothing effect corresponding to this step can be presented. More precisely,
%\begin{equation*}
%\begin{split}
%&\lambda(\sigma)\sum_{m=1}^{n} \int_{\mathbb{R}^{n}} \left(\partial_{x_{m}}J^{s}u(x,t)\right)^{2}\left(\chi_{\epsilon,\tau,\sigma}\,\chi_{\epsilon,\tau,\sigma}'\right)\left(x,t\right)\,\mathrm{d}x\\
%&\cong \sigma_{1}\sum_{m=1}^{n} \int_{\mathbb{R}^{n}} \left(\partial_{x_{m}}J^{s}u(x,t)\right)^{2}\left(\chi_{\epsilon,\tau,\sigma}\,\chi_{\epsilon,\tau,\sigma}'\right)(x,t)\,\mathrm{d}x\\
%& +2\sum_{m=1}^{n}\sigma_{m}\int_{\mathbb{R}^{n}}\left(\partial_{x_{m}}J^{s}u(x,t)\right)\left(\partial_{x_{1}}J^{s}u(x,t)\right)\left(\chi_{\epsilon,\tau,\sigma}\,\chi_{\epsilon,\tau,\sigma}'\right)(x,t)\,\mathrm{d}x.
%\end{split}
%\end{equation*}
Notice that the terms in the l.h.s above are positive, 
and after integrating in time these will  provide the  smoothing effect. % after apply Gronwall's inequality and integrating  in the  time variable.

% Finally, we focus our attention to the nonlinear part. Nevertheless, in the previous step it was displayed a decomposition  that as  in the previous case  will be useful for handling the term $A_{6}.$
%The nonlinear part  admit the same decomposition as that one in \eqref{nlp1}, so that  for the sake of brevity we  will not rewrite it here again.

%First, we will start by dealing with the more subtle  terms. 
%As in the previous case these correspond to the terms $B_{6,1}$ and $B_{6,2},$ where we shall remark,  the  main idea was to decompose the commutator terms up to certain point and then combine it with Leibniz rule.

Next, we handle the non-linear part. In this sense,   we will denote   the commutator term    ${\displaystyle 
	\Psi_{\zeta_{\epsilon,\tau,s,\sigma}}=\left[J^{s}; \chi_{\epsilon,\tau,\sigma}\right]\in \mathrm{OP}\mathbb{S}^{s-1}.}$

 According to Proposition \ref{prop1},  the symbol $\zeta_{\epsilon,\tau,s,\sigma}$  admits the following decomposition: 
 \begin{equation*}
 \begin{split}
 \zeta_{\epsilon,\tau,s,\sigma}(x,\xi)\sim \sum_{\alpha} \frac{(2\pi i)^{-|\alpha|}}{\alpha!}\left(\partial^{\alpha}_{\xi}\left(\langle \xi\rangle^{s}\right)\partial_{x}^{\alpha}\left(\chi_{\epsilon,\tau,\sigma}(x,t)\right)\right)\quad \mbox{for any}\quad t\in \mathbb{R}.
 \end{split}
 \end{equation*}
 More precisely,  
 \begin{equation*}\label{}
 \begin{split}
 &\zeta_{\epsilon,\tau,s,\sigma}(x,\xi)\\
 &=\sum_{1\leq|\alpha|\leq l}\frac{(2\pi i)^{-|\alpha|}}{\alpha!}\left\{\partial^{\alpha}_{\xi}\left(\langle \xi\rangle^{s}\right)\partial_{x}^{\alpha}\left(\chi_{\epsilon,\tau,\sigma}(x,t)\right)\right\}+\kappa_{s-l-1}(x,\xi)\\
 &=\sum_{j=1}^{l}\sum_{|\alpha|= j}\frac{(2\pi i)^{-|\alpha|}}{\alpha!}\left\{\partial^{\alpha}_{\xi}\left(\langle \xi\rangle^{s}\right)\partial_{x}^{\alpha}\left(\chi_{\epsilon,\tau,\sigma}(x,t)\right)\right\}+\kappa_{s-l-1}(x,\xi)\\
 &=c_{1}(s)\left(\sum_{|\alpha|=1}(2\pi i\xi)^{\alpha}\sigma^{\alpha}\langle\xi\rangle^{s-2}\chi_{\epsilon,\tau}^{(1)}\left(\sigma\cdot x+\nu t\right)+\sum_{ \mathclap{\substack{|\alpha|=2\\\alpha=\alpha_{1}+\alpha_{2}\\|\alpha_{1}|=|\alpha_{2}|=1}}}\delta_{\alpha_{1},\alpha_{2}}\frac{\sigma^{\alpha}}{\alpha!}\chi_{\epsilon,\tau}^{(2)}\left(\sigma\cdot x+\nu t\right)\langle \xi\rangle^{s-2}\right)\\
% &\quad +\frac{s(s-2)}{(2\pi)^{4}}\sum_{|\alpha|=2}\frac{\sigma^{\alpha}}{\alpha!}\chi_{\epsilon,\tau}''\left(\sigma\cdot x+\nu t\right)\left(2\pi i\xi\right)^{\alpha}\langle \xi \rangle^{s-4}+
% \frac{s(s-2)}{(2\pi)^{4}}\sum_{|\alpha|=2}\frac{\sigma^{\alpha}}{\alpha!}\chi_{\epsilon,\tau}''\left(\sigma\cdot x+\nu t\right)\left(2\pi i\xi\right)^{\alpha}\langle \xi \rangle^{s-4}\\
 &\, +c_{2}(s)\sum_{\mathclap{\substack{|\alpha|=3\\\alpha=\alpha_{1}+\alpha_{2}+\alpha_{3}\\|\alpha_{1}|=\cdots=|\alpha_{3}|=1}}}\frac{\sigma^{\alpha}}{\alpha!}\left(\delta_{\alpha_{1},\alpha_{2}}(2\pi i \xi )^{\alpha_{3}}+\delta_{\alpha_{3},\alpha_{1}}(2\pi i\xi)^{\alpha_{2}}+\delta_{\alpha_{3},\alpha_{2}}(2\pi i\xi)^{\alpha_{1}}\right)\langle \xi \rangle^{s-4}\chi_{\epsilon,\tau}^{(3)}\left(\sigma\cdot x+\nu t\right)\\
 &\, +c_{3}(s)\sum_{\mathclap{\substack{|\alpha|=3\\\alpha=\alpha_{1}+\alpha_{2}+\alpha_{3}\\|\alpha_{1}|=\cdots=|\alpha_{3}|=1}}}\frac{\sigma^{\alpha}}{\alpha!}\left((2\pi i\xi )^{\alpha_{3}}\langle \xi \rangle ^{s-6}\right)\chi_{\epsilon,\tau}^{(3)}\left(\sigma\cdot x+\nu t\right)\\
 &\,
  + \cdots +\sum_{|\alpha|= l}\frac{(2\pi i)^{-|\alpha|}\sigma^{\alpha}}{\alpha!}\left\{\partial^{\alpha}_{\xi}\left(\langle \xi\rangle^{s}\right)\chi_{\epsilon,\tau}^{(|\alpha|)}\left(\sigma\cdot x+\nu t\right)\right\}+\kappa_{s-l-1}(x,\xi),
 \end{split}
 \end{equation*}
 where $\kappa_{s-l-1} \in \mathbb{S}^{l+1-s}\subset \mathbb{S}^{0}.$  Precisely, we associate to the symbol $\kappa_{s-l-1}$   the operator 
 \begin{equation*}
\Psi_{\kappa_{s-l-1}}g(x):= \int_{\mathbb{R}^{n}}e^{2\pi i x\cdot \xi} \kappa_{s-l-1}(x,\xi)\, \widehat{g}(\xi)\,\mathrm{d}\xi,\qquad g\in \mathcal{S}(\mathbb{R}^{n}).
 \end{equation*}
 In view that $\kappa_{s-l-1}\in \mathbb{S}^{0}$ and Theorem \ref{zkth1} it is clear that  $\Psi_{\kappa_{s-l-1}}$ maps  $L^{2}(\mathbb{R}^{n})$ into itself i.e,
 \begin{equation*}
 \left\|\Psi_{\kappa_{s-l-1}}f\right\|_{L^{2}}\lesssim \|f\|_{L^{2}} 
 \end{equation*}
 for $f$ in an appropriated class of functions.
 
 Also, for multi-index $\alpha,\beta $ with $\beta\leq \alpha$ we define  
 \begin{equation*}
 \eta_{\alpha,\beta}(x,\xi):=\frac{(2\pi i\xi)^{\beta}}{\left(1+|\xi|^{2}\right)^{\frac{|\alpha|}{2}}},\qquad x,\xi\in\mathbb{R}^{n},
 \end{equation*}
 whence  $ \eta_{\alpha,\beta}\in \mathbb{S}^{|\beta|-|\alpha|}\subset\mathbb{S}^{0},$ and we associated to it the operator 
 \begin{equation}\label{e11.1}
 \Psi_{\eta_{\alpha,\beta}}g(x):=\int_{\mathbb{R}^{n}}e^{2\pi ix\cdot \xi}\eta_{\alpha,\beta}(x,\xi)\widehat{g}(\xi)\,\mathrm{d}\xi,\quad g\in \mathcal{S}(\mathbb{R}^{n}),
 \end{equation}
 that according to Theorem \ref{zkth1} it satisfies $${\displaystyle \left\|\Psi_{\eta_{\alpha,\beta}}g\right\|_{L^{2}} \lesssim \|g\|_{L^{2}}.}$$
 
 Now, by rearranging the terms above in the decomposition of the symbol $\zeta_{\epsilon,\tau,s,\sigma}$,    we obtain %a more simplified version,   that  in terms of operators   allows us   to incorporate  $ \Psi_{\eta_{\alpha,\beta}}$ and $\Psi_{\kappa_{s-l-1}}$ into a more compact formula, in detail  
\begin{equation*}
\begin{split}
\Psi_{\zeta_{\epsilon,\tau,s,\sigma}}f(x)=\sum_{j=1}^{l}\sum_{|\alpha|=j}\sum_{\beta\leq\alpha} \omega_{\alpha,\beta,\sigma,s} \,\partial_{x}^{\alpha}\chi_{\epsilon, \tau,\sigma}(x,t)\Psi_{\eta_{\alpha,\beta}}J^{s-|\alpha|}f(x)+ \Psi_{\kappa_{s-m-1}}f(x), 
\end{split}
\end{equation*}
where  $f\in \mathcal{S}(\mathbb{R}^{n})$ and   $\omega_{\alpha,\beta,\sigma,s}$ denotes a constant depending on the parameters indicated.

Now, we turn back our attention to the terms involving this commutator term. So that, combining    interpolation and  inequality \eqref{zk8} we get  
\begin{equation}\label{e12.1}
\begin{split}
&\left\|B_{6,1}\right\|_{L^{2}}\\
&\leq \sum_{j=1}^{l}\sum_{|\alpha|=j}\sum_{\beta\leq\alpha}|\omega_{\alpha,\beta,\sigma,s}| \,\left\|\chi_{\epsilon, \tau}^{(|\alpha|)}(\sigma \cdot(\cdot)+\nu t) \right\|_{L^{\infty}_{x}}\left\|\Psi_{\eta_{\alpha,\beta}}J^{s-|\alpha|}\left(\partial_{x_{1}}\left(\left(u\chi_{\epsilon,\tau,\sigma}\right)^{2}\right)\right)\right\|_{L^{2}}\\
&\quad +\left\| \Psi_{\kappa_{s-l-1}}\left(\partial_{x_{1}}\left(\left(u\chi_{\epsilon,\tau,\sigma}\right)^{2}\right)\right)\right\|_{L^{2}}\\
%&\textcolor{blue}{\lesssim_{\sigma,\epsilon,\tau}\sum_{j=1}^{l}\sum_{|\alpha|=j}\sum_{\beta\leq\alpha}\widetilde{\omega_{\alpha,\beta,\sigma,s}}\left\|J^{s-|\alpha|}\left(\partial_{x_{1}}\left(\left(u\chi_{\epsilon,\tau,\sigma}\right)^{2}\right)\right)\right\|_{L^{2}}+\left\| \partial_{x_{1}}\left(\left(u\chi_{\epsilon,\tau,\sigma}\right)^{2}\right)\right\|_{L^{2}}}\\
%&\textcolor{blue}{= \gamma_{\sigma,\epsilon,\tau}\sum_{j=1}^{l}\sum_{|\alpha|=j}\sum_{\beta\leq\alpha} \widetilde{\omega_{\alpha,\beta,\sigma,s}} \,\left\|J^{s-1}\left(\partial_{x_{1}}\left(\left(u\chi_{\epsilon,\tau,\sigma}\right)^{2}\right)\right)\right\|_{L^{2}}+\left\| J^{1}\left(\left(u\chi_{\epsilon,\tau,\sigma}\right)^{2}\right)\right\|_{L^{2}}}\\
%&\textcolor{blue}{=\left(1+\gamma_{\sigma,\epsilon,\tau}\sum_{j=1}^{l}\sum_{|\alpha|=j}\sum_{\beta\leq\alpha} \widetilde{\widetilde{\omega_{\alpha,\beta,\sigma,s}}} \right)\left\|J^{s}\left(\left(u\chi_{\epsilon,\tau,\sigma}\right)^{2}\right)\right\|_{L^{2}}}\\
&\lesssim_{\epsilon,\tau,\sigma,l,n}\left\|u\right\|_{L^{\infty}}\left\|J^{s}\left(u\chi_{\epsilon,\tau,\sigma}\right)\right\|_{L^{2}}.
\end{split}
\end{equation}
Analogously,
\begin{equation}\label{e13.1}
\begin{split}
&\left\|B_{6,2}\right\|_{L^{2}}\\
&\leq \sum_{j=1}^{l}\sum_{|\alpha|=j}\sum_{\beta\leq\alpha} |\omega_{\alpha,\beta,\sigma,s}| \,\left\|\partial_{x}^{\alpha}\chi_{\epsilon, \tau,\sigma}(\cdot,t)\right\|_{L^{\infty}_{x}}\left\|\Psi_{\eta_{\alpha,\beta}}J^{s-|\alpha|}\left(\partial_{x_{1}}\left(\left(u\widetilde{\phi_{\epsilon,\tau}}\right)^{2}\right)\right)\right\|_{L^{2}}\\
&\quad +\left\| \Psi_{\kappa_{s-l-1}}\left(\partial_{x_{1}}\left(\left(u\widetilde{\phi_{\epsilon,\tau}}\right)^{2}\right)\right)\right\|_{L^{2}}\\
%&\textcolor{blue}{\lesssim_{\sigma,\epsilon,\tau}\sum_{j=1}^{l}\sum_{|\alpha|=j}\sum_{\beta\leq\alpha} \widetilde{\omega_{\alpha,\beta,\sigma,s}} \,\left\|J^{s-|\alpha|}\left(\partial_{x_{1}}\left(\left(u\widetilde{\phi_{\epsilon,\tau}}\right)^{2}\right)\right)\right\|_{L^{2}}+\left\| \partial_{x_{1}}\left(\left(u\widetilde{\phi_{\epsilon,\tau}}\right)^{2}\right)\right\|_{L^{2}}}\\
%&\textcolor{blue}{= \gamma_{\sigma,\epsilon,\tau,\sigma}\sum_{j=1}^{l}\sum_{|\alpha|=j}\sum_{\beta\leq\alpha}\widetilde{\omega_{\alpha,\beta,\sigma,s}} \,\left\|J^{s-1}\left(\partial_{x_{1}}\left(\left(u\widetilde{\phi_{\epsilon,\tau}}\right)^{2}\right)\right)\right\|_{L^{2}}+\left\| J^{1}\left(\left(u\widetilde{\phi_{\epsilon,\tau}}\right)^{2}\right)\right\|_{L^{2}}}\\
%&\textcolor{blue}{\lesssim\left(1+\gamma_{\sigma,\epsilon,\tau}\sum_{j=1}^{l}\sum_{|\alpha|=j}\sum_{\beta\leq\alpha} \widetilde{\widetilde{\omega_{\alpha,\beta,\sigma,s}}}\right)\left\|J^{s}\left(\left(u\widetilde{\phi_{\epsilon,\tau,\sigma}}\right)^{2}\right)\right\|_{L^{2}}}\\
&\lesssim_{\epsilon,\tau,\sigma,l,n,s}\left\|u\right\|_{L^{\infty}}\left\|J^{s}\left(u\widetilde{\phi_{\epsilon,\tau,\sigma}}\right)\right\|_{L^{2}}.
\end{split}
\end{equation}
Since we have finished estimating the terms that require more effort, we focus our  attention on the remaining terms.

In the first place,
we note that the same argument used in the previous case for the terms $B_{6,3}$ and $B_{6,6}$ produce
\begin{equation*}
\|B_{6,3}\|_{L^{2}}\lesssim_{\epsilon,\sigma} \|u_{0}\|_{L^{2}}\|u\|_{L^{\infty}} \quad \mbox{and}\quad \left\|B_{6,6}\right\|_{L^{2}}
\lesssim \|u\|_{L^{\infty}}\|u_{0}\|_{L^{2}}.
\end{equation*}
 In the second place, the inequality \eqref{zk2} imply that 
 \begin{equation}\label{e14.1}
 \left\|\left[J^{s}; u\chi_{\epsilon, \tau,\sigma}\right]\partial_{x_{1}}\left(u\chi_{\epsilon,\tau,\sigma}\right)\right\|_{L^{2}}\lesssim \left\|J^{s}\left(u\chi_{\epsilon, \tau,\sigma}\right)\right\|_{L^{2}}\left\|\partial \left(u \chi_{\epsilon, \tau,\sigma }\right)\right\|_{L^{\infty}},
 \end{equation}
 and 
 \begin{equation}\label{e15.1}
 \begin{split}
 \left\|\left[J^{s}; u\chi_{\epsilon, \tau,\sigma }\right]\partial_{x_{1}}\left(u\phi_{\epsilon, \tau,\sigma }\right)\right\|_{L^{2}}
 &\lesssim \left\|J^{s}\left(u\phi_{\epsilon,\tau,\sigma }\right)\right\|_{L^{2}}\left\|\nabla \left(u\chi_{\epsilon, \tau,\sigma }\right)\right\|_{L^{\infty}}\\
 &\quad  +\left\|J^{s}\left(u\chi_{\epsilon, \tau,\sigma }\right)\right\|_{L^{2}}\left\|\partial_{x_{1}}\left(u\phi_{\epsilon, \tau,\sigma }\right)\right\|_{L^{\infty}}.
 \end{split}
 \end{equation}
 As in the previous case, our analysis requires to  estimate several terms in \eqref{e12.1}- \eqref{e15.1} for which we have not provided  with upper bounds.  To  finish our argument
 we  estimate %will roughly describe how to bound % them. For that, we take hand of the decomposition
\begin{equation*}
\begin{split}
J^{s}\left(u\chi_{\epsilon, \tau,\sigma }\right)
&=\chi_{\epsilon, \tau,\sigma }J^{s}u+\left[J^{s}; \chi_{\epsilon, \tau,\sigma }\right]\left(u\chi_{\epsilon, \tau,\sigma }+u\phi_{\epsilon, \tau,\sigma }+u\psi_{\epsilon,\sigma}\right)\\
&=I+II+III+IV.
\end{split}
\end{equation*}
Notice that $I$  is the quantity to be estimated. Instead, the  remainder terms are of lower order and after integrating in the time variable these can be bounded by combining \eqref{e17} and  lemma \ref{lemm}.

Similarly,  the terms $\left\|J^{s}\left(u\phi_{\epsilon, \tau,\sigma}\right)\right\|_{L^{2}},\, \left\|J^{s}\left(u\widetilde{\phi_{\epsilon, \tau,\sigma }}\right)\right\|_{L^{2}}$ can be bounded by combining lemma \ref{lemm} and   \eqref{e17}.

Finally, we gather the information relative to this step,   followed by an application of Gronwall's inequality and integration in time whence we get   that  for any $ \nu\geq 0, \epsilon>0, \, \tau \geq 5 \epsilon,$  
\begin{equation}\label{e21}
\sup_{0\leq t\leq T}\left\|J^{s}u\right\|_{L^{2}\left(\mathcal{H}_{\{\sigma,\epsilon-\nu t\}}\right)}^{2}+\sum_{m=1}^{n}\lambda(\sigma)\left\|\partial_{x_{m}}J^{s}u\right\|_{L^{2}_{T}L^{2}\left(\mathcal{Q}_{\{\sigma,\epsilon-\nu t, \tau-\nu t \}}\right)}^{2}\leq c^{****}_{(s)},
\end{equation}
where $c^{****}_{(s)}$ is a positive constant depending on the following quantities  \\${\displaystyle c^{****}_{(s)}=c^{****}_{(s)}\left(\epsilon;\tau;\sigma;\lambda;\nu;n;s;T;\left\|J^{s}u_{0}\right\|_{L^{2}\left(\mathcal{H}_{\{\sigma,\epsilon\}}\right)};\left\|u\right\|_{L^{\infty}_{T}H^{s_{n}+}}\right)>0.}$

	This last inequality finish the inductive argument. 
	
%	To obtain explicitly the results indicated in \ref{zk9}  we require several steps already described in the proof. Nevertheless, for the sake of brevity we will omit it here  but
%	without first indicating the main steps that lead to its obtaining. 
	
	Notice that 
from \eqref{e21} it can be deduced  \eqref{g1}  after combining   lemma \ref{zk37} and the properties of the weighted functions. Instead, to obtain \eqref{g2.1}  it can be  used an argument  quite similar  to the one  used in the proof of claim 1. 

Finally  we gather the estimates obtained to conclude  the proof of \eqref{g1} and \eqref{g2.1}.
\end{proof}
%In the next subsection we present several   consequences of  the theorem \ref{zk9}.
Next, we focus our attention in to understand the behavior of the $J^{s}u$ when we restrict  to the remainder part of the space, instead of the half-space where the propagation occurs. For that, we firts present   previous results that will help  us to provide the bounds required in the proof of Corollary \ref{cor11}.
\begin{cor}\label{cor1}
 Let  $\sigma=(\sigma_{1},\sigma_{2},\dots,\sigma_{n})\in \mathbb{R}^{n}$  with
$
\sigma_{1}>0,\,\, \sigma_{n},\dots,\sigma_{n}\geq  0.$	Let $f:\mathcal{H}_{\{\sigma,0\}}\longrightarrow [0,\infty)$ be a continuous   function, such that  for every $\alpha>0,$
	\begin{equation}\label{hyp1}
	\int_{\mathcal{Q}_{\{\sigma,0,\alpha\}}}f(x)\,\mathrm{d}x\leq c\alpha^{q},
	\end{equation} 
	for some $q>0$ and some positive constant $c.$
	
	Then, for  every $\delta>0,$
	\begin{equation}\label{HYP2}
	\int_{\mathcal{H}_{\{\sigma,0\}}}\frac{f(x)}{\langle \sigma\cdot x\rangle^{q+\delta}}\,\mathrm{d}x\leq c(\delta,\sigma,q).
	\end{equation}
\end{cor}
 \begin{proof}
Firstly, we  make  the following  change of variable   
\begin{equation*}
\int_{\mathcal{H}_{\{\sigma,0\}}}\frac{f(x)}{\langle \sigma\cdot x\rangle^{q+\delta}}\,\mathrm{d}x=\frac{1}{\sigma_{1}}\int_{\mathcal{H}_{\{\mathrm{e}_{1},0\}}}\frac{f\left(Ay\right)}{\left\langle   y_{1}\right\rangle^{q+\delta}}\,\mathrm{d}y
\end{equation*}	
where  $A\in \mathcal{M}^{n\times n}(\mathbb{R}).$ % have been chosen   satisfying   $ \sigma A=\mathrm{e}_{1}.$ 
More precisely,
\begin{equation*}
A=\begin{pmatrix}
\frac{1}{\sigma_{1}}&-\frac{\sigma_{2}}{\sigma_{1}}&\cdots&-\frac{\sigma_{n}}{\sigma_{1}}\\
0 &1&\cdots&0\\
\vdots&\vdots&\ddots&\vdots\\
0&0&\cdots& 1
\end{pmatrix}
.
\end{equation*}
Next, notice that 
\begin{equation}\label{decomp2}
\mathcal{H}_{\{\mathrm{e}_{1},0\}}\subset\mathcal{Q}_{\{\mathrm{e}_{1},-1,1\}}\cup\bigcup_{k\in \mathbb{N}}\mathcal{Q}_{\left\{\mathrm{e}_{1},2^{k-1},2^{k+1}\right\}}.
\end{equation}
So that, in virtue of the decomposition  indicated above, we consider $\{\psi_{k}\}_{k\geq 0}$ be a smooth partition of unity  of $\mathbb{R}^{+}$ such that 
\begin{equation*}
\supp\left(\psi_{k}\right)\subseteq [2^{k-1},2^{k+1}]\qquad\mbox{for}\quad k=1,2,\cdots,
\end{equation*} 
and  $\supp\left(\psi_{0}\right)\subseteq [-1,1].$ Under the  conditions specified above  we have that 
\begin{equation}\label{decomp1}
1=\psi_{0}(x)+\sum_{k\in\mathbb{N}} \psi\left(2^{-k}x\right)\quad \mbox{for all}\quad x\in\mathbb{R}, x\geq 0.
\end{equation}
Then, combining the hypothesis \eqref{hyp1} and  \eqref{decomp1}-\eqref{decomp2}  we obtain 
\begin{equation*}
\begin{split}
\int_{\mathcal{H}_{\{\sigma,0\}}} \frac{f(x)}{\langle \sigma\cdot x\rangle^{q+\delta}}\,\mathrm{d}x&=\frac{1}{\sigma_{1}}\int_{\mathcal{H}_{\{\mathrm{e}_{1},0\}}} \frac{f(Ax)}{\langle x_{1}\rangle^{q+\delta}}\,\mathrm{d}x\\
&\leq\frac{1}{\sigma_{1}}\int_{\mathcal{Q}_{\left\{\mathrm{e}_{1},0,1\right\}}} \frac{\psi_{0}(x_{1})}{\langle x_{1}\rangle^{q
		+\delta}}f(Ax)\,\mathrm{d}x\\
&\quad +\frac{1}{\sigma_{1}}\sum_{k=1}^{\infty}\int_{\mathcal{Q}_{\left\{\mathrm{e}_{1},2^{k-1},2^{k+1}\right\}}} \frac{\psi_{0}(x_{1})}{\langle x_{1}\rangle^{q+\delta}}f(Ax)\,\mathrm{d}x\\
%&\leq|\det(A)|\int_{\mathcal{Q}_{\left\{\mathrm{e}_{1},0,1\right\}}} \frac{f(Ax)}{\langle a_{11}x_{1}\rangle^{p+\delta}},\mathrm{d}x\\
%&\quad +|\det(A)|\sum_{k=1}^{\infty}\int_{\mathcal{Q}_{\left\{\mathrm{e}_{1},0,2^{k+1}\right\}}} \frac{\psi_{0}(x_{1})}{\langle a_{11}x_{1}\rangle^{p+\delta}}f(Ax)\,\mathrm{d}x\\
%& \leq |\det(A)|\left(1+\frac{1}{|a_{11}|^{p+\delta}}\sum_{k=1}^{\infty}\frac{2^{2p+\delta}}{2^{k\delta}}\right)\\
&\leq \underbrace{\frac{c}{\sigma_{1}}\left(1+\left(\frac{2^{2q+\delta}}{2^{\delta}-1}\right)\right)}_{=: \, c(\delta,\sigma,q)},
\end{split}
\end{equation*}
for all $\delta>0.$

Therefore,
\begin{equation*}
\int_{\mathcal{H}_{\{\sigma,0\}}} \frac{f(x)}{\langle \sigma \cdot x\rangle^{q+\delta}}\,\mathrm{d}x\lesssim c(\delta,\sigma,q),\quad \mbox{for all}\quad \delta>0.
\end{equation*}
 \end{proof}
\begin{rem}\label{rem1.1}
At this point 	several issues have  to be emphasized and clarified.
	\begin{itemize}
		\item[(i)]  We shall remark that  the corollary also applies  when integrating a non-negative  function on the set $\mathcal{Q}_{\{\sigma,-(\alpha+\epsilon),-\epsilon\}}$ which implies decay on the  complement of the half-space $\mathcal{H}_{\{\sigma,0\}}.$ 
		\item[(ii)] 	It is also important to emphasize   that the constant $c$  appearing  in \eqref{hyp1} also appears implicitly in \eqref{HYP2},  as it was evidenced in the proof of  corollary \ref{cor1}. 
	\end{itemize}
\end{rem}

\begin{proof}[Proof of Corollary \ref{cor11}]
	Without loss of generality we will assume that $\beta=0$ in Theorem \ref{zk9}. So that, it is clear that 
for $t\in(0,T)$ (fixed),	 and for all $\epsilon>0,$
	\begin{equation*}
	\begin{split}
	\int_{\mathcal{H}_{\{\sigma,\epsilon-\nu t\}}}\left(J^{s}u(x,t)\right)^{2}\,\mathrm{d}x&=\int_{\mathcal{Q}_{\{\sigma,\epsilon-\nu t,\epsilon\}}}\left(J^{s}u(x,t)\right)^{2}\,\mathrm{d}x+\int_{\mathcal{H}_{\{\sigma,\epsilon\}}}\left(J^{s}u(x,t)\right)^{2}\,\mathrm{d}x\\
	&\lesssim c^{*}.
	\end{split}
	\end{equation*}
	Notice that the second term in the r.h.s above   is bounded, to see this it is enough  to take $\nu=0$ in Theorem \ref{zk9}. So that, it only remains to estimate the missing term above.
	
	Since,
	\begin{equation*}
	\begin{split}
	\int_{\mathcal{Q}_{\{\sigma,\epsilon-\nu t,\epsilon\}}}\left(J^{s}u(x,t)\right)^{2}\,\mathrm{d}x&=\int_{\mathcal{Q}_{\{\sigma,-(\epsilon+\nu t) , -\epsilon\}}}\left(J^{s}u\left(x+\left(\frac{2\epsilon}{|\sigma|^{2}}\right)\sigma,t\right)\right)^{2}\,\mathrm{d}x\\
	&\lesssim c^{*}tt^{-1} \quad \mbox{for}\quad t\in(0,T) 
	\end{split}
	\end{equation*}
	and $\nu>0.$
	
	So that,  combining  corollary  \ref{cor1} and the remark \ref{rem1.1} with $\alpha=t$ and $q=s,$  we obtain 
	\begin{equation*}
	\begin{split}
&\int_{\mathcal{H}_{\{\sigma,-\epsilon\}}^{c}}\frac{1}{\left\langle \sigma\cdot x+2\epsilon\right\rangle^{s+\delta}}\left(J^{s}u\left(x+\left(\frac{2\epsilon}{|\sigma|^{2}}\right)\sigma,t\right)\right)^{2}\,\mathrm{d}x\\
&=\int_{\mathcal{H}_{\{\sigma,\epsilon\}}^{c}}\frac{1}{\left\langle \sigma\cdot x\right\rangle^{s+\delta}}\left(J^{s}u\left(x,t\right)\right)^{2}\,\mathrm{d}x\\
&\lesssim_{\delta,s,\sigma} \frac{1}{t}\quad\mbox{for}\quad \delta>0.
\end{split}
	\end{equation*}
	Finally, we gather the estimates above to obtain 
	\begin{equation*}
	\int_{\mathbb{R}^{n}} \frac{1}{\left\langle \left(\sigma\cdot x\right)_{-}\right\rangle^{s+\delta}}\left(J^{s}u\left(x,t\right)\right)^{2}\,\mathrm{d}x\lesssim_{\delta,s,\sigma}\frac{1}{t}\quad \mbox{for all} \quad t\in(0,T).
	\end{equation*}
\end{proof}

\section{Acknowledgment}
I would   would like to    thanks to Prof. Felipe Linares  for call my attention on this problem as well as   its valuable  comments that help to improve a previous version  of this work. % Also I would like to thanks to IMPA in  Rio de Janeiro, Brazil for providing the adequate  environment   where part of this research was undertaken.% Last but not least, I would like  to thanks to the CMM from Universidad de Chile for providing the facilities that enhance the culmination of this work. 

%\section{Addendum} 
%{\color{blue}{
% \begin{thm}[Faa di Bruno]
% 	Let $\Omega\subseteq \mathbb{R}^{n}$  and $U\subseteq \mathbb{R}^{m}$ be open sets, let  $u:\Omega\longrightarrow \mathbb{R}$ be of class $C^{m}$ and let  $v: U:\longrightarrow \Omega$ be of class $C^{m}.$  Then $u\circ v$ belongs to  $C^{m}(U)$ and for every  multi-index $\alpha\in \left(\mathbb{N}_{0}\right)^{M},$  with  $1\leq |\alpha|\leq m,$
% 	\begin{equation}
% 	\partial_{y}^{\alpha}\left(u\circ v\right)(y)=\sum c_{\alpha,\beta,\gamma,l} \,\left(\partial_{x}^{\beta}u\right)(v(y))\,\prod_{j=1}^{|\beta|}\partial_{y}^{\gamma_{j}}v_{l_{j}}(y),
% 	\end{equation}
% 	where  $c_{\alpha,\beta,\gamma,l}\in \mathbb{R},$ the  sum is done overall $\beta\in \mathbb{N}_{0}^{N}$ with  $1\leq |\beta|\leq |\alpha|,$  $\gamma=(\gamma_{1},\gamma_{2},\cdots,\gamma_{|\beta|}),$  $\gamma_{j}\in \left(\mathbb{N}_{0}\right)^{M}$  with $|\gamma_{j}|>0,$  and  ${\displaystyle \sum_{j=1}^{|\beta|}\gamma_{j}=\alpha}$ and  $l=(l_{1},l_{2},\cdots,l_{|\beta|}),$  $l_{j}\in\{1,2,\cdots,N\},$  $j=1,2\cdots,|\beta|.$
% 			\end{thm}
% }}

\end{document}